\documentclass[reqno]{amsart}

\usepackage[T1]{fontenc}
\usepackage{amsmath,amssymb,amsxtra,amscd}
\usepackage{amsthm}
\usepackage[all]{xy}
\usepackage{FusionNotation}

\newtheorem{thm}{Theorem}[section]
\newtheorem{prop}[thm]{Proposition}
\newtheorem{lem}[thm]{Lemma}
\newtheorem{cor}[thm]{Corollary}

\newtheorem{mthm}{Theorem}

\theoremstyle{definition}
\newtheorem{defn}[thm]{Definition}
\newtheorem{eqdefn}[thm]{Equivalent definition}
\newtheorem{example}[thm]{Example}
\newtheorem{rmk}[thm]{Remark}

\newcommand{\Afree}{\ensuremath{A_{\operatorname{fr}}}}
\newcommand{\op}[1]{\ensuremath{{#1}^{\operatorname{op}}}}
\newcommand{\opmap}{\ensuremath{\operatorname{op}}}

\newcommand{\PreFus}{\mathcal{P}}
\newcommand{\PreOrb}{\curs{Pre\text{-}Orb}}
\newcommand{\PreFix}{\curs{Pre\text{-}Fix}}

\author{K\'ari Ragnarsson}
\address{Department of Mathematics\\
         Depaul University\\ 
         2320 N.~Kenmore Avenue\\ 
         Chicago, IL 60614\\
         USA}
\author{Radu Stancu}
\address{CNRS UMR 6140 - LAMFA,
Universite de Picardie,
33, Rue Saint-Leu,
80039 Amiens CX 1,
France}
\title{Saturated fusion systems as idempotents in the double Burnside ring}
\begin{document}

\date{October 30, 2009}

\begin{abstract}
We give a new, unexpected characterization of saturated fusion systems on a $p$-group $S$ in terms of idempotents in the $p$-local double Burnside ring of $S$ that satisfy a Frobenius reciprocity relation, and reformulate fusion-theoretic phenomena in the language of idempotents. Interpreting our results in stable homotopy, we answer a long-standing question on stable splittings of classifying spaces of finite groups, and generalize the Adams--Wilkerson criterion for recognizing rings of invariants in the cohomology of an elementary abelian $p$-group.  This work is partly motivated by a conjecture of Haynes Miller which proposes retractive transfer triples as a purely homotopy-theoretic model for $p$-local finite groups.  We take an important step toward proving this conjecture by showing that a retractive transfer triple gives rise to a $p$-local finite group when two technical assumptions are made, thus reducing the conjecture to proving those two assumptions.
\end{abstract}


\maketitle

\section{Introduction}

Fusion systems are an abstract model for the $p$-local structure of a finite group. To a finite group $G$ with Sylow $p$-subgroup $S$ one associates the category $\Ff_S(G)$ whose objects are the subgroups of $S$, and whose morphisms are the group homomorphisms induced by conjugation in $G$ and inclusion. Alperin--Brou\'e showed in \cite{AB} that a similar structure arises when one considers the $G$-conjugation among Brauer pairs in a block of defect $S$ in the group algebra of $G$ in characteristic $p$, and this prompted Puig to give an axiomatic definition for an abstract fusion system. More precisely, a fusion system on a finite $p$-group $S$ is a category $\Ff$ whose objects are the subgroups of $S$, and whose morphism sets model a system of conjugations among subgroups of $S$ induced by the inclusion of $S$ in an ambient object, without reference to the ambient object. (In particular every morphism is a group monomorphism and $\F$ contains all morphisms induced by conjugation in $S$.) 

Among fusion systems, the important ones are the \emph{saturated} fusion systems. Informally, a saturated fusion system on $S$ models a ``Sylow inclusion'' of $S$ in an ambient object. Formally, a fusion system is saturated if its morphism sets satisfy two axioms that mimic the Sylow theorems. One is a ``prime to $p$'' axiom, corresponding to the index of a Sylow subgroup being prime to $p$; and the other is a ``maximal extension'' axiom, replacing the result that a Sylow subgroup contains all $p$-subgroups up to conjugacy. Saturated fusion systems are now widely studied. In modular representation theory they are considered a helpful venue in which to reformulate and approach the Alperin weight conjecture \cite{markus:AlperinWeight}.
Building on work of Martino--Priddy, Broto--Levi--Oliver popularized saturated fusion systems among homotopy theorists as a model for studying the $p$-completed classifying space of a finite group. Lately saturated fusion system have been embraced by group theorists as a possible framework for one of the masterpieces of modern mathematics: the classification of finite simple groups. Recent work of Aschbacher in \cite{aschbacher:NormalSubsystems}, \cite{aschbacher:S3Free}, \cite{aschbacher:Characteristic2Type} and \cite{aschbacher:GenFittingSubsystem}, and Aschbacher--Chermak \cite{aschbacher-chermak:GroupApproachToSol} transports deep group theoretic tools into the world of fusion systems.

Saturated fusion systems were first defined by Puig, who originally called them full Frobenius systems. His definition was simplified by Broto--Levi--Oliver in \cite{BLO2}, and we follow their conventions and terminology. Further simplifications of the saturation axioms were made by the second author and Kessar in \cite{kessarstancu}, and by Roberts and Shpectorov in \cite{RobertsShpectorov}. These simplified (but equivalent) definitions are all of a similar nature: each consists of a prime to $p$ axiom and an extension axiom on the morphism sets, with one or both axioms being weaker than in the Broto--Levi--Oliver definition. In this paper we give a substantially different characterization of saturated fusion systems. Instead of axioms on morphism sets, we formulate the saturation property for a fusion system on $S$ in the double Burnside ring $A(S,S)$ of left-free $(S,S)$-bisets. 

The idea of relating fusion systems to bisets originates from Linckelmann--Webb. Looking at the $\Fp$-cohomology of fusion systems, they realized that in the case of a Sylow inclusion $S \leq G$, the $(S,S)$-biset $G$ plays a special role. Cohomology is a Mackey functor, and so an $(S,S)$-biset induces a map $H^*(BS;\Fp) \to H^*(BS;\Fp)$. The map induced by the $(S,S)$-biset $G$ is idempotent up to scalar with image isomorphic to $H^*(BG;\Fp)$. Linckelmann--Webb synthesized the properties of the $(S,S)$-biset $G$ and found the appropriate replacement $\Omega$ for an abstract fusion system $\F$ on $S$ that allows one to think of the cohomology of $\F$ as the image of the map  $H^*(BS;\Fp) \to H^*(BS;\Fp)$ induced by $\Omega$. To this end they defined a \emph{characteristic biset} $\Omega$ for $\F$ to be an $(S,S)$-biset with augmentation $|S\backslash \Omega|$ prime to $p$ that is $\F$-stable in the sense that restricting either $S$-action to a subgroup $P$ via a morphism in $\F$ yields a result isomorphic to restriction along the inclusion $P \hookrightarrow S$, and also satisfies an additional condition relating the irreducible components of $\Omega$ to $\F$ (see Section \ref{sec:Char} for a precise definition). More generally a \emph{characteristic element} for $\F$ is an element in the $p$-localized double Burnside ring $A(S,S)\pLoc$ that has the Linckelmann--Webb properties. 

The existence of characteristic elements for saturated fusion systems was established by Broto--Levi--Oliver in \cite{BLO2}. Characteristic elements are by no means unique. Indeed, a saturated fusion system $\F$ has infinitely many of them. However, the first author showed in \cite{KR:ClSpectra} that among characteristic elements there is exactly one that is idempotent in the $p$-local double Burnside ring, and we refer to this as the \emph{characteristic idempotent} of $\F$. Another result from \cite{KR:ClSpectra} shows that a fusion system can be reconstructed as the stabilizer fusion system of any characteristic element $\Omega$, meaning the largest fusion system $\F$ with respect to which $\Omega$ is $\F$-stable. Thus a characteristic element contains exactly the same information as its fusion system. In light of this, it is natural to ask whether non-saturated fusion systems might admit characteristic elements (and hence idempotents). Our first main result answers this question and also provides the first novel characterization of saturated fusion systems.

\begin{mthm}[\cite{puig:book}]\label{mthm:CharImpliesSat}
A fusion system is saturated if and only it has a characteristic element.
\end{mthm}

This theorem was also proved by Puig in \cite{puig:book}. We attribute the result to Puig, but also present our independently discovered proof of Theorem \ref{mthm:CharImpliesSat} in Section \ref{sec:CharImpliesSat} as it is an easy consequence of the lemmas need to prove the more important results that follow. We are grateful to Serge Bouc for pointing out the connection.

Since the characteristic idempotent of a saturated fusion system is unique, Theorem \ref{mthm:CharImpliesSat} gives a bijection between saturated fusion systems and their characteristic idempotents. This interplay is interesting and useful, and raises the question of which fusion-related phenomena can be reformulated in terms of characteristic idempotents, which we address in Section \ref{sec:Reformulate}. However, since the Linckelmann--Webb properties of a characteristic element are defined in terms of the fusion system for which it is characteristic, we need to have an intrinsic criterion for recognizing characteristic elements, without referring to their fusion systems. This is the subject of our next main result.

For $(S,S)$-bisets $X$ and $Y$, let $(X \times Y)_\Delta$ be the biset with $(S\times S)$ acting coordinatewise on the left, and $S$ acting on the right via the diagonal. We say that $X$ satisfies \emph{Frobenius reciprocity} if there is an isomorphism of bisets
\begin{equation} \label{eq:FrobeniusIntro}
 (X \times X)_\Delta \cong (X \times 1)_\Delta \circ X,
\end{equation}
where $- \circ - = - \times_S -$ is the multiplication in the double Burnside ring and $1 = S$ is the unit. More generally, an element in $A(S,S)\pLoc$ satisfies Frobenius reciprocity if it satisfies the obvious linearized version of \eqref{eq:FrobeniusIntro}.  Frobenius reciprocity is discussed in more detail in Section \ref{sec:FrobImpliesSat}, and the connection to classical Frobenius reciprocity in group cohomology is explained in \ref{subsec:Translate}.

Our second main result shows that ---amazingly--- Frobenius reciprocity is equivalent to saturation. To state the theorem we need some technical conditions. An $(S,S)$-biset $X$ is \emph{bifree} if both the left and right $S$-actions are free, and a general element in $A(S,S)\pLoc$ is bifree if it is in the subring generated by bifree bisets. Considering only bifree elements imposes no real restriction since characteristic elements are always bifree. A \emph{right-characteristic element} for a fusion system $\F$ is one that is stable under restricting the right $S$-action along morphisms in $\F$, but not necessarily the left action. The \emph{right-stabilizer fusion system} of a right-characteristic element recovers the fusion system.

\begin{mthm}\label{mthm:FrobImpliesSat}
Let $S$ be a finite $p$-group. A  bifree element in $A(S,S)\pLoc$ with augmentation not divisible by $p$ is a right characteristic element for a saturated fusion system on $S$ if and only it satisfies Frobenius reciprocity. 
\end{mthm}

The characteristic idempotent of a saturated fusion system is also the unique idempotent right-characteristic element. Therefore the correspondence between saturated fusion systems and characteristic idempotents can be combined to obtain a striking result.

\begin{mthm}\label{mthm:Bijection}
For a finite $p$-group $S$, there is a bijective correspondence between saturated fusion systems on $S$ and bifree idempotents in $A(S,S)\pLoc$ of augmentation $1$ that satisfy Frobenius reciprocity. The bijection sends a saturated fusion system to its characteristic idempotent, and an idempotent to its stabilizer fusion system.
\end{mthm}

This result gives us a completely new way to think of saturated fusion systems, offering the possibility of a much simpler definition. Rather than looking at a category of subgroups of $S$ with Sylow-like axioms on the morphism sets, we can encode saturation in the one-line Frobenius reciprocity relation \eqref{eq:FrobeniusIntro}. Theorem \ref{mthm:Bijection} opens up new avenues of research, and provides insight into the nature of saturated fusion systems and their role in representation theory and stable homotopy.
Theorems \ref{mthm:FrobImpliesSat} and \ref{mthm:Bijection} are proved in Section \ref{sec:FrobImpliesSat}.

A modified version of Theorem \ref{mthm:Bijection}, where we replace the bifreeness condition by demanding that an element is dominant, meaning that it is not contained in the Nishida ideal $J(S) \subseteq A(S,S)\pLoc$ (\cite{Ni}), is proved in \ref{subsec:Relax}. This formulation lends itself better to applications in the stable homotopy theory of classifying spaces, and in the second part of the paper ---Sections \ref{sec:Splittings} to \ref{sec:RTT}--- we consider some of these applications. The key to this is the Segal conjecture, proved by Carlsson in \cite{Car}, which allows us to identify the group of homotopy classes of stable selfmaps of $BS_+$, where the subscript $+$ denotes an added, disjoint basepoint, with the submodule of $\pComp{A(S,S)}$ consisting of elements with integer augmentation. Under this identification, the characteristic idempotent of a saturated fusion system $\F$ gives rise to a stable idempotent selfmap of $BS_+$. Taking the mapping telescope of this idempotent gives us a stable summand of $BS_+$, which we denote $\ClSpectrum{\F}$ and refer to as the \emph{classifying spectrum of $\F$}. This construction was suggested by Linckelmann--Webb and studied extensively in \cite{KR:ClSpectra}, where it was shown that classifying spectra of saturated fusion systems have all the important homotopy-theoretic properties associated to suspension spectra of $p$-completed classifying spaces. The Segal conjecture and classifying spectra are discussed in Section \ref{sec:Stable}.

The mapping telescope construction can be applied to any idempotent selfmap of $BS_+$ to obtain a stable summand. This gives rise to a theory of stable splittings of classifying spaces, which was the focus of much research in the 1980's and 1990's. An interesting example is that of a Sylow inclusion $S \leq G$, where a simple transfer argument shows that $\pComp{BG_+}$ is a stable summand of $BS_+$. A long-standing question in the subject asks when a stable summand of $BS_+$ has the stable homotopy type of a $p$-completed classifying space. Expanding the scope of the question to allow classifying spectra of saturated fusion systems, we obtain the following answer in Section \ref{sec:Splittings}.

\begin{mthm} \label{mthm:Splitting}
Let $S$ be a finite $p$-group. A stable summand of $BS_+$ has the homotopy type of the classifying spectrum of a saturated fusion system on $S$ if and only if it can be split off by an idempotent in $\pComp{A(S,S)}$ that is dominant and satisfies Frobenius reciprocity.
\end{mthm}

Another application allows us to generalize a variant of a theorem proved by Adams--Wilkerson in \cite{AW}. Adams--Wilkerson gave a criterion for determining when a subring of the $\Fp$-cohomology of a torus is a ring of invariants under the action of a group $W$ of order prime to $p$. Using ideas of Lannes, Goerss--Smith--Zarati showed in \cite{GSZ} that this criterion applies equally well to recognizing a ring of invariants in the $\Fp$-cohomology of an elementary abelian $p$-group $V$ under the action of a group $W \leq \Aut(V)$ of order prime to $p$, and in \cite{KR:Transfers&plfgs} the first author reformulated this criterion in terms of a Frobenius reciprocity condition. We can now prove a generalized version of this result, where we generalize from an elementary abelian $p$-group $V$ to an arbitrary finite $p$-group $S$, and from rings of invariants under a group of automorphisms to modules of elements that are stable with respect to a saturated fusion system $\F$ on $S$.  To accommodate this generalization, we need to lift the result from $\Fp$-cohomology to the double Burnside ring. This becomes slightly complicated, as the ring structure in cohomology is of a very different nature from the ring structure of the double Burnside ring. Instead of subrings of $A(S,S)\pLoc$, we must therefore consider submodules of $A(S,S)\pLoc$ that behave well with respect to the diagonal of $S$, yielding the following result. The proof can be found in Section \ref{sec:AW} along with definitions of the terms used in the statement. 

\begin{mthm} \label{mthm:AWAnalogue}
Let $S$ be a finite $p$-group, and let $R$ be a $\kappa$-preserving submodule of $A(S,S)\pLoc$. There exists a saturated fusion system $\F$ on $S$ such that $R$ is the module of $\F$-stable elements in $A(S,S)\pLoc$ if and only if $R$ is not contained in the Nishida ideal of $A(S,S)\pLoc$ and the inclusion $R \hookrightarrow A(S,S)\pLoc$ admits a retractive transfer.
\end{mthm}

Our last application is to a conjecture of Haynes Miller, which provided the starting point for the work in this paper as well as \cite{KR:ClSpectra} and \cite{KR:Transfers&plfgs}.

In their seminal paper \cite{BLO2}, Broto--Levi--Oliver introduced $p$-local finite groups as an abstract framework for the $p$-local homotopy theory of classifying spaces of finite groups. A $p$-local finite group is a triple $(S,\F,\Link)$, where $\F$ is a saturated fusion system on a finite $p$-group $S$, and $\Link$ is a \emph{centric linking system} associated to $\F$. The latter is a category whose $p$-completed geometric realization $\ClSp$ is considered the classifying space of the $p$-local finite group. The definition of a centric linking system is recalled in Section \ref{sec:RTT}. 

Miller conjectured that an equivalent theory should be obtained by considering a homotopy monomorphism $f \colon BS \to X$, where $X$ is a $p$-complete, nilpotent space, and a stable map $t\colon \PtdStable{X} \to \PtdStable{BS}$ that satisfies $\Stable{f} \circ t \simeq 1_{\PtdStable{X} }$ and the Frobenius reciprocity relation 
\[ (1_{\PtdStable{X}}\wedge t) \circ \Delta_X \simeq (\PtdStable{f} \wedge 1_{\PtdStable{BS}}) \circ \Delta_{BS} \circ t\, ,\]
where $\Delta$ denotes diagonal maps. We refer to $t$ as a \emph{retractive transfer for $f$}, and to $(f,t,X)$ as a \emph{retractive transfer triple on $S$}. This question was taken up in the first author's thesis, where it was shown that a retractive transfer triple on an elementary abelian $p$-group does indeed have the homotopy type of a $p$-local finite group. The converse was treated generally, showing that for a $p$-local finite group $(S,\F,\Link)$ on \emph{any} finite $p$-group $S$, the natural inclusion $\theta \colon BS \to \ClSp$ admits a unique retractive transfer $t$, and $(\theta,t,\ClSp)$ is a retractive transfer triple on $S$. These results appeared in \cite{KR:Transfers&plfgs}. 

In Section \ref{sec:RTT} we make significant progress toward proving Miller's conjecture. Theorem \ref{mthm:Bijection} immediately implies that a retractive transfer triple $(f,t,X)$ gives rise to a saturated fusion system, corresponding to the idempotent $t \circ \PtdStable{f}$ (Proposition \ref{prop:F(RTT)}). Making two additional technical assumptions, we prove the following theorem.
%This fusion system consists of homomorphisms that stabilize $f$ in stable homotopy, and assuming that 
% assumptions we can show that this fusion system is the same as the one consisting of morphisms that stabilize $f$ in \emph{unstable} homotopy, and that $X$ has the homotopy type of the classifying space of $p$-local finite group.

\begin{mthm} \label{mthm:pLFGasRTT}
Let $(f,t,X)$ be a dominant retractive transfer triple on a finite $p$-group $S$ and assume that
\begin{enumerate}
  \item for every $P \leq S$, the map $\Stable \colon [BP,X] \to \{BP_+,X_+\}$ is injective; and
  \item $f$ preserves $\F_{S,f}(X)$-centric subgroups.        
\end{enumerate}
Then $(S,\F_{S,f}(X),\Link_{S,f}(X))$ is a $p$-local finite group with classifying space homotopy equivalent to $X$. 
\end{mthm}
This theorem is proved in Section \ref{sec:RTT} as Theorem \ref{thm:RTTgivesPLFG}, where the second condition is stated more precisely and the remaining terminology is explained.

The additional conditions in Theorem \ref{mthm:pLFGasRTT} are quite strong, but reasonable in the sense that they are always satisfied by the classifying space of a $p$-local finite group. Theorem \ref{mthm:pLFGasRTT} is more than a special case of Miller's conjecture, since, combined with results from \cite{KR:Transfers&plfgs}, it establishes an equivalence between $p$-local finite groups and retractive transfer triples satisfying the two conditions. This is still not a satisfactory resolution to Miller's conjecture as the additional conditions are very hard to check in practice and a priori appear very restrictive. Evidence suggests that the conditions are actually always satisfied by a retractive transfer triple, and Theorem \ref{mthm:pLFGasRTT} should properly regarded as reducing Miller's conjecture to proving that this is the case. This challenge will be taken up in a subsequent paper.

\subsection*{Outline} The paper has three main parts. In the first part, we recall background material and establish notational conventions that will be used throughout the paper. The theory of fusion systems is recalled in Section \ref{sec:Fusion}; bisets, the double Burnside ring and fixed-point homomorphisms are covered in Section \ref{sec:Burnside}; and Section \ref{sec:Char} contains a discussion of characteristic bisets and idempotents. The second part of the paper deals with the new results on fusion systems. In Section \ref{sec:BisetFusion} we introduce the notions of stabilizer, fixed-point, and orbit-type fusion systems, list their basic properties and reformulate the Linckelmann--Webb properties in this context. In Section \ref{sec:CharImpliesSat} we set up congruences for the fixed points of characteristic bisets, and use these to tease out the saturation axioms, proving Theorem \ref{mthm:CharImpliesSat}. Section \ref{sec:FrobImpliesSat} is the focal point of the paper, in which we introduce Frobenius reciprocity in the double Burnside ring and, using tools from Section \ref{sec:CharImpliesSat}, show that it implies saturation. Section \ref{sec:Reformulate} contains reformulations of selected fusion-theoretic phenomena in the language of idempotents. The third part of the paper covers applications to algebraic topology. In Section \ref{sec:Stable} we make the transition from algebra to stable homotopy by recalling the Segal conjecture, proving versions of Theorems \ref{mthm:FrobImpliesSat} and \ref{mthm:Bijection} that lend themselves better to interpretation in stable homotopy, and recalling the construction of classifying spectra for saturated fusion systems. The application to stable splittings is then covered in Section \ref{sec:Splittings}, followed by the generalization of the Adams--Wilkerson theorem in Section \ref{sec:AW}, and finally the application to Miller's conjecture in Section \ref{sec:RTT}.

\subsection*{Acknowledgements} This project was started while the authors attended the program on Representation Theory of Finite Groups at MSRI in the spring of 2008, with loose ends tied up at the workshop on Representation Theory of Finite Groups at Oberwolfach in March 2009. The authors are grateful to these institutions for their support, and for providing the opportunity for this collaboration. The first author also thanks his thesis advisor, Haynes Miller, for posing the question that motivated this work. Without his suggestion the far-fetched idea of proving saturation through Frobenius reciprocity would probably never have seen the light of day.

\section{Fusion systems} \label{sec:Fusion}
Fusion systems and their saturation axioms were introduced by Puig \cite{puig} in an effort to axiomatize the $p$-local structure of a finite group and, more generally, of a block of a group algebra. Broto, Levi and Oliver developed this axiomatic approach in \cite{BLO2}, and gave a different set of saturation axioms which they prove to be equivalent to Puig's definition. In this section we present Broto, Levi and Oliver's axiomatic system, that is adopted through our paper. We include in this section some basic properties of fusion systems together with a new simplification of the saturation axioms. We also introduce the concept of pre-fusion system, which is a structure designed to keep track of a set of homomorphisms used to generate a fusion systems.

\subsection{Basic notations and definitions}
For subgroups $H$ and $K$ of a finite group $G$, denote the {\it transporter} from $H$ to $K$ in $G$ by $N_G(H,K)\defeq \{g\in G| c_g(H)\le K\}$ where $c_g(x) \defeq gxg^{-1}$ is the {\it conjugation homomorphism}. For $g \in G$ we write $\lsup{H}{g}$ for $gHg^{-1}$, and $H^g$ for $g^{-1}Hg$. We say that $H$  and $K$ are $G$-conjugate if $\lsup{H}{g}=K$ for some $g \in G$, and denote the $G$-conjugacy class of $H$ by $[H]_G$. Also, as usual, the {\it normalizer} of $P$ is denoted by $N_S(P) \defeq N_S(P,P)$ and the {\it centralizer} of $P$ is $C_S(P) \defeq \{y\in N_S(P)\,|\,c_y|_P=\id_P\}$. Other useful notations: $\Hom_S(P,Q) \defeq N_S(P,Q)/C_S(P)$ and $\Aut_S(P) \defeq \Hom_S(P,P)$. 

\begin{defn}
A {\it fusion system $\Ff$ on a finite $p$-group $S$} is a category whose objects are the subgroups of $S$ and whose set of morphisms between the subgroups $P$ and $Q$ of $S$, is a set $\Hom_\Ff(P,Q)$ of injective group homomorphisms from $P$ to $Q$, with the following properties:
\begin{enumerate}
\item[(1)] $\Hom_S(P,Q)\subset\Hom_\Ff(P,Q)$;
\item[(2)] for any $\varphi\in\Hom_\Ff(P,Q)$ the induced isomorphism $P\simeq\varphi(P)$ and its inverse are morphisms in $\Ff$;
\item[(3)] the composition of morphisms in $\Ff$ is the usual composition of group homomorphisms.
\end{enumerate}
\end{defn}

Let $\Ff$ be a fusion system on $S$. We say that two subgroups $P$ and $Q$ of $S$ are \emph{$\Ff$-conjugate} if there exist an isomorphism $\varphi\in\Hom_\Ff(P,Q)$. The $\Ff$-conjugacy class of $P$ is denoted by $[P]_\Ff$. 

\begin{defn}
Let $\F$ be a fusion system on a finite $p$-group $S$, and let $P \leq S$
\begin{enumerate}
\item[(a)] $P$ is {\it fully $\Ff$-centralized} if $|\C SP|\ge|\C S{Q}|$ for all $Q\in [P]_\Ff$,
\item[(b)] $P$ is {\it fully $\Ff$-normalized} if $|\N SP|\ge|\N S{Q}|$ for all $Q\in [P]_\Ff$,
\item[(c)] $P$ is {\it $\Ff$-centric} if $C_S(Q) = Z(Q)$ for all $Q\in [P]_\Ff$.
\end{enumerate}
\end{defn}
We define the  \emph{$\Ff$-representations} from $P$ to $Q$ as the quotient 
 \[ \Rep_\Ff(P,Q) \defeq \Aut_Q(Q) \backslash \Hom_\Ff(P,Q) \, , \] 
 and the \emph{outer automorphisms} in $\Ff$ as
 \[ \Out_\Ff(P) \defeq \Rep_\Ff(P,Q) = \Aut_P(P) \backslash \Aut_\Ff(P) \, .\]

\subsection{Morphisms of fusion systems and fusion-preserving homomorphisms}

Let $\Ff$ be a fusion system on $S$ and $\Ff'$ be a fusion system on $S'$. 

\begin{defn}\label{defn:FusMorph}
A {\it morphism} from $\Ff$ to $\Ff'$ is a pair $(\alpha,\alpha_0)$, where $\alpha:\Ff\to\Ff'$ is a covariant functor and $\alpha_0:S\to S'$ is a group homomorphism satisfying $\alpha(P)=\alpha_0(P)$ and $\alpha(\phi)\circ\alpha_0(u)=\alpha_0\circ\phi(u)$ for all $u\in P\le S$ and 
$\phi\in\Hom_\Ff(P,S)$.
The {\it kernel} of this morphism is $\Ker(\alpha):=\Ker(\alpha_0)\le S$ and the {\it image}, denoted by $\Image(\alpha)$ is the fusion system on $\alpha(S)$ with morphism sets given by the image of $\alpha$.
\end{defn}

\begin{defn}\label{defn:FusPres}
We say that a group homomorphism $\beta:S\to S'$ is $(\Ff,\Ff')$-\emph{fusion preserving} if 
$$\beta|_Q\circ\Hom_\Ff(P,Q)\subset\Hom_{\Ff'}(\beta(P),\beta(Q))\circ\beta|_P\,.$$
\end{defn}

If $(\alpha,\alpha_0):\Ff\to\Ff'$ is a morphism of fusion systems then $\alpha_0$ is $(\Ff,\Ff')$-fusion preserving. Conversely, if $\alpha_0 \colon S\to S'$ is an $(\Ff,\Ff')$-fusion-preserving homomomorphism then $\alpha_0$ induces a unique functor $\alpha \colon \Ff \to \Ff$ such that the pair $(\alpha,\alpha_0)$ is a morphism of fusion systems.

\subsection{Saturation axioms}
Fusion systems provide a model for the conjugation action on $S$ by an ambient object, but this model is far too general to be interesting in practice. Classically, the interesting fusion systems are the ones coming from the $p$-local structure of a finite group or of a block of the group algebra of a finite group in characteristic $p$. In both cases, the fusion systems satisfy certain axioms that correspond to the Sylow theorems, and this generalizes to the following definition, originally due to Puig \cite{puig} but presented here in the form developed by Broto--Levi--Oliver \cite{BLO2}. 

\begin{defn} \label{defn:SatAxioms}
A fusion system $\Ff$ on a finite $p$-group $S$ is {\it saturated} if it satisfies the following axioms:
\begin{enumerate}
\item[I] If $P\le S$ is fully $\Ff$-normalized then $P$ is fully $\Ff$-centralized and $\Aut_S(P)$ is a Sylow $p$-subgroup of $\Aut_\Ff(P)$.
\item[II] If $\varphi\in\Hom_\Ff(P,S)$ is a homomorphism such that $\varphi(P)$ is fully $\Ff$-centralized, then $\varphi$ extends to a morphism $\overline\varphi \in \Hom_\Ff(N_\varphi,S)$, where
$$N_\varphi=\{x\in\N SP\,|\,\exists y\in\N S{\varphi(P)},\,\varphi(\,^xu)=\,^y\varphi(u),\,\forall u\in P\}\,.$$
\end{enumerate}
\end{defn}

Notice that $N_\varphi$ is the largest subgroup of $\N SP$ such that
$^\varphi(N_\varphi/\C SP)\le\Aut_S(\varphi(P))$. Thus we always have $P\C SP\le N_\varphi \le N_S(P)$. Also notice that if $R \leq \N SP$ is a subgroup containing $P$ to which $\varphi$ can be extended in $\Ff$, then $R \leq N_\varphi$. More generally, if $\varphi$ extends in $\Ff$ to a group $R$ with  $P < R \leq S$, then $P \leq N_R(P) = N_\varphi \cap R$. In particular, if $N_\varphi = P$, then $\varphi$ can not be extend in $\F$.

\subsection{Simplified saturation axioms}
Even if very useful in this form, the set of axioms in Definition \ref{defn:SatAxioms} has redundancies. An equivalent set of axioms is given in \cite{kessarstancu}, and in this subsection we 
%
%
%\begin{eqdefn}
%A fusion system $\Ff$ on $S$ is saturated if
%\begin{enumerate}
%\item[FN1.] $\Aut_S(S)$ is a Sylow $p$-subgroup of $\Aut_\Ff(S)$.
%\item[FN2.] Every $\varphi\in\Hom_\Ff(P,S)$, with the property that $\varphi(P)$ is fully $\Ff$-normalized, extends to a morphism $\overline{\varphi} \in \Hom_\Ff(N_\varphi,S)$.
%\end{enumerate}
%\end{eqdefn}
%
%
%
give another simplification of the Broto--Levi--Oliver set of axioms by imposing Axiom I only for $S$ while keeping Axiom II for all $P$.
\begin{eqdefn}
A fusion system $\Ff$ on $S$ is saturated if
\begin{enumerate}
 \item[I${}_S$] $\Aut_S(S)$ is a Sylow subgroup of $\Aut_\Ff(S)$.
 \item[II${}_{\hspace{5pt}}$] If $\varphi\in\Hom_\Ff(P,S)$ is a homomorphism such that $\varphi(P)$ is fully $\Ff$-centralized, then $\varphi$ extends to a morphism $\overline\varphi \in \Hom_\Ff(N_\varphi,S)$.
\end{enumerate}
\end{eqdefn}

We give here a proof that this new set of axioms implies the set of axioms in Definition \ref{defn:SatAxioms}. Suppose that we have a fusion system $\Ff$ on $S$ satisfying Axioms I${}_S$ and II. We prove Axiom I by induction.

\begin{lem} \label{lem:FCtoBLO}
Let $P$ be a fully $\Ff$-centralized subgroup of $S$ with $|\N SP|\ge|\N S{Q}|$ for any fully $\Ff$-centralized subgroup $Q\in [P]_\Ff$. Then $\Aut_S(P)$ is a Sylow $p$-subgroup of $\Aut_\Ff(P)$.
\end{lem}

\begin{proof}
The proof is by induction on the index $n$ of $P$ in $S$. If $n=1$ then $P=S$ and we are done using Axiom I${}_S$. Let $n>1$ and suppose the assumption in the lemma true for all subgroups of $S$ of index smaller than $n$.

We suppose that the assumption is not true for $P$ and try to reach a contradiction. With this assumption, $\Aut_S(P)$ is properly contained in a Sylow $p$-subgroup $U$ of $\Aut_\Ff(P)$. Thus $\Aut_S(P)$ is a proper subgroup of its normalizer $\N U{\Aut_S(P)}$. Take a morphism $\varphi\in\N U{\Aut_S(P)}\setminus \Aut_S(P)$.

Then $N_\varphi=\N SP$ as $\varphi$ normalizes $\Aut_S(P)$. As $P$ is fully  $\Ff$-centralized, Axiom II implies that $\varphi$ extends to an automorphism $\overline\varphi$ of $\N SP$. Given that $\varphi$ is a $p$-automorphism, we can choose $\overline\varphi$ to be also a $p$-automorphism. Now, by induction, there exists a fully $\Ff$-centralized subgroup $T$ of $S$ which is $\Ff$-isomorphic to $\N SP$ and satisfying that $\Aut_S(T)$ is a Sylow $p$-subgroup of $\Aut_\Ff(T)$. 

Any isomorphism $\psi:\N SP\to T$ sends $\overline\varphi$ to an automorphism  $\theta:=\psi\overline\varphi\psi^{-1}\in\Aut_\Ff(T)$. Given that $\overline\varphi$ is a $p$-automorphism, by modifying $\psi$, if necessary, by an automorphism in $\Aut_\Ff(T)$, 
we can suppose that $\theta\in\Aut_S(T)$. In particular there exists $u\in S$ such that $\theta(x)=uxu^{-1},\,\forall x\in T$. Moreover $u\psi(P)u^{-1}=\psi\overline\varphi\psi^{-1}(\psi(P))=\psi(P)$ so $u\in\N S{\psi(P)}=T$. The last equality is true given that $P$ and $\varphi(P)$ are both fully  $\Ff$-centralized and $P$ is chosen so that $|\N SP|\ge|\N S{\varphi(P)}|$. Thus the inclusions $\psi(\C SP)\subseteq\C S{\psi(P)}$ and $\psi(\N SP)\subseteq\N S{\psi(P)}$ are both equalities.

Hence we have $\psi^{-1}(u)\in\N SP$ and, moreover, $\varphi(v)=\psi^{-1}(u)\,v\,(\psi^{-1}(u))^{-1},\,\forall v\in P$ which is a contradiction to the supposition that $\varphi\not\in\Aut_S(P)$.
\end{proof}

\begin{lem}
Every fully $\Ff$-normalized subgroup of $S$ is fully $\Ff$-centralized.
\end{lem}

\begin{proof}
Let $P$ be fully $\Ff$-normalized and take $Q$, $\Ff$-isomorphic to $P$, fully $\Ff$-centralized with $\Aut_S(Q)$ a Sylow $p$-subgroup of 
$\Aut_\Ff(Q)$. Such a subgroup $Q$ exists by Lemma \ref{lem:FCtoBLO}. Take $\varphi\in\Hom_\Ff(P,Q)$ such that $\,^\varphi\Aut_S(P)\le\Aut_S(Q)$.
Then $N_\varphi=\N SP$ and $\varphi$ extends to $\overline\varphi\in\Hom_\Ff(\N SP,\N SQ)$. But given that $P$ is fully $\Ff$-normalized we have that
$|\N SQ|=|\N SP|$ and thus $Q$ is also fully $\Ff$-normalized. The fact that $|\Aut_S(R)|\le|\Aut_S(Q)|$ combined with $|\N SP|=|\N SQ|$ gives 
$|\C SP|\ge|\C SQ|$ which in turn has to be an equality as $Q$ is fully $\Ff$-centralized. Hence $P$ is also fully $\Ff$-centralized.
\end{proof}

The previous two lemmas give Axiom I for all $P\le S$.

\subsection{Normal fusion subsystems and quotient fusion systems}

Analogous notions to normal subgroups of finite groups were introduced for fusion systems by Linckelmann \cite{markus}. Normal fusion subsystems of a fusion $\Ff$ on $S$ are constructed on subgroups of $S$ that are closed with respect to the fusion in $\Ff$.

\begin{defn}
Let $\Ff$ be a fusion system on $S$ and $P\le S$. We say that $P$ is 
{\it strongly $\Ff$-closed} if for all $Q\le P$ and all $\varphi\in\Hom_\Ff(Q,S)$ we have $\varphi(Q)\le P$.
\end{defn}
\begin{defn}
Let $\Ff$ be a fusion system on a finite $p$-group $S$ and $\Ee$ a fusion 
subsystem of $\Ff$ on a subgroup $P$ of $S$. We say that $\Ee$ is normal 
in $\Ff$ if $P$ is strongly $\Ff$-closed and, for every $\Ff$-isomorphism 
$\varphi:Q\to R$ and any two subgroups $T$, $U$ of $Q\cap P$, 
we have
$$\varphi\circ\Hom_{\Ee}(T,U)\circ\varphi^{-1}\subseteq
			\Hom_{\Ee}(\varphi(T),\varphi(U))\,.$$
\end{defn}

Aschbacher gave a very useful characterization of normal fusion subsystems in \cite{aschbacher:NormalSubsystems} which we recall below. Note that Aschbacher uses a different terminology, calling {\it invariant} fusion subsystem what we call normal fusion subsystem here.

\begin{thm}[\cite{aschbacher:NormalSubsystems}]\label{thm:NormalFusion} 
Let $\Ff$ be a saturated fusion system on $S$ and $\Ee$ be a saturated fusion subsystem of $\Ff$ on a strongly $\Ff$-closed subgroup $P$ of $S$. The following are equivalent:
\begin{enumerate}
\item[a)] $\Ee$ is a normal fusion subsystem of $\Ff$.
\item[b)] $\Aut_\Ff(P)\subseteq\Aut(\Ee)$ and for every $T,U\le P$ and $\psi\in\Hom_\Ff(T,U)$ there exist $\varphi\in\Hom_\Ee(T,U)$ and $\chi\in\Aut_\Ff(P)$ such that $\chi\vert_{\varphi(T)} \circ \varphi =\psi$.
\end{enumerate}
\end{thm}

Also analogous to the case of finite groups, Puig \cite{puig} defines quotient fusion systems.
\begin{defn}
Let $\Ff$ be a fusion system on $S$, and let $P$ be a strongly $\Ff$-closed subgroup of $S$. By the {\it quotient system $\Ff/P$}, we mean the fusion system on $S/P$, such that for any two subgroups $R$ and $Q$ of $S$ containing $P$, we have that $\Hom_{\Ff/P}(R/P,Q/P)$ is the set of homomorphisms induced on the quotients by $\Hom_\Ff(R,Q)$.
\end{defn}

With the above notations, the canonical projection $\pi:S\to S/P$ is $(\Ff,\Ff/P)$-fusion preserving. When $\Ff$ is saturated, $\pi$ induces a morphism of fusion systems. This is a consequence of a theorem of Puig \cite{puig}.

\begin{thm}[\cite{puig}]\label{thm:QuotientFusion} 
Let $\Ff$ be a saturated fusion system on $S$ and let $P$ be a strongly $\Ff$-closed subgroup of $S$. Then the quotient system $\Ff/P$ is saturated. Moreover the canonical projection $\pi \colon S\to S/P$ induces a morphism of fusion systems from $\Ff$ to $\Ff/P$.
\end{thm}

\begin{rmk}\label{rmk:KerAlpha}
It is easy to see that if $(\alpha,\alpha_0)$ is a morphism of fusion systems from $\Ff$ to $\Ff'$ then $\Ker(\alpha_0)$ is a strongly $\Ff$-closed subgroup. Moreover, if $\Ff$ is saturated, then $\Ff/\Ker(\alpha_0)$ is isomorphic to $\Image(\alpha)$. However, when $\Ff$ is not saturated, $\Image(\alpha)$ may not even be a fusion system, as non-composable morphisms may be sent to composable ones.
\end{rmk}

\subsection{Pre-fusion systems}

A fusion system on a given finite $p$-group $S$ is determined by its morphism sets. Thus one can construct a fusion system $\F$ by specifying a set of morphisms it should contain and then taking $\F$ to be the fusion system generated by those morphisms. This approach will be taken in Section \ref{sec:Char} when we construct fusion systems from bisets. To capture this construction we introduce the notion of pre-fusion systems.

\begin{defn} \label{defn:PreFus}
A \emph{pre-fusion system} on a finite $p$-group $S$ is a collection
\[ \PreFus = \{ \Hom_{\PreFus}(P,Q) \mid P,Q \leq S \} \,,\]
satisfying the following conditions
\begin{enumerate}
 \item $\Hom_{\PreFus}(P,Q) \subseteq \Inj(P,Q)$ for each pair of subgroups $P,Q \leq S$, 
  \item If $\varphi \in \Hom_{\PreFus}(P,Q)$ and $\varphi(P) \leq R \leq S$, then the composite $P \xrightarrow{\varphi} \varphi(P) \hookrightarrow R$ is in $\PreFus$. 
\end{enumerate}
\end{defn}

Note that a pre-fusion system need not be a category as we we require neither that the composite of two morphisms in $\PreFus$ is again in $\PreFus$, nor that the identity morphism of a subgroup in $S$ is in $\PreFus$. The second condition says we can restrict or extend morphisms in the target, and it follows that a pre-fusion system $\PreFus$ on $S$ is determined by the sets $\Hom_\PreFus(P,S)$.

We shall employ set operations on pre-fusion systems and fusion systems, with the understanding that the operations are applied to each morphism set. For example, if $\PreFus_1$ and $\PreFus_2$ are two pre-fusion systems on a finite $p$-group $S$, then $\PreFus_1 \cap \PreFus_2$ is the pre-fusion system with morphism sets $\Hom_{\PreFus_1}(P,Q) \cap \Hom_{\PreFus_2}(P,Q) $. It is easy to see that the intersection of two fusion systems is clearly a fusion systems, and this allows us to make the following definition.

\begin{defn}
The \emph{closure of $\PreFus$}, denoted $\overline{\PreFus}$, is the smallest fusion system on $S$ such that $\Hom_{\PreFus}(P,Q) \subseteq \Hom_{\overline{\PreFus}}(P,Q)$ for each pair of subgroups $P,Q \leq S$. We say that $\PreFus$ is \emph{closed} if $\overline{\PreFus} = \PreFus$.
\end{defn}

On occasion we will consider pre-fusion systems with a weaker form of closure.

\begin{defn}
A pre-fusion system $\PreFus$ on a finite $p$-group $S$ is \emph{level-wise closed} if the following holds for all $P,Q,R \leq S$.
\begin{enumerate}
 \item $\Hom_S(P,Q) \subseteq \Hom_\PreFus(P,Q)$.
 \item If $\varphi \in \Hom_\PreFus(P,Q)$ is a group isomorphism, then $\varphi^{-1} \in \Hom_\PreFus(Q,P)$.
 \item If $\varphi \in \Hom_\PreFus(P,Q)$ and $\psi \in \Hom_\PreFus(Q,R)$ are group isomorphisms, then $\psi \circ \varphi \in \Hom_\PreFus(P,R)$.
\end{enumerate}
\end{defn}

It is easy to show that a level-wise-closed pre-fusion system that is closed under restriction is closed. In a level-wise-closed pre-fusion system $\PreFus$ the morphism sets $\Hom_\PreFus(P,P)$ are groups of automorphisms, and we denote them by $\Aut_\PreFus(P)$. Furthermore, the notions of $\PreFus$-conjugacy, fully $\PreFus$-centralized and fully $\PreFus$-normalized subgroup extend to this context. Hence we can consider the following local saturation conditions.

\begin{defn} \label{defn:SatAtP}
Let $\PreFus$ be a level-wise-closed pre-fusion system on a finite $p$-group $S$. For a subgroup $P\leq S$, we say that $\PreFus$ is \emph{saturated at $P$} if the following two conditions hold.
\begin{enumerate}
 \item[I${}_P$] If $Q \in [P]_\PreFus$ is fully $\PreFus$-normalized, then $Q$ is fully $\PreFus$-centralized and $\Aut_S(Q)$ is a Sylow $p$-subgroup of $\Aut_\PreFus(Q)$.
 \item[II${}_P$] If $\varphi \in \Hom_\PreFus(P,S)$ is a homomorphism such that $\varphi(P)$ is fully $\PreFus$-centralized, then $\varphi$ extends to a homomorphism $\overline{\varphi} \in \Hom_{\overline{\PreFus}}(N_\varphi,S)$.
\end{enumerate}
\end{defn}

Clearly, a fusion system on $S$ is saturated if and only if it is saturated at every subgroup $P$ of $S$. As the notational distinction is slight, let us explicitly point out that in condition II${}_P$ we require only that the extension $\overline{\varphi}$ be in the closure $\overline{\PreFus}$, and not necessarily in $\PreFus$ itself.

%%%%%%%%%%%%%%%%%%%%%%%%%%%%%%%%%%%%%%%%%%%%%%%%%%%%%%%%%%%%%%%%%%%%%%%%%%%%%%%
\section{Bisets, the Burnside category and Mackey functors} \label{sec:Burnside}
%%%%%%%%%%%%%%%%%%%%%%%%%%%%%%%%%%%%%%%%%%%%%%%%%%%%%%%%%%%%%%%%%%%%%%%%%%%%%%%
In this section we recall the structure and main properties of the double Burnside ring $A(G,G)$ of $(G,G)$-bisets for a finite group $G$. In fact we work more generally, studying the modules $A(G,H)$ of finite $(G,H)$-bisets for finite groups $G$ and $H$. These modules form the morphism sets in the Burnside category $\Burnside$, and we think of Mackey functors as functors defined on this category. Throughout this section $G$, $H$ and $K$ will denote finite groups.

\subsection{The double Burnside module of bisets}
We begin by establishing our conventions for bisets.

\begin{defn}
A \emph{$(G,H)$-biset} is a set equipped with a right $G$-action and a left $H$-action, such that the actions commute. A biset is \emph{left-free} if the $H$-action is free, and \emph{right-free} if the $G$-action is free. A biset is \emph{bifree} if it is both left- and right-free. 
\end{defn} 

Given a $(G,H)$-biset $X$ one obtains a $(H \times G)$-set $\widehat{X}$, with the same underlying set, and $(H \times G)$-action given by $(h,g)x \defeq hxg^{-1}$. This gives a bijective correspondence between $(G,H)$-bisets and $(H\times G)$-sets, and it is often convenient to characterize a $(G,H)$-biset by the corresponding $(H \times G)$-set.

\begin{defn}
The \emph{Burnside module of $G$ and $H$} is the Grothendieck group $A(G,H)$ of isomorphism classes of finite, left-free $(G,H)$-sets. 
\end{defn}

That is, the isomorphism classes of finite, left-free $(G,H)$-bisets form a monoid with cancellation under disjoint union, and $A(G,H)$ is the group completion of this monoid.

The Burnside module $A(G,H)$ is a free abelian $\Z$-module with basis the isomorphism classes of indecomposable left-free $(G,H)$-sets. These are the $(G,H)$-sets corresponding to transitive $(H\times G)$-sets that restrict to free $(H \times 1)$-sets. We proceed to describe and parametrize this basis, starting with the next definition.

\begin{defn}
A \emph{$(G,H)$-pair} is a pair $(K,\varphi)$, consisting of a subgroup $K \leq G$ and a homomorphism $\varphi \colon K \to H$. 
\end{defn}

From a $(G,H)$-pair $(K,\varphi)$, one obtains a left-free $(G,H)$-set 
 \[ H \times_{(K,\varphi)} G \defeq (H\times G)/{\sim}\,, \]
where $\sim$ is the relation 
 \[ (x,ky) \sim (x\varphi(k),y),~ \text{for all~} x \in H, y \in G, k \in K\,, \]
and $G$ and $H$ act in the obvious way.
The corresponding $(H \times G)$-set is isomorphic to 
\[ (H\times G)/ \Delta(K,\varphi)\,, \]
where $\Delta(K,\varphi) \leq H\times G$ is the twisted diagonal (or graph) of $(K,\varphi)$, given by  
 \[ \Delta(K,\varphi) = \{  (\varphi(k),k) \mid k \in K  \} \,.\]
The isomorphism class of $(H\times G)/ \Delta(K,\varphi)$, and hence of $H \times_{(K,\varphi)} G$, is determined by the conjugacy class of $\Delta(K,\varphi)$. This motivates the next definition.

\begin{defn}
Let $(K,\varphi)$ and $(L,\psi)$ be
$(G,H)$-pairs. We say that $(K,\varphi)$ is \emph{conjugate} to $(L,\psi)$, 
and write $(K,\varphi) \con (L,\psi)$, if $\Delta(K,\varphi)$ is 
conjugate to $\Delta(L,\psi)$ in $H \times G$.  We refer to $(G,H)$-conjugacy classes of $(G,H)$-pairs as $(G,H)$-classes, and denote the $(G,H)$-class of $(K,\varphi)$ by $\langle K,\varphi \rangle$. 

Similarly we say that $(K,\varphi)$ is \emph{subconjugate} to $(L,\psi)$, and write $(K,\varphi) \subcon (L,\psi)$, if $\Delta(K,\varphi)$ is subconjugate to $\Delta(L,\psi)$. In this case we also say that the $(G,H)$-class $\langle K,\varphi \rangle$ is \emph{subconjugate} to $\langle L,\psi \rangle$, and write
 $\langle K,\varphi \rangle \subcon \langle L,\psi \rangle$.
\end{defn}

Notice that the subconjugacy relation $\langle K,\varphi \rangle \subcon \langle L,\psi \rangle$ implies that every representative of the $(G,H)$-class $\langle K,\varphi \rangle$ is subconjugate to every representative of the $(G,H)$-class $\langle L,\psi \rangle$. 

The following characterization of subconjugacy is sometimes useful. The proof is left to the reader.

\begin{lem} 
Let $(K,\varphi)$ and $(L, \psi)$ be  $(G,H)$-pairs. Then $(K,\varphi)$ is subconjugate to $(L,\psi)$ if and only if there exist $x \in N_G(K,L)$ and $y \in N_H(\varphi(K),\psi(L))$ 
such that $c_y \circ \varphi = \psi \circ c_x$. Conjugacy holds if and only if the additional condition $L=\lsup{K}{x}$ is satisfied.
\end{lem}

We can now describe the $\Z$-module structure of double Burnside modules.
\begin{lem} \label{lem:BurnsideBasis}
The Burnside module $A(G,H)$ is a free $\Z$-module with one basis element $[K,\varphi]$ for each $(G,H)$-class $\langle K, \varphi \rangle $.
\end{lem}
\begin{proof}
A proof has already been outlined in the discussion above. The only thing left to show 
is that a transitive $(H\times G)$-set that restricts to a free $(H \times 1)$-set must 
be isomorphic to $(H\times G)/\Delta(K,\varphi)$ for some $(G,H)$-pair $(K,\varphi)$. This is left to the reader.
\end{proof}

Notice in particular that $A(G,H)$ is a finitely generated $\Z$-module, and hence Noetherian. Therefore the $p$-localization ${A(G,H)}\pLoc$ and $p$-completion $\pComp{A(G,H)}$ can be obtained by tensoring with $\Z\pLoc$ and $\Zp$, respectively. It follows that Lemma \ref{lem:BurnsideBasis} holds after $p$-localization or $p$-completion. 

The basis described in Lemma \ref{lem:BurnsideBasis} will be used throughout the paper, and we refer to it as the \emph{standard basis} of $A(G,H)$. 

\begin{defn}
For each $(G,H)$-class $\langle K,\varphi \rangle$, let 
$c_{\langle K,\varphi \rangle} \colon A(G,H) \to \Z$ be the homomorphism sending $X \in A(G,H)$ to the coefficient at $[K,\varphi]$ in the standard basis decomposition of $X$.
\end{defn}
The homomorphisms $c_{\langle K,\varphi \rangle}$ are equivalently defined by requiring that
 \[ X = \sum_{\langle K,\varphi \rangle} c_{\langle K,\varphi \rangle}(X) [K,\varphi]  \]
for every $X \in A(G,H)$. At times we will break this up into a double sum
 \[ X = \sum_{[K]_G} \left( \sum_{[\varphi] \in \Rep(K,H)}c_{\langle K,\varphi \rangle}(X) \; [K,\varphi] \right) \, ,\]
where the outer sum runs over $G$-conjugacy classes of subgroups, and the inner sum runs over $H$-conjugacy classes of morphsisms.

We also denote the $p$-localization or $p$-completion of $c_{\langle K,\varphi \rangle}$ by $c_{\langle K,\varphi \rangle}$.

\subsection{Fixed points} An alternative and extremely useful way to keep track of $(G,H)$-bisets is by fixed points, and this technique is fundamental to the proofs of the main theorems in this paper.

For a $(G,H)$-biset $X$, and a $(G,H)$-pair $(K,\varphi)$, we set 
 \[ X^{ (K,\varphi) } \defeq \{ x \in X \mid \forall k \in K: xk = \varphi(k)x  \}\,. \]  
Notice that $X^{ (K,\varphi) } = \widehat{X}^{\Delta(K,\varphi)}$ (as sets). As the number $|\widehat{X}^{\Delta(K,\varphi)}|$ does not change when we conjugate $\Delta(K,\varphi)$,
the same is true for the number $|X^{ (K,\varphi)}|$, and we can make the following
definition. 

\begin{defn}
For a $(G,H)$-class $\langle  K,\varphi \rangle$, let $\Phi_{\langle  K,\varphi \rangle} \colon A(G,H) \to \Z$ be the $\Z$-module homomorphism defined by setting
\[ \Phi_{\langle  K,\varphi \rangle}(X) = \left| X^{(K,\varphi)} \right| \]
for $(G,H)$-set $X$, and extending linearly.
\end{defn}

Collecting the numbers $\Phi_{\langle  K,\varphi \rangle}(X)$ for all $(G,H)$-classes $\langle  K,\varphi \rangle$, one obtains the \emph{table of marks} for $X$, so called because it determines $X$ up to isomorphism. This is recorded in the following proposition, which we use extensively throughout the paper.

\begin{prop}[\cite{Burnside}]\label{prop:PhiIsInjective}
For finite groups $G$ and $H$, the morphism
 \[ \Phi \colon A(G,H) \xrightarrow{\prod\limits_{\langle K,\varphi \rangle} \Phi_{\langle K,\varphi \rangle} } \prod_{\langle K,\varphi \rangle}  \Z\,,  \]
where the products run over all $(G,H)$-classes $\langle K,\varphi \rangle $, is injective.
\end{prop}

Again, this proposition holds with $\Z\pLoc$ or $\Zp$ coefficients, and we also use $\Phi_{\langle  K,\varphi \rangle}$ to denote the $p$-localization or $p$-completion of $\Phi_{\langle  K,\varphi \rangle}$. 

The following lemma describes the relationship between the standard basis and fixed-point methods of bookkeeping for $(G,H)$-bisets.

\begin{lem} \label{lem:FixedPtsOnBasis}
Let $G$ and $H$ be finite groups, and let $\langle K,\varphi \rangle$ and $\langle L,\psi \rangle$ be $(G,H)$-classes. Then
\[ \Phi_{\langle K,\varphi \rangle}([L,\psi]) = \frac{|N_{\varphi,\psi}|}{|L|} \cdot |C_H(\varphi(K))|\,, \]
where 
\[ N_{\varphi,\psi} = \{ x \in N_G(K,L) \mid \exists y \in H : c_y \circ \varphi = \psi \circ c_x  \}\,.\]
In particular, $\Phi_{\langle K,\varphi \rangle}([L,\psi]) = 0$ unless $\langle K,\varphi \rangle$ is subconjugate to $\langle L,\psi \rangle$.
\end{lem}
\begin{proof}
We count the pairs $(x,y)\in G\times H$ whose class in $H \times_{(L,\psi)} G$ is preserved by the action of every element in $\Delta(K,\varphi)$. 
They are such that for every $k\in K$ there exists $l\in L$ with $(\varphi(k)y,xk^{-1})=(y\psi(l),l^{-1}x)$. This gives $x\in N_{\varphi,\psi}$. Once 
$x\in N_{\varphi,\psi}$ is fixed, $y$ is determined up to an element in $C_H(\varphi(K))$. Thus there are 
$|N_{\varphi,\psi}| \cdot |C_H(\varphi(K))|$ pairs $(x,y)\in G\times H$ whose class in $H \times_{(L,\psi)} G$ is preserved by the action of every element in 
$\Delta(K,\varphi)$. Now, $\Delta(L,\psi)$ acts freely on these pairs by $(\psi(l),l)\cdot(y,x)=(y\psi(l),l^{-1}x)$ and any orbit of the action is an equivalence class in $H \times_{(L,\psi)} G$. The result follows.
\end{proof}

Note in particular that $N_\varphi = N_{\varphi,\varphi}$. We shall need the following observation later.

\begin{lem}\label{lem:NisBiset}
Let $G$ and $H$ be finite groups, and let $\langle K,\varphi \rangle$ and $\langle L,\psi \rangle$ be $(G,H)$-classes. Then $N_{\varphi,\psi}$ is a left-free $(N_\varphi,N_\psi)$-biset.
\end{lem}
\begin{proof} One only need to show that the left $N_\psi$-multiplication and the right $N_\varphi$-multiplication define group actions on  $N_{\varphi,\psi}$. Left-freeness of $N_{\varphi,\psi}$ follows from the left-freeness of $[L,\psi]$. Take $x\in N_{\varphi,\psi}$ and $u\in N_\psi$ . By definition, there exist $y\in H$ and $v\in N_H(\psi(L))$ such that $c_y \circ \varphi = \psi \circ c_x$ and $c_v \circ \psi = \psi \circ c_u$. But then 
$$c_{vy} \circ \varphi = c_v \circ c_y \circ \varphi = c_v \circ \psi \circ c_x=\psi \circ c_u \circ c_x =\psi\circ c_{ux}$$
and, hence, $ux\in N_{\varphi,\psi}$. This shows that the left $N_\psi$-multiplication on $N_{\varphi,\psi}$ is a group action. The right $N_\varphi$-multiplication on  $N_{\varphi,\psi}$ is treated similarly.
\end{proof}

\subsection{The Burnside category and Mackey functors}
There is a composition pairing
\[ A(H,K) \times A(G,H) \to A(G,K)\,, \]
induced on isomorphism classes of bisets by
\[ [X] \circ [Y] = [X \times_H Y] \,. \]
This can be described on basis elements via the double coset formula
\[ [A,\varphi]_H^K \circ [B,\psi]_G^H = \DCF{A}{\varphi}{B}{\psi}{H}_G^K \, .\]
The composition pairing is associative and bilinear, prompting us to make the following definition.

\begin{defn}
The \emph{Burnside category} $\Burnside$ is the category whose objects are the finite groups and whose morphism sets are given by 
\[ \Mor_{\Burnside}(G,H) = A(G,H), \]
with composition given by the composition pairing $\circ$. 
\end{defn}

The Burnside category is a $\Z$-linear category, in the sense that morphism sets are $\Z$-modules and composition is bilinear. Given a commutative ring $R$, we obtain an $R$-linear category $\Burnside R$ by tensoring every morphism set in $\Burnside$ with $R$. We are primarily interested in the cases $\Burnside \Z\pLoc$ and $\Burnside\Zp$, which we call the \emph{$p$-local} and \emph{$p$-complete} Burnside categories, respectively. 

\begin{defn}
Let $R$ be a commutative ring. A \emph{globally-defined, $R$-linear Mackey functor} is a functor $M \colon \Burnside \to R\Modules$, where $R\Modules$ is the category of $R$-modules. The functor can be either covariant or contravariant.
\end{defn}
Notice that an $R$-linear Mackey functor $M$ extends uniquely to a functor $\Burnside R \to R\Modules$, which we also denote by $M$. When $R$ is a $p$-local ring, we say that an $R$-linear Mackey functor is \emph{$p$-local}. In this case a functor $M \colon \Burnside \to R\Modules$ extends uniquely to a functor $M \colon \Burnside\Z\pLoc \to R\Modules$, and we often think of a $p$-local Mackey functor as a functor defined on $\Burnside\Z\pLoc$. Similarly, when $R$ is a $p$-complete ring, we say that an $R$-linear Mackey functor is \emph{$p$-complete}, and a $p$-complete Mackey functor is equivalent to a functor defined on $\Burnside\Zp$.

A Mackey functor, as defined here, is a functor $M$ defined on finite groups which, in the contravariant setting, allows restriction along group homomorphisms and transfer (induction) along inclusions of subgroups, with conjugation in a group $G$ acting trivially on $M(G)$. Standard examples are the group cohomology functor $H^*(-;R)$, and the functor $A(-,G)\otimes R$ for a fixed group $G$. More generally, any generalized cohomology theory $E^*$ gives rise to a Mackey functor $G \mapsto E^*(BG)$.

The term ``globally defined,'' here means that the functor is defined for all finite groups. (Authors such as Bouc, Th\'evenaz and Webb have studied more broadly defined biset functors, defined on a larger category where morphism sets arise from bisets with no condition of left-freeness.) We will also have occasion to consider more restrictive Mackey functors.

\begin{defn} 
Let $\Burnside_p $ be the full subcategory of $\Burnside $ whose objects are the finite $p$-groups. For a commutative ring $R$, a \emph{$p$-defined, $R$-linear Mackey functor} is a functor $M \colon \Burnside_p \to R\Modules$.
\end{defn}

The ``classical notion'' of Mackey functor for a group $G$ is a pair of functors $(M_*,M^*)$, one covariant the other contravariant, that are defined on the category of finite $G$-sets and satisfy $M_*(X) = M^*(X)$. The functors are required to be additive with respect to disjoint union and to satisfy a ``pullback condition'' that corresponds to the double coset formula. Such a functor can be seen to be equivalent to a functor $M$ defined on the category $\Burnside_G$ whose objects are the subgroups of $G$, and with morphism sets $A_G(H,K)$ the submodules of $A(H,K)$ generated by basis elements of the form $[L,c_g]_H^K$ for $g \in N_G(H,K)$. In other words, $M$ only allows restriction along conjugation in $G$. The analogous construction for fusion systems is the following.

\begin{defn} \label{def:FMackey}
Let $\F$ be a fusion system on a finite $p$-group $S$. For subgroups $P,Q \leq S$, let $A_\F(P,Q)$ be the submodule of $A(P,Q)$ generated by basis elements of the form $[T,\varphi]$ with $\varphi \in \Hom_\F(P,Q)$. An element $X$ in $A(P,Q)$ is \emph{$\F$-generated} if $X \in A_\F(P,Q)$. Let $\Burnside_{\F}$ be the subcategory whose objects are the subgroups of $S$, and whose morphism sets are the modules  $A_{\F}(P,Q)$. 
\end{defn}

To make sense in the context of a fusion system $\F$, a Mackey functor should at least be defined on the category $\Burnside_\F$. 

\begin{defn}
Let $\F$ be a fusion system on a finite $p$-group $S$. For a commutative ring $R$, an \emph{$\F$-defined} Mackey functor is a functor $\Burnside' \to R\Modules$ defined on a subcategory $\Burnside'$ of $\Burnside$ containing $\Burnside_\F$.
\end{defn}

\subsection{Augmentation}
A particular Mackey functor that will be used throughout the paper is the \emph{augmentation functor} $\countS$, which we describe here. 

\begin{defn}
The \emph{augmentation} $\countS \colon A(G,H) \to \Z$ is the homomorphism defined on the isomorphism class of a biset $X$ by \[ \countS([X]) = |G\backslash X| = |X|/|G| .\] 
\end{defn}

The augmentation homomorphism admits a convenient description on basis elements.

\begin{lem}
The augmentation of a basis element $[K,\varphi]$ of $A(G,H)$ is given by 
\[ \countS([K,\varphi]) = |G|/|K| .\]
\end{lem}

It is not hard to show that the augmentation homomorphisms satisfy 
 \[\countS(X \circ Y) = \countS(X) \cdot \countS(Y), \]
and hence we can regard them as the components of a Mackey functor
 \[ \countS \colon \Burnside \to \Z, \]
where $\Z$ is regarded as a $\Z$-linear category with a single object and morphism set $\Z$, composition being given by multiplication. 
We also use the symbol $\countS$ to denote the $p$-localization or $p$-completion of $\countS$.

\subsection{Opposite sets}
Given a $(G,H)$-biset $X$, one obtains an $(H,G)$-biset $\op{X}$ with the same underlying set by reversing the $G$ and $H$ actions. That is, if $x \in X$ and $\op{x}$ is the same element, now regarded as an element of $\op{X}$, then $g\op{x}h \defeq h^{-1}xg^{-1}$. Notice that $\widehat{\op{X}}$ is obtained by regarding the $(H\times G)$-set $\widehat{X}$ as a $(G\times H)$-set in the obvious way.

Taking opposite sets sends left-free bisets to right-free bisets and vice-versa. Consequently it does not induce a map of Burnside modules. However, the opposite of a bi-free biset is again bi-free, prompting the following definition.

\begin{defn} 
The \emph{free Burnside module} of $G$ and $H$ is the submodule $\Afree(G,H)$ of $A(G,H)$ generated by isomorphism classes of bi-free $(G,H)$-bisets. The \emph{opposite homomorphism} is the homomorphism
 \[ \opmap \colon \Afree(G,H) \longrightarrow \Afree(H,G) \]
that sends the isomorphism class of a bi-free $(G,H)$-biset to the isomorphism class of its opposite.
\end{defn}

The fixed points of the opposite homomorphism play an important role in this paper.

\begin{defn}
An element $X \in \Afree(G,G)$ is \emph{symmetric} if $\op{X} = X$.
\end{defn}

The standard basis of $A(G,H)$ restricts to a basis of $\Afree(G,H)$. We outline the steps showing this, leaving proofs to the reader.

\begin{defn}
A $(G,H)$-pair $(K,\varphi)$ and its $(G,H)$-class $\langle K,\varphi \rangle$ are \emph{free} if $\varphi$ is injective.
\end{defn}

\begin{lem}
A $(G,H)$-pair $(K,\varphi)$ is free if and only if the corresponding biset $[K,\varphi]$
is bi-free.
\end{lem}

\begin{cor}
The free Burnside module $A(G,H)$ is a free $\Z$-module with one basis element $[K,\varphi]$ for each free $(G,H)$-class $\langle K,\varphi \rangle$.
\end{cor}

Having found a basis for $\Afree(G,H)$, the next step is to describe the effect of the opposite homomorphism on basis elements and fixed-point morphisms.

\begin{lem} \label{lem:OpDesc}
For a free $(G,H)$-pair $(K,\varphi)$,
 \[ \op{\left([K,\varphi]_G^H\right)} = [\varphi(K),\varphi^{-1}]_H^G \,.\]
\end{lem}
\begin{proof}
We have $\widehat{[K,\varphi]} \cong (H \times G)/\Delta(K,\varphi)$ and $\widehat{[\varphi(K),\varphi^{-1}]} \cong (G\times H)/\Delta(\varphi(K),\varphi^{-1})$. The result follows from the fact that $\Delta(K,\varphi)$ is sent to $\Delta(\varphi(K),\varphi^{-1})$ under the natural isomorphism $H\times G \cong G\times H$.
\end{proof}

\begin{lem} \label{lem:Phi(op(X))}
For $X \in \Afree(G,H)$ and a free $(G,H)$-pair $(K,\varphi)$,
 \[ \Phi_{\langle K,\varphi \rangle}( X) = \Phi_{\langle \varphi(K),\varphi^{-1} \rangle}(\op{X})  \, . \]
\end{lem}
\begin{proof}
It is enough to prove this when $X$ is a biset. In this case it follows straight from definition that the underlying fixed-point sets $X^{(K,\varphi)}$ and $(\op{X})^{(\varphi(K),\varphi^{-1}) }$ are the same.
\end{proof}

Finally we record the behavior of the opposite morphism with respect to composition.

\begin{lem}
If $X \in \Afree(G,H)$ and $Y \in \Afree(H,K)$, then 
\[ \op{(Y \circ X)} = \op{X} \circ \op{Y}. \]
\end{lem}

\begin{cor} \label{lem:op(X)XSym}
If $X \in \Afree(G,H)$, then $\op{X} \circ X \in \Afree(G,G)$ and $X \circ \op{X} \in \Afree(H,H)$ are symmetric.
\end{cor}

The entire discussion of opposite sets and free modules carries over to the $p$-local setting.

%%%%%%%%%%%%%%%%%%%%%%%%%%%%%%%%%%%%%%%%%%%%%%%
\section{Characteristic bisets and idempotents} \label{sec:Char}
%%%%%%%%%%%%%%%%%%%%%%%%%%%%%%%%%%%%%%%%%%%%%%%

In this section we recall the notion of characteristic elements for fusion systems, which play a central role in this paper, and list some of their important properties. Characteristic elements were introduced by Linckelmann--Webb in order to produce a transfer theory for the cohomology of fusion systems, and subsequently to construct classifying spectra for fusion systems. Their definition was motivated by the special role that the $(S,S)$-biset $G$ plays in the $\Fp$-cohomology of a Sylow inclusion $S \leq G$. ($G$ acts on $H^*(BS;\Fp)$ as cohomology is a Mackey functor.) Linckelmann--Webb distilled the important properties of the biset $G$ in this context, and generalized them to fusion systems. A characteristic biset for a fusion system $\F$ on $S$ is a $(S,S)$-biset with certain properties that mimic the properties of the biset $G$ in the group case, and a characteristic element is the generalization to arbitrary elements in $A(S,S)$. We begin this section by an account of the motivating example, and then move on to discuss the generalization to fusion systems. 

\subsection{Motivation}
Given a finite group $G$ with Sylow $p$-subgroup $S$, the restriction map $H^*(BG;\Fp) \xrightarrow{i^*} H^*(BS;\Fp)$ is a monomorphism with image the $G$-stable elements in $S$. That is, the elements $x \in H^*(BS;\Fp)$ such that for every subgroup $P$ of $S$, restricting $x$ to $H^*(BP;\Fp)$ along a conjugation in $G$ has the same effect as restricting along the inclusion.  The transfer map $ H^*(BS;\Fp) \xrightarrow{tr}  H^*(BG;\Fp)$ provides a right inverse, up to scalar, to $i^*$. More precisely, the composite $tr \circ i^*$ is multiplication by the index $|G:S|$ on $H^*(BG;\Fp)$, in particular an automorphism. This implies that the composite $i^* \circ tr$ is idempotent up to scalar on $H^*(BS;\Fp)$, with image isomorphic to $H^*(BG;\Fp)$. 

Now, $H^*(B(-);\Fp)$ is a globally-defined Mackey functor and so $H^*(BS;\Fp)$ admits an $A(S,S)$-action. Under this action the class of the $(S,S)$-biset $G$ acts by $i^* \circ tr$. We can now identify $H^*(BG;\Fp)$ with the image of $H^*([G]) \colon H^*(BS;\Fp) \to H^*(BS;\Fp)$, which consists of the $G$-stable elements in $H^*(BS;\Fp)$. Under this identification, the restriction map $H^*(BG;\Fp) \to H^*(BS;\Fp)$ corresponds to the inclusion ${\rm Im}(H^*([G])) \hookrightarrow H^*(BS;\Fp)$, and the transfer $H^*(BS;\Fp) \to H^*(BG;\Fp)$ corresponds to the map $H^*([G]) \colon H^*(BS;\Fp) \to  {\rm Im}(H^*([G]))$. 

The key point here is that one can approach the $\Fp$-cohomology of $BG$ without knowing $G$ itself. The isomorphism class of $G$ is determined by $G$-stability, which depends on the fusion system, and the transfer theory can be recovered from $[G]$. The notion of $G$-stability generalizes readily to fusion systems, so one can apply this approach to the cohomology of fusion systems if one has an appropriate replacement for $[G]$ in the fusion setting. Linckelmann--Webb analyzed the properties of $[G]$ that give the desired effect in $\Fp$-cohomology, and formulated them in terms of fusion systems, leading to the definition of characteristic elements.

\subsection{Characteristic elements} To state the definition of a characteristic element, we need to define the notion of $\F$-stability in the double Burnside ring.
\begin{defn}
Let $\F$ be a fusion system on a finite $p$-group $S$. We say that an element 
$X$ in $A(S,S)\pLoc$ is \emph{right $\F$-stable} if for every $P\leq S$ and every $\varphi \in \Hom_{\F}(P,S)$, the following equation holds in $A(P,S)\pLoc$,
 \[ X \circ [P,\varphi]_P^S = X \circ [P,\incl]_P^S \,. \] 
Similarly, $X$ is \emph{left $\F$-stable} if $P\leq S$ and every $\varphi \in \Hom_{\F}(P,S)$, the following equation holds in $A(S,P)_{(p)}$,
 \[ [\varphi(P),\varphi^{-1}]_S^P \circ X = [P,1]_S^P \circ X\,.  \]
We say that $X$ is \emph{fully $\F$-stable} if it is both left and right $\F$-stable. 
% We say that $X$ is \emph{$\Ff$-generated} if $c_{\langle Q,\varphi\rangle}(X)\ne 0$ implies that $\varphi$ is a morphism in $\Hom_\Ff(Q,S)$.
\end{defn}

\begin{defn} \label{def:Char}
Let $\F$ be a fusion system on a finite $p$-group $S$. We say that an element 
$\Omega$ in $A(S,S)_{(p)}$ is a \emph{right (resp.~left, fully) characteristic element} for $\F$ if
it satisfies the following three conditions.
\begin{enumerate}
  \item[(a)] $\Omega$ is $\Ff$-generated (see Definition \ref{def:FMackey}).
  \item[(b)] $\Omega$ is right (resp.~left, fully) $\F$-stable.
  \item[(c)] $\countS(\Omega)$ is prime to $p$.
\end{enumerate}
We refer to conditions (a), (b) and (c) as the \emph{Linckelmann--Webb} properties.
\end{defn}

Observe that an element is fully characteristic for $\F$ if and only if it is both left and right characteristic. Furthermore, an element $\Omega$ is right characteristic if and only if $\op{\Omega}$ is left characteristic. We will usually drop the prefix ``fully'' and take ``characteristic element'' to mean ``fully characteristic element''. In practice there is usually no loss of generality in considering only fully characteristic elements, because of the following lemma.

\begin{lem} Let $\F$ be a fusion system on a finite $p$-group $S$. If $\Omega$ is a right characteristic element for $\F$, then $\op{\Omega} \circ \Omega$ is a symmetric, fully characteristic element for $\F$. If $\Omega$ is a left characteristic element for $\F$, then $\Omega \circ \op{\Omega}$ is a symmetric, fully characteristic element for $\F$.
\end{lem}
\begin{proof} The symmetry, augmentation and stability conditions are clear, and $\F$-generation is an easy consequence of the double coset formula.
\end{proof}

The existence of characteristic elements was established by Broto--Levi--Oliver in \cite{BLO2}. In fact, Broto--Levi--Oliver proved a slightly stronger result, showing that every saturated fusion system has a \emph{characteristic biset}, meaning a characteristic element that is the isomorphism class of an actual biset.

\begin{thm}[\cite{BLO2}]\label{thm:CharExists}
Every saturated fusion system has a characteristic biset, in particular a characteristic element.
\end{thm}

\subsection{Characteristic idempotents}
Linckelmann--Webb showed that a right characteristic element for $\F$ induces a selfmap of $H^*(BS;\Fp)$ that is idempotent up to scalar with image the $\F$-stable elements of $H^*(BS;\Fp)$. (A proof can be found in \cite{BLO2}.) Thus characteristic elements are appropriate for defining transfer in the $\Fp$-cohomology of fusion systems. However, if one tries to replace $\Fp$-cohomology with another Mackey functor, this is not so simple as a characteristic element will not act by an idempotent in general. Moreover, a given saturated fusion system has infinitely many characteristic elements, which can give rise to different transfer constructions. Both of these problems were circumvented in \cite{KR:ClSpectra} by introducing characteristic idempotents.

\begin{defn}
Let $\F$ be a fusion system on a finite $p$-group $S$. A \emph{characteristic idempotent} for $\F$ is a characteristic element for $\F$ that is idempotent.
\end{defn}

\begin{thm}[\cite{KR:ClSpectra}]\label{thm:CharIdemExists}
A saturated fusion system $\F$ on a finite $p$-group $S$ has a \emph{unique} characteristic idempotent $\omega_\F \in A(S,S)\pLoc$. Furthermore, $\omega_\F$ is symmetric.
\end{thm}
\begin{proof} The existence and uniqueness of a characteristic idempotent is shown in \cite[Propositions 4.9 and 5.6]{KR:ClSpectra} (see also \cite[Remark 5.8]{KR:ClSpectra}). Since $\op{\omega_\F}$ is also a characteristic idempotent for $\F$, uniqueness implies that $\op{\omega_\F} = \omega_\F$. 
\end{proof}

\begin{rmk}
More precisely, \cite[Remark 5.8]{KR:ClSpectra} says that the characteristic idempotent is the unique idempotent right characteristic element and the unique idempotent left characteristic element. Thus a right or left characteristic element that is idempotent is automatically fully characteristic. This is why we only talk of characteristic idempotents, without a left or right qualifier. 
\end{rmk}

\subsection{Fixed points of characteristic element}

We will prove our main theorems by carefully analyzing and keeping track of fixed points of characteristic elements. Therefore it is important to reformulate the properties of characteristic elements in terms of fusion systems, as we do in the following lemma. 

\begin{lem} \label{lem:RephraseStable}
Let $\F$ be a fusion system on a finite $p$-group $S$, and let $X$ be an element in $A(S,S)_{(p)}$. 
\begin{itemize}
 \item[(a)] $X$ is $\Ff$-generated if and only if $\Phi_{\langle Q,\psi \rangle}(X) = 0$
 for all $(S,S)$-classes $\langle Q,\psi \rangle$ where $\psi$ is not in $\F$.
 \item[(b)] $X$ is right $\F$-stable if and only if for every $(S,S)$-class $\langle Q,\psi \rangle$, and every $\varphi \in \Hom_{\F}(Q,S)$,
  \[ \Phi_{\langle Q,\psi \rangle} (X) = \Phi_{\langle \varphi(Q),\psi \circ \varphi^{-1} \rangle} (X)\,. \]
 \item[(c)] $X$ is left $\F$-stable if and only if for every $(S,S)$-class $\langle Q,\psi \rangle$, and every $\varphi \in \Hom_{\F}(\psi(Q),S)$,
  \[ \Phi_{\langle Q,\psi \rangle} (X) = \Phi_{\langle Q, \varphi \circ \psi \rangle} (X)\,. \]
\end{itemize}
\end{lem}
\begin{proof}
$ $
\begin{itemize}
 \item[(a)] 
 Assume $X$ is $\F$-generated. If $\Phi_{\langle Q,\psi \rangle}(X) \ne 0$, then Lemma \ref{lem:FixedPtsOnBasis} implies that $\langle Q,\psi \rangle$ is subconjugate to a $(S,S)$-class $ \langle P,\varphi \rangle$ with $c_{\langle P,\varphi \rangle}(X)\ne 0$. By $\F$-generation this implies that $\varphi\in\Hom_\Ff(P,S)$, and since $\Ff$ is closed under restriction and conjugation, it follows that $\psi\in\Hom_\Ff(Q,S)$.

If $X$ is not $\F$-generated, then let $\langle Q,\psi \rangle$ be maximal, with respect to subconjugacy, among $(S,S)$-classes with $c_{\langle Q,\psi \rangle}(X)\ne 0$ and $\psi \notin \Hom_\Ff(Q,S)$. Lemma \ref{lem:FixedPtsOnBasis} then implies that
\[ \Phi_{\langle Q,\psi \rangle} (X) = c_{\langle Q,\psi \rangle} \cdot \Phi_{\langle Q,\psi \rangle} ([Q,\psi]), \]
which is nonzero.
 
 \item[(b)] First consider the case where $X$ is a biset. Then, for any $Q\le P\le S$, $\varphi\in\Inj(P,S)$ and $\psi\in\Inj(Q,S)$ there are canonical bijections between the fixed-point sets $(X \circ [P,\varphi]_P^S)^{(Q,\psi)} $ and $X^{(\varphi(Q),\psi\circ\varphi^{-1})}$, and also between $X^{(Q,\psi)}$ and $(X\circ [P,\incl]_P^S)^{(Q,\psi)} $. This results in equations of fixed-point homomorphisms, 
 \[ \Phi_{\langle Q,\psi \rangle} (X\circ [P,\varphi]) = \Phi_{\langle \varphi(Q),\psi\circ\varphi^{-1} \rangle} (X),\] 
 and 
\[ \Phi_{\langle Q,\psi \rangle} (X\circ [P,\incl]) = 
   \Phi_{\langle Q,\psi \rangle} (X),\] 
valid for any $X \in A(S,S)\pLoc$. Consequently, if X is right $\Ff$-stable we get an equality $\Phi_{\langle Q,\psi \rangle} (X) = \Phi_{\langle \varphi(Q),\psi \circ \varphi^{-1} \rangle} (X)$ for all $\varphi\in\Hom_\Ff(Q,S)$ and $\psi\in\Inj(Q,S)$. 

Conversely, if for a fixed $\varphi\in\Hom_\Ff(P,S)$ and all pairs $Q\le P$ and $\psi\in\Inj(Q,S)$ we have $\Phi_{\langle Q,\psi \rangle} (X) = \Phi_{\langle \varphi(Q),\psi \circ \varphi^{-1} \rangle} (X)$, then the fixed-points map 
$$\Phi\colon A(P,S)_{(p)}\xrightarrow{\prod\limits_{\langle Q,\psi \rangle}\Phi_{\langle Q,\psi \rangle}}\prod_{\langle Q,\psi \rangle}\Z_{(p)} \, ,$$ 
where the product runs over all $(P,S)$-classes $\langle Q,\psi \rangle $, have the same image at $X\circ [P,\varphi]$ and $X\circ [P,\incl]$. 
Hence, by Proposition \ref{prop:PhiIsInjective}, $X\circ [P,\varphi]=X\circ [P,\incl]$.
 
 \item[(c)] Analogous to (b).
\end{itemize}
\end{proof}

Among characteristic elements, idempotents can be formulated in terms their standard basis representation.

\begin{defn}
For an element $X$ in $A(S,S)\pLoc$, and a subgroup $P \leq S$, set
\[ m_P(X) = \sum_{[\varphi] \in \Rep_\F(P,S)} c_{\langle P,\varphi \rangle}(X)\,, \]
where the sum runs over representatives of conjugacy classes of group homomorphisms. 
\end{defn}

\begin{lem}[\cite{KR:ClSpectra}] \label{lem:CoeffsOfIdem}
Let $\F$ be a saturated fusion system on a finite $p$-group $S$, and let $\Omega$ be a characteristic
element for $\F$. 
\begin{enumerate}
\item[(a)] $\Omega$ is idempotent if and only if $m_S(\Omega) = 1$, and $m_P(\Omega) = 0$ for $P < S$.
\item[(b)] $\Omega$ is idempotent $\modp$ if and only if $m_S(\Omega) \equiv 1~ \modp$, and $m_P(\Omega) \equiv 0~ \modp$  for $P < S$.
\end{enumerate}
\end{lem}
\begin{proof}
The ``only if'' implication in part (a) is shown in \cite[Lemma 5.5]{KR:ClSpectra}. Conversely,
if $\Omega$ is a characteristic element for $\F$ satisfying $m_S(\Omega) = 1$, and $m_P(\Omega) = 0$ for $P <S$,
we have
\begin{eqnarray*} 
    \Omega \circ \Omega 
  &=& \Omega \circ \left( \sum_{[P]_S } \left( \sum_{[\varphi] \in \Rep_\F(P,S)} c_{\langle P,\varphi \rangle}(\Omega) \; [P,\varphi] \right)\right)\\
  &=& \sum_{[P]_S } \left( \sum_{[\varphi] \in \Rep_\F(P,S)} c_{\langle P,\varphi \rangle}(\Omega) \; (\Omega \circ [P,\varphi]) \right)\\
  &=& \sum_{[P]_S } \left(\sum_{[\varphi] \in \Rep_\F(P,S)} c_{\langle P,\varphi \rangle}(\Omega) \right) (\Omega\circ [P,\incl])\\
  &=& \sum_{[P]_S } m_P(\Omega) \; (\Omega \circ [P,\incl])\\
  &=& \Omega,
\end{eqnarray*}
so $\Omega$ is idempotent. Applying the same arguments $\modp$ proves part (a). 
\end{proof}

\subsection{The universal stable element property}
We have previously discussed $\F$-stability in $\Fp$-cohomology and in the double Burnside ring. The notion generalizes readily to any Mackey functor, and we will show below that characteristic idempotents characterize $\F$-stable elements.

\begin{defn}
Let $\F$ be a fusion system on a finite $p$-group $S$, and let $M$ be a  (covariant or contravariant) $\F$-defined Mackey functor. An element $x \in M(S)$ is \emph{$\F$-stable} if for every $P \leq S$ and every $\varphi \in \Hom_\F(P,S)$ we have 
 \[ M([P,\varphi]_P^S)(x) = M([P,\incl]_P^S)(x) \in M(P)  \]
in the contravariant case, or 
 \[ M([\varphi(P),\varphi^{-1}]_S^P)(x) = M([P,\id]_S^P)(x) \in M(P),  \]
in the covariant case. 
In either case we denote by $M(\F)$ the set of $\F$-stable elements in $M(S)$.
\end{defn}

A key property of characteristic idempotents is following theorem, which is implicit in \cite{KR:ClSpectra}. The theorem can be interpreted as saying that the characteristic idempotent of a saturated fusion system $\F$ is a universal $\F$-stable element for Mackey functors.

\begin{thm}[Universal stable element theorem] \label{thm:UnivStable}
Let $\F$ be a saturated fusion system on a finite $p$-group $S$, and let $M$ be a (covariant or contravariant) $p$-local, $\F$-defined Mackey functor.
Then an element $x \in M(S)$ is $\F$-stable if and only if $M(\omega_\F)(x) = x$.
\end{thm}
\begin{proof}
One direction is easy: If $M(\omega_\F)(x) = x$, then $\F$-stability of $\omega_\F$ implies the $\F$-stability of $x$. We prove the converse only for contravariant $M$, the covariant case being analogous. Now, if $x$ is $\F$-stable, then we have
\begin{eqnarray*}
  M(\omega_\F) (x) &= &  \sum_{[P]_S } \left( \sum_{[\varphi] \in \Rep_\F(P,S)} c_{\langle P,\varphi \rangle}(\omega_\F) \; M([P,\varphi]) (x) \right) \\
     &= &  \sum_{[P]_S } \left( \sum_{[\varphi] \in \Rep_\F(P,S)} c_{\langle P,\varphi \rangle}(\omega_\F) \right)  M([P,\incl]) (x)  \\
     &=& \sum_{[P]_S } m_P(\omega_\F) \; M([P,\incl])(x).
\end{eqnarray*}
By Lemma \ref{lem:CoeffsOfIdem} we have $m_S = 1$, and $m_P = 0$ for $P < S$, so 
\[ M(\omega_\F) (x) = M([S,\id])(x) = x.\]
\end{proof}

The same argument proves a universal stable element theorem for (left or right) $A(S,S)\pLoc$-modules, where one defines $\F$-stability for $x \in M$ by demanding that for all $P \leq S$ and $\varphi \in \Hom_\F(P,S)$, one has
$[P,\varphi]_S^S \cdot x = [P,\incl]_S^S \cdot x$ (left modules) or $x \cdot [P,\varphi]_S^S = x \cdot [P,\incl]_S^S $ (right modules), as appropriate.

\begin{thm}\label{thm:UnivStableII} Let $\F$ be a saturated fusion systems on a finite $p$-group $S$, and let $M$ be a left (resp.~right) $A(S,S)\pLoc$-module. Then an element $x \in M$ is $\F$-stable if and only if $\omega_\F \cdot x = x$ (resp.~$x \cdot \omega_\F = x$).
\end{thm} 

%%%%%%%%%%%%%%%%%%%%%%%%%%%%%%%%%%%%%%%%%%%%%%%%%%%%%%%%%%%%%%%%%%
\section{Fusion systems induced by bisets} \label{sec:BisetFusion}
%%%%%%%%%%%%%%%%%%%%%%%%%%%%%%%%%%%%%%%%%%%%%%%%%%%%%%%%%%%%%%%%%%
In this section we introduce certain fusion systems on $S$ induced by $(S,S)$-bisets and, more generally, elements in $A(S,S)\pLoc$. These fusion systems are the stabilizer fusion system, fixed-point fusion system, and orbit-type fusion system. The latter two are initially defined at the level of pre-fusion systems, and their fusion closures are always equal (although the pre-fusion systems generally are not). When applied to a characteristic element for a fusion system $\F$, each of the three fusion systems  constructions recover $\F$. Thus characteristic elements contain exactly the same information as their fusion systems, encoded in the double Burnside ring. Furthermore, equality of the three fusion systems holds only for characteristic elements, giving us a criterion to recognize characteristic elements from the fusion systems they induce. As we will show in the next section, the existence of a characteristic element implies saturation for a fusion system, and thus this criterion can also be thought of as a saturation criterion.

\subsection{Stabilizer fusion systems}
The stabilizer fusion system of an element is the largest fusion system that stabilizes the element. Stabilizer fusion systems come in three flavours, depending whether one is looking at right stability, left stability, or both. Formally, they are defined as follows.

\begin{defn} \label{def:Stab}
Let $S$ be a finite $p$-group, and let $X$ be an element in $A(S,S)_{(p)}$.
\begin{enumerate}
\item[(a)]
The \emph{right stabilizer fusion system} of $X$ is the fusion system $\FusRStab(X) $ with morphism sets
  \[ \Hom_{\FusRStab(X)}(P,Q) = \{ \varphi \in \Inj(P,Q) \mid X \circ [P,\varphi]_P^S = X \circ [P,\incl]_P^S \}\,.  \]
\item[(b)]
The \emph{left stabilizer fusion system} of $X$ is the fusion system $\FusLStab(X) $ with morphism sets
  \[ \Hom_{\FusLStab(X)}(P,Q)= \{ \varphi \in \Inj(P,Q) \mid [\varphi(P),\varphi^{-1}]_S^P \circ X = [P,1]_S^P \circ X \}\,.  \]  
\item[(c)]
The \emph{full stabilizer fusion system} of $X$ is the intersection 
\[ \FusStab(X) = \FusRStab(X) \cap \FusLStab(X)\,.\] 
\end{enumerate}
\end{defn}

We leave it to the reader to verify that the three stabilizer fusion systems are indeed fusion systems. As the following lemma shows, the three stabilizer fusion systems are related, so in practice it usually suffices to prove results for left or right stabilizer systems.

\begin{lem}\label{lem:StabOp}
Let $S$ be a finite $p$-group. For an element $X$ in $\Afree(S,S)_{(p)}$, we have
\[ \FusRStab(X) = \FusLStab(\op{X}) \quad \text{and} \quad \FusLStab(X) = \FusRStab(\op{X})\,.\] 
In particular, if $X$ is symmetric, then 
\[ \FusRStab(X) = \FusLStab(X) = \FusStab(X)\,. \] 
\end{lem}
\begin{proof}
Lemma \ref{lem:OpDesc} implies that $\FusRStab(X) = \FusLStab(\op{X})$, from which the other claims follow.
\end{proof}

As most of our arguments in the next two sections are in terms of fixed-point homomorphisms, it is helpful to record how morphisms stabilizing an element can be recognized in that context. 

\begin{lem} \label{lem:PhiConstantOnStab}
Let $S$ be a finite $p$-group and let $X$ be an element in $A(S,S)_{(p)}$.
For every subgroup $P \leq S$, the following hold 
\begin{enumerate}
 \item[(a)] If $\psi\in\Hom(P,S)$ and $\varphi \in \Hom_{\FusRStab(X)}(P,S)$, then 
  \[ \Phi_{\langle P,\psi \rangle}(X) = \Phi_{\langle \varphi(P),\psi\circ\varphi^{-1} \rangle}(X) \, .\]
 \item[(b)] If $\psi\in\Hom(P,S)$ and $\varphi \in \Hom_{\FusLStab(X)}(\psi(P),S)$, then
 \[ \Phi_{\langle P,\psi \rangle}(X) = \Phi_{\langle P,\varphi\circ\psi \rangle}(X) \, . \]
\end{enumerate}
\end{lem}
\begin{proof} As $X$ is right $\FusRStab(X)$-stable and left $\FusLStab(X)$-stable, this follows from Lemma \ref{lem:RephraseStable}.
\end{proof}

\subsection{Fixed-point and orbit-type fusion systems}

\begin{defn} \label{def:FixOrb}
Let $S$ be a finite $p$-group, and let $X$ be an element in $\Afree(S,S)_{(p)}$.
\begin{enumerate}
\item[(a)] The \emph{orbit-type pre-fusion system} of $X$ is the pre-fusion system $\PreOrb(X) $ with morphism sets
 \[ \Hom_{\PreOrb(X)}(P,Q) = \{ \varphi \in \Inj(P,Q) \mid c_{\langle P,\varphi \rangle}(X) \neq 0 \}\,. \]
The \emph{orbit-type fusion system} of $X$, denoted $\FusOrb(X)$ is the closure of $\PreOrb(X)$.
 \item[(b)] The \emph{fixed-point pre-fusion system} of $X$ is the pre-fusion system $\PreFix(X) $ with morphism sets
 \[ \Hom_{\PreFix(X)}(P,Q) = \{ \varphi \in \Inj(P,Q) \mid \Phi_{\langle P,\varphi \rangle}(X) \neq 0 \}\,. \]
The \emph{fixed-point fusion system} of $X$, denoted $\FusFix(X)$ is the closure of $\PreFix(X)$.
\end{enumerate}
\end{defn}

Although the fixed-point and orbit-type pre-fusion systems are different in general, their closures are the same, as shown in the following lemma.

\begin{lem} \label{lem:Orb=Fix}
Let $S$ be a finite $p$-group. For any $X \in \Afree(S,S)_{(p)}$, 
\[ \FusOrb(X) = \FusFix(X)\,. \]
\end{lem}
\begin{proof}
It suffices to show that $\PreOrb(X) \subseteq \FusFix(X)$, and $\PreFix(X) \subseteq \FusOrb(X)$.

As both $\FusOrb(X)$ and $\FusFix(X)$ are fusion systems they are closed under restriction and conjugation by elements of $S$. 
Thus $\psi\in\Hom_{\FusOrb(X)}(P,S)$ implies that, for every $\langle Q,\varphi \rangle \subcon \langle P,\psi \rangle$, we have $\varphi \in \Hom_{\FusOrb(X)}(Q,S)$. The analogous statement for $\FusFix(X)$ is also true.

Suppose that $\varphi\in\Hom_{\PreOrb(X)}(Q,S)$, so $c_{\langle Q,\varphi \rangle}(X)\ne 0$. Let $\langle P,\psi \rangle$ be a maximal (with respect to subconjugacy) $(S,S)$-class so that $\langle Q,\varphi \rangle$ is subconjugate to $\langle P,\psi \rangle$, and $c_{\langle P,\psi \rangle}(X)\ne 0$. By Lemma \ref{lem:FixedPtsOnBasis}, maximality implies 
$\Phi_{\langle P,\psi \rangle}(X)\ne 0$, so $\psi\in\FusFix(X)$. Using the remark in the previous paragraph we get $\varphi\in\FusFix(X)$.

Suppose now that $\varphi\in\Hom_{\PreFix(X)}(Q,S)$. Then, using Lemma \ref{lem:FixedPtsOnBasis} again, there exists an $(S,S)$-class $\langle P,\psi \rangle$ to which $\langle Q,\varphi \rangle$ is subconjugate and such that
$c_{\langle P,\psi \rangle}(X)\ne 0$. Using the remark in the first paragraph we get $\varphi\in\FusOrb(X)$.
\end{proof}

\begin{lem} \label{lem:FixOp}
Let $S$ be a finite $p$-group and let $X$ be an element in $\Afree(S,S)_{(p)}$.
For every $P \leq S$ and every monomorphism $\varphi \colon P \to S$, we have
\[ \varphi \in \Hom_{\PreFix(X)}(P,S) \quad \text{if and only if}\quad \varphi^{-1} \in \Hom_{\PreFix(\op{X})}(\varphi(P),S) \,. \]  
In particular,
 \[ \FusFix(X) = \FusFix(\op{X})\,. \]
\end{lem}
\begin{proof} The first claim is an immediate consequence of Lemma \ref{lem:Phi(op(X))}, and the second claim follows since $\FusFix(X)$ is closed under inverses.
\end{proof}

Given that the only extra property of $\FusFix(X)$ over $\PreFix(X)$ that was used in the proof of Lemma \ref{lem:FixOp} is the closure under inverses, we can give a more precise statement.  
\begin{lem} \label{lem:PreFixForLWClosed}
Let $S$ be a finite $p$-group and let $X$ be an element in $\Afree(S,S)_{(p)}$. If $\PreFix(X)$ is level-wise
closed, then $\PreFix(\op{X}) = \PreFix(X).$
\end{lem}

\subsection{Fusion systems induced by characteristic elements}
The Linckelmann--Webb properties for can be rephrased in terms of fusion systems induced by characteristic elements.

\begin{lem} \label{lem:CharRephrase}
Let $\F$ be a fusion system on a finite $p$-group $S$. An element $\Omega$ in $\Afree(S,S)_{(p)}$ is right characteristic for $\F$ if and only if $\countS(\Omega)$ is prime to $p$, and
\[ \PreOrb(\Omega) \subseteq \F \subseteq \FusRStab(\Omega)\,. \]
The analogous statement holds for left and fully characteristic elements.
\end{lem}
\begin{proof}
Using Lemma \ref{lem:Orb=Fix}, the inclusion $\PreOrb(\Omega) \subseteq \F$ is equivalent to (a) in Definition \ref{def:Char}, while the inclusion $\F \subseteq \FusRStab(\Omega)$ is equivalent to condition (b). 
\end{proof}

With significant extra work, the inclusion $\F \subseteq \FusRStab(\Omega)$ in Lemma \ref{lem:CharRephrase} can in fact be strengthened to an equality.

\begin{thm}[\cite{KR:ClSpectra}] \label{thm:StabOfChar}
Let $\F$ be a fusion system on a finite $p$-group $S$. If $\Omega$ is a right characteristic element for $\F$, then $\FusRStab(\Omega) =\F$. The analogous result holds for left characteristic and fully characteristic elements for $\F$.
\end{thm}
\begin{proof}
This was proved in \cite{KR:ClSpectra} in the case where $\Omega$ is a characteristic idempotent, and the same argument works for left, right or fully characteristic elements.
\end{proof}

\begin{prop} \label{prop:PreFixOfChar}
Let $\F$ be a fusion system on a finite $p$-group $S$. If $\Omega$ is a left or right characteristic element for $\F$, then
\[ \PreFix(\Omega) = \FusFix(\Omega) = \FusOrb(\Omega) = \F. \]
\end{prop}
\begin{proof}
We already have
\[  \PreFix(\Omega) \subseteq \FusFix(\Omega) = \FusOrb(\Omega) \subseteq \F,\]
where the first inclusion is immediate, the equality is by Lemma \ref{lem:Orb=Fix}, and the last inclusion follows from $\PreOrb(\Omega) \subseteq \F$ upon taking closures. Thus it suffices to show that $\F \subseteq \PreFix(\Omega)$. In other words we show that, for every $(S,S)$-pair $(P,\varphi)$ with $\varphi$ in $\F$, we have $\Phi_{\langle P,\varphi \rangle}(\Omega)\neq 0$. 

Consider first the left characteristic case. By Lemma \ref{lem:RephraseStable}, left $\F$-stability of $\Omega$ implies that $\Phi_{\langle P,\psi \rangle}(\Omega) = \Phi_{\langle P,\varphi\rangle}(\Omega)$ for all $\psi \in \Hom_\F(P,S)$. Hence it is enough to consider the case where $\varphi(P)$ is fully $\F$-centralized. But in this case Lemma \ref{lem:CongSpecial} shows that 
$\frac{\Phi_{\langle P,\varphi \rangle}(\Omega)}{|C_S(\varphi(P))|}$ is nonzero $\modp$. In particular, $\Phi_{\langle P,\varphi \rangle}(\Omega) \neq 0$. This proves $\F \subseteq \PreFix(\Omega)$, and the string of equalities follows.

If $\Omega$ is a right characteristic element, then $\op{\Omega}$ is left characteristic, and we have $\PreFix(\op{\Omega}) = \F$. In particular, $\PreFix(\op{\Omega})$ is levelwise closed, so $ \PreFix(\Omega) = \PreFix(\op{\Omega})$, and the result follows.
\end{proof}

Note that it is generally not true that $\PreOrb(\Omega) = \F$ for a characteristic element $\Omega$. For instance, the $(S,S)$-biset $[S,\id]$ is always a characteristic idempotent for the minimal fusion system on $S$, and $\PreOrb([S,\id])$ contains only one morphism: the identity of $S$.

The equations $\FusRStab(\Omega) = \F$ and $\PreFix(\Omega) = \F$ were independently proved by Puig in \cite{puig:book} in the case where $\Omega$ is a right characteristic biset. 

We conclude this section by observing that these results can also be applied when the fusion system is not specified. This gives a criterion for recognizing characteristic elements, and reconstructing their fusion system.

\begin{cor} \label{cor:FixStabImpliesChar} 
Let $S$ be a finite $p$-group, and let $\Omega$ be an element in $\Afree(S,S)_{(p)}$ such that $\countS(\Omega)$ is prime to $p$. If 
$\PreFix(\Omega) \subseteq \FusRStab(\Omega)$, then $\Omega$ is a right characteristic element for $\FusRStab(\Omega)$, and 
 \[ \PreFix(\Omega) = \FusFix(\Omega) = \FusOrb(\Omega) = \FusRStab(\Omega). \]
The analogous statement holds for left and fully characteristic elements.
\end{cor}
\begin{proof}
We prove this for right characteristic elements. The case of left or fully characteristic elements can be proved by a similar argument, or by taking opposite sets.

Taking closures, the inclusion $\PreFix(\Omega) \subseteq \FusRStab(\Omega)$ gives $\FusFix(\Omega) \subseteq \FusRStab(\Omega)$, and since 
 \[\PreOrb(\Omega) \subseteq \FusOrb(\Omega) = \FusFix(\Omega) \, ,\] 
this implies $\PreOrb(\Omega) \subseteq \FusRStab(\Omega)$. Taking $\F \defeq \FusRStab(\Omega)$ in Lemma \ref{lem:RephraseStable}, we deduce that $\Omega$ is a right characteristic element for $\F$. Theorem \ref{thm:StabOfChar} and Proposition \ref{prop:PreFixOfChar} then imply the stated equalities.
\end{proof}

%%%%%%%%%%%%%%%%%%%%%%%%%%%%%%%%%%%%%%%%%%%%%%%%%%%%%%%%%%%%%%%%%%%%%%%%%%%%%%
\section{Characteristic elements imply saturation}\label{sec:CharImpliesSat}%%
%%%%%%%%%%%%%%%%%%%%%%%%%%%%%%%%%%%%%%%%%%%%%%%%%%%%%%%%%%%%%%%%%%%%%%%%%%%%%%

In this section we prove our first main result, Theorem \ref{mthm:CharImpliesSat} of the introduction, which says that if a fusion has a characteristic element, then it is saturated. This result, which was independently ---and first--- proved by Puig in \cite{puig:book}, is the first substantially different formulation of the saturation property for fusion systems (other than some variant of a prime to $p$ property and an extension property for morphism sets), and also shows that saturation can be detected in the double Burnside ring.

The proof is carefully broken down into parts, in order to reuse some arguments to prove Theorem \ref{mthm:FrobImpliesSat} in Section \ref{sec:FrobImpliesSat}. In \ref{subsec:Congruences} we establish some congruences for fixed-point homomorphisms that we will use repeatedly. In \ref{subsec:LWSat} we use these congruences to prove local saturation axioms, and in \ref{subsec:ProofOfCharImpliesSat} we collect the local results to complete the proof of Theorem \ref{mthm:CharImpliesSat}.

\subsection{Congruence relations for fixed-point homomorphisms} \label{subsec:Congruences}

The key to extracting information about the stabilizer and fixed-point fusion systems induced by an element in the double Burnside ring are the congruence relations in the following lemma. 

\begin{lem}[\cite{BCGLO2}]\label{lem:Cong}
Let $S$ be a finite $p$-group, let $X \in \Afree(S,S)_{(p)}$, and put $\PreFus \defeq \PreFix(X)$.
\begin{itemize}
 \item[(a)] For each $\varphi \in \Hom_{\PreFus}(P,S)$, the number $\Phi_{\langle P,\varphi \rangle}(X)$ is divisible by $|C_S(\varphi(P))|$ (in $\Z_{(p)}$). Furthermore,
  \[ \sum_{[\varphi] \in \Rep_{\PreFus}(P,S) } \frac{\Phi_{\langle P,\varphi \rangle}(X)}{|C_S(\varphi(P))|} \equiv \countS(X) \quad \modp \,.\]
 \item[(b)] For each $Q \in [P]_\PreFus$, the number 
$ \sum\limits_{\varphi \in \Hom_{\PreFus}(P,Q)} \Phi_{\langle P,\varphi \rangle}(X) $ is divisible by 
$|N_S(\varphi(P))|$ (in $\Z_{(p)}$). Furthermore,
  \[ \sum_{[Q] \in [P]_\PreFus} \frac{\sum\limits_{\varphi \in \Hom_{\PreFus}(P,Q)} \Phi_{\langle P,\varphi \rangle}(X)}{|N_S(Q)|} \equiv \countS(X) \quad \modp\,, \]
where the sum runs over $S$-conjugacy classes of subgroups $Q \leq S$ that are $\PreFus$-images of $P$.
\end{itemize}
\end{lem}
\begin{proof}
This follows by adapting and expanding on an argument used in the proof of Proposition 1.16 in \cite{BCGLO2}. We outline that argument here, referring the reader to \cite{BCGLO2} for details, and emphasize the parts that need to be adapted.

Consider first the case where $X$ is an $(S,S)$-biset. For the sake of clarity, we shall distinguish between the left and right $S$-actions by considering $X$ as an $(S_1,S_2)$-set, with the understanding that  $S_1 = S_2 = S$. Then $S_2\backslash X$ is a right $S_1$-set, and we let $X_0 \subseteq X$ be the pre-image of $(S_2\backslash X)^P$ under the projection $X \to S_2\backslash X$, where $P$ acts on the right via the inclusion $P \leq S_1$. As explained in \cite{BCGLO2}, this means that for every $x \in X_0$, there is a group monomorphism $\theta(x) \colon P \to S_2$ such that, for all $g \in P$, we have $\theta(x)(g) x = x g$. Thus we get a map $\theta \colon X_0 \to \Inj(P,S_2)$ such that $\theta^{-1}(\varphi) = X^{(P,\varphi)},$ and we have 
 \[ |X_0| = \sum_{\varphi \in \Hom(P,S_2)} |\theta^{-1}(\varphi) | = \sum_{\varphi \in \Inj(P,S_2)} \Phi_{\langle P,\varphi \rangle}(X)\,.  \]
Furthermore, $\theta(ax) = c_a \circ \theta(x)$ for $a \in S_2$ and $x \in X_0$, so we get an induced map
 $\widetilde{\theta} \colon S_2\backslash X_0 \to \widetilde{\Inj}(P,S_2) \defeq S_2\backslash \Inj(P,S_2)$. 

Letting $I(P)$ be the set of subgroups of $S_2$ that are isomorphic to $P$, we get a map $\Image \colon \Inj(P,S_2) \to I(P)$ sending a monomorphism to its image. This induces a map $\widetilde{\Image} \colon \widetilde{\Inj}(P,S_2) \to \widetilde{I}(P)$, where $\widetilde{I}(P)$ is the set of $S_2$-conjugacy classes in $I(P)$. 

These maps all fit into a commutative diagram
\[ \xymatrix{
  X_0 \ar[rr]^{\theta \hphantom{\Inj}} \ar[d]^{q}  && \Inj(P,S_2) \ar[rr]^{\hphantom{Sj} \Image } \ar[d]^{q}  && I(P) \ar[d]^{q} \\ 
  S_2\backslash X_0  \ar[rr]^{ \widetilde{\theta} \hphantom{\Inj}} && \widetilde{\Inj}(P,S_2) \ar[rr]^{\hphantom{Sj} \widetilde{\Image}} && \widetilde{I}(P)
}
\]
where the vertical maps are the canonical projections onto $S_2$-orbits.

For each $\varphi \in \Inj(P,S_2)$, the conjugacy class $[\varphi] \in \widetilde\Inj(P,S_2)$ contains $|S|/|C_{S_2}(\varphi(P)|$ distinct monomorphisms $\varphi'$, each of which is conjugate to $\varphi$, so $\Phi_{\langle P,\varphi' \rangle}(X) = \Phi_{\langle P,\varphi \rangle}(X)$. Therefore $ |(q\circ \theta)^{-1}([\varphi])| = \frac{|S|}{|C_{S_2}(\varphi(P))|} \Phi_{\langle P,\varphi \rangle}(X)$. Since the $S_2$-action on $X_0$ is free, we obtain 
 \[ \widetilde{\theta}^{-1}([\varphi]) = \frac{ |(q\circ \theta)^{-1}([\varphi])| }{|S_2|} = \frac{\Phi_{\langle P,\varphi \rangle}(X)}{|C_{S_2}(\varphi(P))|}\,, \]
and in particular, $\Phi_{\langle P,\varphi \rangle}(X)$ is divisible by $|C_{S_2}(\varphi(P))|$. Furthermore, using the congruence 
\[ |S_2\backslash X_0| = |(S_2\backslash X)^P| \equiv |S_2 \backslash X| \quad \modp\,, \]
we obtain
\[ 
  \sum_{[\varphi] \in \widetilde{\Inj}(P,S_2) } \frac{\Phi_{\langle P,\varphi \rangle}(X)}{|C_S(\varphi(P))|} 
= \sum_{[\varphi] \in \widetilde{\Inj}(P,S_2) } |\widetilde{\theta}^{-1}([\varphi])|
= |S_2\backslash X_0| 
\equiv \countS(X) \quad \modp\,.\]
This congruence, and the divisibility property, extend to general $X \in \Afree(S_1,S_2)_{(p)}$ by linearity of the morphisms $\countS$ and $\Phi_{\langle P,\varphi \rangle}$. Part (a) now follows by observing that we need only sum over $[\varphi] \in \Rep_{\PreFus}(P,S),$ as $\Phi_{\langle P,\varphi \rangle}(X_0)=0$ when $\varphi \notin \Hom_{\PreFus}(P,S)$. 

Part (b) is proved similarly by first showing that, when $X$ is a biset,
\[ |(\widetilde{\Image} \circ \widetilde{\theta})^{-1}([Q])| = \frac{\sum\limits_{\varphi \in \Hom_{\PreFus}(P,Q)} \Phi_{\langle P,\varphi \rangle}(X)}{|N_S(Q)|} \]
for each $[Q] \in \widetilde{I}(P)$, then summing over $[Q] \in \widetilde{I}(P)$ and extending the resulting congruence to general $X \in \Afree(S_1,S_2)_{(p)}$ by linearity.
\end{proof}

\subsection{Levelwise saturation results} \label{subsec:LWSat}

Adding a stability condition, we get the following characterization of fully centralized and fully normalized subgroups.

\begin{lem} \label{lem:CongSpecial}
Let $S$ be a finite $p$-group, and let $X$ be an element in $\Afree(S,S)_{(p)}$ with $\countS(X)$ not divisible by $p$, such that $\PreFus \defeq \PreFix(X)$ is level-wise closed. Let $P \leq S$ and assume that for all $\varphi, \psi \in \Hom_\PreFus(P,S)$ we have $\Phi_{\langle P,\varphi \rangle}(X) = \Phi_{\langle P,\psi \rangle}(X).$ Then the following hold
\begin{enumerate}
 \item[(a)] For $\varphi \in \Hom_\PreFus(P,S)$, the image $\varphi(P)$ is fully $\PreFus$-centralized if and only if 
  \[ \frac{\Phi_{\langle P,\varphi \rangle}(X)}{|C_S(\varphi(P))|} \not\equiv 0 \quad \modp\,.\]
 \item[(b)] For $\varphi \in \Hom_\PreFus(P,S)$ we have 
  \[ \sum\limits_{\psi \in \Hom_{\PreFus}(P,\varphi(P))} \Phi_{\langle P,\psi \rangle}(X) = |\Aut_\PreFus(\varphi(P))| \cdot \Phi_{\langle P,\varphi \rangle}(X)\,,  \]
and $\varphi(P)$ is fully $\PreFus$-normalized if and only if 
  \[ \frac{|\Aut_\PreFus(\varphi(P))| \cdot \Phi_{\langle P,\varphi \rangle}(X)}{ |N_S(\varphi(P))| } \not\equiv 0 \quad \modp \,.\] 
\end{enumerate}
\end{lem}
\begin{proof}
By assumption, there is a constant $k\in \Z_{(p)}$ such that $\Phi_{\langle P,\varphi \rangle} =k$ for every $\varphi \in \Hom_\PreFus(P,S)$. By Lemma \ref{lem:Cong}, $k$ is divisible by $|C_S(\varphi(P))|$ for each $\varphi \in \Hom_\PreFus(P,S)$, and we have the congruence
\[ \sum_{[\varphi] \in \Rep_{\PreFus}(P,S) } \frac{k}{|C_S(\varphi(P))|} \equiv \countS(X) \not\equiv 0\quad \modp\,. \]
Hence there is some $[\psi] \in \Rep_{\PreFus}(P,S)$ such that $k/|C_S(\psi(P))| \not\equiv 0~ \modp$. Since $|C_S(\psi(P))|$ is the highest power of $p$ that divides $k$, $\psi(P)$ must be fully $\PreFus$-centralized. It follows that, for $\varphi \in \Hom_\PreFus(P,S)$, $\varphi(P)$ is fully $\PreFus$-centralized if and only $k/|C_S(\varphi(P))| \not\equiv 0~ \modp$, proving part (a). 

The equation in part (b) follows from the facts that $ \Phi_{\langle P,\psi \rangle}(X) = \Phi_{\langle P,\varphi \rangle}(X)$ for all $\psi \in \Hom_\PreFus(P,\varphi(P))$, and that $|\Hom_\PreFus(P,\varphi(P))| = |\Aut_\PreFus(P)|$. The criterion for when $\varphi(P)$ is fully $\PreFus$-normalized is now proved in a similar way to part (a).
\end{proof}

We are now ready to identify a set of conditions that guarantee that the fixed-point pre-fusion system of an element in the double Burnside ring is locally saturated, in the sense of Definition \ref{defn:SatAtP}.

\begin{prop}\label{prop:SatAtP}
Let $S$ be a finite $p$-group, and let $X$ be an element in $\Afree(S,S)\pLoc$ with $\countS(X)$ not divisible by $p$, such that $\PreFus \defeq \PreFix(X)$ is level-wise closed. Let $P \leq S$ and assume that for all $\varphi, \psi \in \Hom_\PreFus(P,S)$ we have $\Phi_{\langle P,\varphi \rangle}(X) = \Phi_{\langle P,\psi \rangle}(X).$ Then $\PreFus$ is saturated at $P$.
\end{prop}
\begin{proof}
Let $\varphi \in \Hom_\PreFus(P,S)$ be such that $\varphi(P)$ is fully $\PreFus$-normalized. Then, by Lemma \ref{lem:CongSpecial}, we have 
\[ \frac{|\Aut_\PreFus(\varphi(P))| \cdot \Phi_{\langle P,\varphi \rangle}(X)}{ |N_S(\varphi(P))| } \not\equiv 0 \quad \modp \,.\]
As $|N_S(\varphi(P))| = |\Aut_S(\varphi(P))| \cdot |C_S(\varphi(P))|$, and $|C_S(\varphi(P))|$ divides $\Phi_{\langle P,\varphi \rangle}(X)$ by Lemma \ref{lem:Cong}, while $|\Aut_S(\varphi(P))|$ divides $|\Aut_\PreFus(\varphi(P))|$ since $\Aut_S(\varphi(P))$ is a subgroup of $\Aut_\PreFus(\varphi(P))$, this implies that 
\[ \frac{|\Aut_\PreFus(\varphi(P))|}{|\Aut_S(\varphi(P))|} \not\equiv 0 ~ \modp \quad \text{and} \quad \frac{\Phi_{\langle P,\varphi \rangle}(X)}{|C_S(\varphi(P))|} \not\equiv 0~ \modp\,. \] 
The former incongruence implies that $\Aut_S(\varphi(P))$ is a Sylow subgroup of $\Aut_\PreFus(\varphi(P))$, and by Lemma \ref{lem:CongSpecial} the latter implies that $\varphi(P)$ is fully $\PreFus$-centralized. This proves I${}_P$.

To prove II${}_P$, let $\varphi \in \Hom_\PreFus(P,S)$ be such that $\varphi(P)$ is fully $\PreFus$-centralized. By Lemma \ref{lem:FixedPtsOnBasis},
 \[ \Phi_{\langle P,\varphi \rangle}(X) 
    = \sum_{\langle Q,\psi \rangle } c_{\langle Q,\psi \rangle}(X) \cdot \Phi_{\langle P,\varphi \rangle}([Q,\psi]) 
    = \sum_{\langle Q,\psi \rangle } c_{\langle Q,\psi \rangle}(X) \cdot \frac{|N_{\varphi,\psi}|}{|Q|} \cdot |C_S(\varphi(P))|,   
\] 
so the incongruence $\frac{\Phi_{\langle P,\varphi \rangle}(X)}{|C_S(\varphi(P))|} \not\equiv 0~ \modp$ from Lemma \ref{lem:CongSpecial} implies that there exists a $(S,S)$-pair $(Q,\psi)$ with $c_{\langle Q,\psi \rangle}(X) \neq 0$ and $\frac{|N_{\varphi,\psi}|}{|Q|} \not\equiv 0~ \modp$.

Recall from Lemma \ref{lem:NisBiset} that $N_{\varphi,\psi}$ is a left-free $(N_\varphi,N_\psi)$-biset. In particular, as $Q \leq N_\psi$, we can regard $Q \backslash N_{\varphi,\psi}$ as a right $N_\varphi$-set. As $|Q \backslash N_{\varphi,\psi}|$ is not divisible by $p$, and $N_\varphi$ is a $p$-group, there must exist at least one $x \in N_{\varphi,\psi}$ such that the orbit $Qx$ is fixed by the $N_\varphi$-action on $Q \backslash N_{\varphi,\psi}$. This means that for each $g \in N_\varphi$, there exists an $h \in Q$ such that $xg = hx$. In other words, $x$ conjugates $N_\varphi$ into $Q$. Recall that the condition $x \in N_{\varphi,\psi}$ implies that there exists $y \in S$ such that $c_y \circ \varphi = \psi \circ c_x$ as homomorphisms $P \to \psi(Q)$. We now have a commutative diagram
\[
\xymatrix{ & Q \ar[rr]^{\psi}    && \psi(Q) \ar[rr]^{c_y^{-1}} && S\\
           N_\varphi \ar[ur]^{c_x} \\
           & P \ar[rr]^{\varphi} \ar[uu]^{c_x} \ar[lu]^{\incl} && \varphi(P) \ar[uu]^{c_y} \ar[uurr]^{\incl},
}
\] 
and putting $\overline{\varphi} = c_y^{-1} \circ \psi \circ c_x \colon N_\varphi \to S$ we get an extension of $\varphi$ as required. Finally, we note that as $c_{\langle Q,\psi\rangle}(X) \neq 0$, we have $\psi \in \PreFix(X)$, and hence $\overline{\varphi} \in \FusFix(X) = \overline{\PreFus}$.
\end{proof}

In the proof of Proposition \ref{prop:SatAtP} the left multiplication action of $N_\psi$ on $N_{\varphi,\psi}$ is free and $Q\le N_\psi$, so the fact that  that $|Q \backslash N_{\varphi,\psi}|$ is not divisible by $p$ implies that $N_\psi = Q$. By the comments made after Definition \ref{defn:SatAxioms}, $\psi$, and hence $\overline{\varphi}$, can not be extended in $\FusFix(X)$. This argument can be used to prove the following result.

\begin{prop}
Let $\F$ be a saturated fusion system on a finite $p$-group $S$. For $P \leq S$ and $\varphi \in \Hom_\F(P,S)$ such that $\varphi(P)$ is fully $\F$-centralized, there exists a homomorphsm $\overline{\varphi} \in \Hom_\F(N_\varphi,S)$ that extends $\varphi$ and cannot be extended further in $\F$. 
\end{prop}

\subsection{Proof of Theorem \ref{mthm:CharImpliesSat}} \label{subsec:ProofOfCharImpliesSat}

We now have all the ingredients to prove Theorem \ref{mthm:CharImpliesSat},
which will be an easy consequence of the following strengthening of Corollary \ref{cor:FixStabImpliesChar}. 

\begin{prop}\label{prop:StabIsSaturated}
Let $S$ be a finite group, and let $\Omega$ be an element in $\Afree(S,S)_{(p)}$ with $\countS(\Omega)$ not divisible by $p$.
\begin{itemize}
 \item[(a)] If $\PreFix(\Omega) \subseteq \FusRStab(\Omega),$ then
  \[ \PreFix(\Omega) = \FusFix(\Omega) = \FusOrb(\Omega) = \FusRStab(\Omega)\,, \]
 and $\FusRStab(\Omega)$ is saturated with right characteristic element $\Omega$. 
 \item[(b)]
 If $\PreFix(\Omega) \subseteq \FusLStab(\Omega),$ then
  \[ \PreFix(\Omega) = \FusFix(\Omega) = \FusOrb(\Omega) = \FusLStab(\Omega)\,, \]
 and $\FusLStab(\Omega)$ is saturated with left characteristic element $\Omega$. 
\end{itemize}
\end{prop}
\begin{proof}
We prove part (b) first. By Corollary \ref{cor:FixStabImpliesChar}, we need only prove the saturation claim. We know $\PreFus = \FusLStab$, so $\PreFus$ is a fusion system, and in particular level-wise closed. Lemma \ref{lem:PhiConstantOnStab} implies that for each $P\leq S$, and for $\varphi,\psi \in \Hom_{\PreFus}(P,S)$ we have $\Phi_{\langle P,\varphi \rangle}(\Omega) = \Phi_{\langle P,\psi \rangle}(\Omega)$. Hence Proposition \ref{prop:SatAtP} applies, showing that $\PreFus$ saturated at $P$ for every $P \leq S$. As $\overline{\PreFus} = \PreFus$, this implies that $\PreFus$, and hence $\FusLStab$, is saturated.

Part (a) follows by applying part (b) to $\op{\Omega}$.
\end{proof}

Theorem \ref{mthm:CharImpliesSat} is a special case of the following theorem.
\begin{thm}\label{thm:CharImpliesSat}
If a fusion system has a right or left characteristic element, then it is saturated.
\end{thm}
\begin{proof}
If $\Omega$ is a right characteristic element for a fusion system $\F$, then $\countS(\Omega)$ 
is not divisible by $p$, and, by Theorem \ref{thm:StabOfChar} and Proposition \ref{prop:PreFixOfChar},
\[ \PreFix(\Omega) = \F = \FusRStab(\Omega) \,.\] 
In particular, part (a) of Proposition \ref{prop:StabIsSaturated} applies to show that $\F$ is saturated. The argument for left characteristic elements is analogous.
\end{proof}

%%%%%%%%%%%%%%%%%%%%%%%%%%%%%%%%%%%%%%%%%%%%%%%%%%%%%%%%%%%%%%%%%%%%%%%%%%%%%
\section{Frobenius reciprocity implies saturation} \label{sec:FrobImpliesSat}
%%%%%%%%%%%%%%%%%%%%%%%%%%%%%%%%%%%%%%%%%%%%%%%%%%%%%%%%%%%%%%%%%%%%%%%%%%%%%

In Section \ref{sec:CharImpliesSat} we saw that the existence of a characteristic element for a fusion system implies that the fusion system is saturated. Since the characteristic idempotent of a saturated fusion system is unique, and a saturated fusion system can be reconstructed as the stabilizer fusion system of its characteristic idempotent, this implies that there is a bijection between saturated fusion systems and characteristic idempotents. This is an interesting result, but with limited practical use since the Linckelmann--Webb properties of a characteristic element are defined in terms of the fusion system it characterizes.

In this section we introduce a condition on elements in the double Burnside ring, called Frobenius reciprocity, and show that under mild technical conditions, satisfying this condition implies that an element is a right characteristic element for its stabilizer fusion system. The Frobenius reciprocity is independent of fusion systems, and thus this gives an intrinsic way of recognizing characteristic elements.

\subsection{Frobenius reciprocity} For finite groups $G_1, G_2, H_1$ and $H_2$, cartesian product induces a bilinear pairing
 \[ A(G_1,H_1) \times A(G_2,H_2) \longrightarrow A(G_1 \times G_2, H_1 \times H_2), \quad (X,Y)\mapsto X \times Y \, , \]
that passes to the $p$-local setting. In particular, for a finite $p$-group $S$, one obtains a bilinear pairing
 \[ A(S,S)\pLoc \times A(S,S)\pLoc \to A(S\times S,S\times S)\pLoc \, . \]
Frobenius reciprocity is a condition on the behavior of an element with respect to this pairing and restricting along the diagonal map $\Delta \colon S \to S \times S.$

\begin{defn} \label{defn:Frobenius}
Let $S$ be a finite $p$-group, and let $X$ be an element in $A(S,S)\pLoc$. We say that $X$ \emph{satisfies Frobenius reciprocity}, or that $X$ is a \emph{Frobenius reciprocity element}, if 
\begin{equation} \label{eq:Frobenius} 
(X \times X) \circ [S,\Delta] = (X \times 1) \circ [S,\Delta] \circ X \in A(S,S\times S)\pLoc, 
\end{equation}
where $1 = [S,1]$ is the identity in $A(S,S)\pLoc$.
\end{defn}

As the terminology suggests, the Frobenius reciprocity condition is related to classical Frobenius reciprocity in cohomology. The link is explained in \ref{subsec:Translate}.

\begin{prop}[\cite{BLO2,KR:ClSpectra}] \label{prop:CharIsFrob}
A characteristic element of a fusion system satisfies Frobenius reciprocity.
\end{prop}
\begin{proof}
This was proved for characteristic bisets on the level of cohomology in \cite{BLO2}, and the proof was lifted to the double Burnside ring and extended to all characteristic elements in \cite{KR:ClSpectra}.
\end{proof}

\subsection{Fixed-point homomorphisms of Frobenius reciprocity elements}
We currently have three copies of $S$ appearing in different roles. To avoid confusion, it is helpful to distinguish them notationally by putting $S_1 \defeq S_2 \defeq S$, and writing $\Afree(S,S_1\times S_2)\pLoc$ instead of $\Afree(S,S\times S)\pLoc$. 

Suppose $X$ is a $(S,S_1)$-biset and $Y$ is a $(S,S_2)$-biset, and let $[X]$ and $[Y]$ be their isomorphism classes. Then $([X] \times [Y]) \circ [S,\Delta]$ is the isomorphism class of the $(S,S_1\times S_2)$-biset $(X\times Y) \times_{(S_1\times S_2)} ((S_1 \times S_2) \times_{(S,\Delta)} S) $, while $([X] \times 1) \circ [S,\Delta] \circ [Y]$ is the isomorphism class of $(X \times S_2) \times_{(S_1 \times S_2)} ((S_1 \times S_2) \times_{(S,\Delta)} S) \times_S Y$. These sets admit a far more convenient description.

\begin{lem} \label{lem:DescSetProducts}
Let $S$ be a finite $p$-group and write $S_1 \defeq S_2 \defeq S$. Let $X$ be a $(S,S_1)$-biset and let $Y$ be a $(S,S_2)$-biset.
\begin{enumerate}
\item[(a)]
The $(S, S_1 \times S_2)$-biset $(X\times Y) \times_{(S_1\times S_2)} ((S_1 \times S_2) \times_{(S,\Delta)} S) $ is isomorphic to $X\times Y$ endowed with the  $(S,S_1 \times S_2)$-action
\[ (b_1,b_2) (x,y) a \defeq (b_1 x a, b_2 y a)\,, \] 
for $(b_1,b_2) \in S_1\times S_2 , a \in S, x \in X, y \in Y $.
\item[(b)]
The $(S, S_1 \times S_2)$-biset $(X \times S_2) \times_{(S_1 \times S_2)} ((S_1 \times S_2) \times_{(S,\Delta)} S) \times_S Y$ is isomorphic to $X\times Y$ endowed with the  $(S,S\times S)$-action
\[ (b_1,b_2) (x,y) a \defeq (b_1 x b_2^{-1}, b_2 y a)\,, \]
for $(b_1,b_2) \in S_1\times S_2 , a \in S, x \in X, y \in Y $.
%\item[(c)]
%The $(S_1 \times S_2)$-biset $(S_1 \times Y) \times_{(S_1 \times S_2)} ((S_1 \times S_2) \times_{(S,\Delta)} S) \times_S X$ is isomorphic to $X\times Y$ endowed with the  $(S,S\times S)$-action
%\[ (b_1,b_2) (x,y) a \defeq (b_1 x a, b_2 y b_1^{-1}), \]
%for $(b_1,b_2) \in S_1\times S_2 , a \in S, x \in X, y \in Y $.
\end{enumerate}
\end{lem}
\begin{proof} $ $

\begin{enumerate}
\item[(a)]
The $(S, S_1 \times S_2)$-biset bijection is given by sending $((x,y),((v_1,v_2),u))$ to $(xv_1u,yv_2u)$. We check that this is a $(S, S_1 \times S_2)$-map. The action on the former biset by $(a,(b_1,b_2))$ gives $(a,(b_1,b_2))((x,y),((v_1,v_2),u))=((b_1x,b_2y),((v_1,v_2),ua))$ which is sent through our map to $(b_1xv_1ua,b_2yv_2ua)$. Now multiplying in the latter biset $(xv_1u,yv_2u)$ by $(a,(b_1,b_2))$ gives $(a,(b_1,b_2))(xv_1u,yv_2u)=(b_1xv_1ua,b_2yv_2ua)$ which is exactly the element obtained before. One can easily check that the inverse of this map is given by sending $(x,y)$ on $((x,y),((1,1),1))$.
\item[(b)]
The $(S, S_1 \times S_2)$-biset bijection is given by sending $((x,t),((v_1,v_2),u),y)$ to $(xv_1v_2^{-1}t^{-1},tv_2uy)$. We check that this is a $(S, S_1 \times S_2)$-map. The action on the former biset by $(a,(b_1,b_2))$ gives
$(a,(b_1,b_2))((x,t),((v_1,v_2),u),y)=((b_1x,b_2t),((v_1,v_2),u),ya)$ which is sent through our map to $(b_1xv_1v_2^{-1}t^{-1}b_2^{-1},b_2tv_2uya)$. Now multiplying in the latter biset $(xv_1v_2^{-1}t^{-1},tv_2uy)$ by $(a,(b_1,b_2))$ gives 
$(a,(b_1,b_2))(xv_1v_2^{-1}t^{-1},tv_2uy)=(b_1xv_1v_2^{-1}t^{-1}b_2^{-1},b_2tv_2uya)$ 
which is again exactly the element obtained before. One can easily check that the inverse of this map is given by sending $(x,y)$ on $((x,1),((1,1),1),y)$.
\end{enumerate}
\end{proof}

It is not hard to see that an $(S,S_1\times S_2)$-pair must of the form $(P,\psi \times \varphi )$, where $P$ is a subgroup of $S$, and $\psi \colon P \to S_1$ and $\varphi \colon P \to S_2$ are homomorphisms.

\begin{lem} \label{lem:ProductsFixPts}
Let $S$ be a finite $p$-group and write $S_1 \defeq S_2 \defeq S$. Let $X \in A(S,S_1)\pLoc$ and let $Y \in A(S,S_2)\pLoc$. For every $(S,S_1\times S_2)$-pair $(P,\psi\times\varphi)$,
\begin{enumerate}
\item[(a)]
$ \Phi_{\langle P,\psi\times\varphi  \rangle} ((X \times Y) \circ [S,\Delta]) = \Phi_{\langle P,\psi  \rangle}(X) \cdot \Phi_{\langle P,\varphi  \rangle}(Y).   $
\item[(b)] If $\varphi$ is injective, then 
$ \Phi_{\langle P,\psi\times\varphi  \rangle} ((X \times 1) \circ [S,\Delta] \circ Y) = \Phi_{\langle \varphi(P),\psi \circ \varphi^{-1}  \rangle}(X) \cdot \Phi_{\langle P,\varphi  \rangle}(Y).   $
%\item[(c)]  $\displaystyle \Phi_{\langle P,\varphi\times\psi  \rangle} ((1 \times X) \circ [S,\Delta] \circ X) = \Phi_{\langle P,\varphi  \rangle}(X) \cdot  $
\end{enumerate}
\end{lem}
\begin{proof}
We prove part (b), leaving the simpler part (a) to the reader. It suffices to prove this for bisets $X$ and $Y$. In this case we observe that if $Z$ is $X \times Y$ endowed with the $(S, S_1\times S_2)$-action $(b_1,b_2)(x,y)a = (b_1xb_2^{-1},b_2ya)$, then the fixed-point set $Z^{(P,\psi\times\varphi)}$ consists of the pairs $(x,y)$ such that for every $a \in P$ we have
\[ (x,ya) = (\psi(a) x \varphi(a)^{-1}, \varphi(a) y). \]
The condition $ya = \varphi(a) y$ for all $a \in P$ is equivalent to $y \in Y^{(P,\varphi)}$. The condition $x = \psi(a) x \varphi(a)^{-1}$ for all $a \in P$ is equivalent to $x\varphi(a) = \psi(a) x$ for all $a \in P$. Assuming $\varphi$ is injective, this can also be written as $xb = \psi\circ \varphi^{-1}(b) x$ for all $b \in \varphi(P)$, which is equivalent to $x \in X^{(\varphi(P),\psi\circ \varphi^{-1})}$. We deduce that 
\[ \psi(a) x \varphi(a)^{-1} = X^{(\varphi(P),\psi\circ \varphi^{-1})} \times Y^{(P,\varphi)} \, ,  \]
and the result follows.
\end{proof}

Consequent to Lemma \ref{lem:ProductsFixPts} we get the following lemma, which will be useful to prove closure and stability results for the fixed-point pre-fusion system of a Frobenius reciprocity element.

\begin{lem} \label{lem:FrobFixPts}
Let $S$ be a finite group, and let $X$ be an element in $A(S,S)_{(p)}$ that satisfies Frobenius reciprocity. For a subgroup $P\leq S$, a group homomorphism $\psi \colon P \to S$, and a group monomorphism $\varphi \colon P \to S$, we have
\[  
  \Phi_{\langle P,\varphi \rangle}(X) \cdot \Phi_{\langle P,\psi \rangle}(X) 
  = \Phi_{\langle P,\varphi \rangle}(X) \cdot \Phi_{\langle \varphi(P),\psi\circ\varphi^{-1} \rangle}(X)\,.
\]
In particular, if $\Phi_{\langle P,\varphi \rangle}(X) \neq 0$, then
\[  
  \Phi_{\langle P,\psi \rangle}(X) = \Phi_{\langle \varphi(P),\psi\circ\varphi^{-1} \rangle}(X)\,.
\]   
\end{lem}
\begin{proof}
The first equation follows from Lemma \ref{lem:ProductsFixPts}, and the second follows by canceling.
\end{proof}

\begin{rmk} \label{rmk:FrobFixPtsNotInj}
Part (b) of Lemma \ref{lem:ProductsFixPts} can be extended to include non-injective $\varphi$ as follows: If $\Ker(\varphi)$ is not contained in $\Ker(\psi)$, then left-freeness of $X$ implies that there is no $x \in X$ with $x\varphi(a) = \psi(a) x$ for all $a \in P$, so $\Phi_{\langle P,\psi\times\varphi  \rangle} ((X \times 1) \circ [S,\Delta] \circ Y) = 0$. If $\Ker(\varphi) \subseteq \Ker(\psi)$, then there is a unique group homomorphism $\rho \colon \varphi(P) \to \psi(P)$ such that $\rho \circ \varphi = \psi$, and the condition $x\varphi(a) = \psi(a) x$ for all $a \in P$ is equivalent to $x b =\rho(b) x$ for all $b \in \varphi(P)$, so we get 
\[ \Phi_{\langle P,\psi\times\varphi  \rangle} ((X \times 1) \circ [S,\Delta] \circ Y) = \Phi_{\langle \varphi(P),\rho  \rangle}(X) \cdot \Phi_{\langle P,\varphi  \rangle}(Y) \, .\]
Lemma \ref{lem:FrobFixPts} can then be extended to deduce
\[  
  \Phi_{\langle P,\psi \rangle}(X) = \Phi_{\langle \varphi(P),\rho \rangle}(X)\,,
\]   
when $\Phi_{\langle P,\varphi \rangle}(X) \neq 0$.
\end{rmk}

\subsection{Level-wise closure and local saturation}
Using Lemma \ref{lem:FrobFixPts}, we now show that, for a Frobenius reciprocity element $X$ in $A(S,S)\pLoc$, $\PreFix(X)$ is level-wise closed and saturated at every $P \leq S$. At times it will be easier to work with $\op{X}$ than $X$, but this makes no difference: once level-wise closure is established we have $\PreFix(X) = \PreFix(\op{X})$, so either way we get the desired results.

\begin{lem} \label{lem:LevelwiseClosed}
Let $S$ be a finite group, let $X$ be a Frobenius reciprocity element in $\Afree(S,S)\pLoc$ with augmentation not divisible
by $p$, and set $\PreFus \defeq \PreFix(\op{X})$. Let $P\leq S$, and let
$\iota$ denote the inclusion $P \hookrightarrow S$.
\begin{enumerate}
 \item[(a)] If $\varphi \in \Hom_{\PreFus}(P,S)$, then 
 $\Phi_{\langle P,\varphi \rangle}(\op{X}) = \Phi_{\langle P,\iota \rangle}(\op{X})$.
 \item[(b)] $\Hom_S(P,S) \subseteq \Hom_{\PreFus}(P,S)$.
 \item[(c)] If $\varphi \in \Hom_{\PreFus}(P,S)$ then $\varphi^{-1} \in \Hom_{\PreFus}(\varphi(P),S)$.
 \item[(d)] If $\varphi \in \Hom_{\PreFus}(P,S)$ and $\psi \in \Hom_{\PreFus}(\varphi(P),S)$, then 
            $\psi \circ \varphi \in \Hom_{\PreFus}(P,S)$.
\end{enumerate}
\end{lem}
\begin{proof}
First recall that $\Phi_{\langle P,\varphi \rangle}(\op{X}) = \Phi_{\langle \varphi(P),\varphi^{-1} \rangle}(X)$
for all $(S,S)$-pairs $(P,\varphi)$. Hence Lemma \ref{lem:FrobFixPts} implies that if $(P, \varphi)$ is an
$(S,S)$-pair with $\varphi \in \Hom_\PreFus(P,S)$, then, for any $(S,S)$-pair $(Q,\psi)$ with $\psi(Q)= \varphi(P)$, we have 
\begin{equation} \label{eq:OpFrobFixPts}
  \Phi_{\langle Q,\psi \rangle}(\op{X}) = \Phi_{\langle Q,\varphi^{-1}\circ\psi \rangle}(\op{X}).
\end{equation}   

Part (a) follows by taking $\psi = \varphi$ in Equation \ref{eq:OpFrobFixPts}. 

Part (b) is equivalent to $\Phi_{\langle P,\iota \rangle}(\op{X}) \neq 0$. By part (a), it suffices to show 
that $\Hom_\PreFus(P,S)$ is nonempty. This follows from part (a) of Lemma \ref{lem:Cong} since 
$\countS(\op{X}) = \countS(X)$ is not divisible by $p$.

For part (c), we take $\psi$ to be the inclusion $i \colon \varphi(P) \hookrightarrow S$ in Equation \ref{eq:OpFrobFixPts} and obtain
  $\Phi_{\langle \varphi(P),\varphi^{-1} \rangle}(\op{X}) = \Phi_{\langle \varphi(P),i \rangle}(\op{X}).$
As $ \Phi_{\langle \varphi(P),i \rangle}(\op{X}) \neq 0$ by part (b), this implies that $\varphi^{-1}$ is in $\PreFus$. 

For part (d), we first observe that by part (c), $\psi \in \Hom_{\PreFus}(\varphi(P),S)$ implies 
$\psi^{-1} \in \Hom_{\PreFus}(\psi\circ\varphi(P),S)$. Applying Equation \ref{eq:OpFrobFixPts} with 
$(\psi\circ\varphi(P),\psi^{-1})$ in place of $(P,\varphi)$ and $(P,\varphi)$ in place of $(Q,\psi)$, we obtain
$\Phi_{\langle P,\psi\circ\varphi \rangle}(\op{X}) = \Phi_{\langle P,\varphi \rangle}(\op{X}) \neq 0,$ so 
$\psi \circ \varphi \in \Hom_{\PreFus}(P,S)$.
\end{proof}

\begin{cor}\label{cor:LWClosedAndStable}
Let $S$ be a finite group, and let $X$ be a Frobenius reciprocity element in $\Afree(S,S)\pLoc$ with augmentation not divisible by $p$. Then $\PreFix(X) = \PreFix(\op{X})$ and both are level-wise closed. Furthermore, the following hold.
\begin{enumerate}
 \item[(a)] If $\varphi \colon P \to S$ and $\psi \colon Q \to S$ are morphisms in $\PreFix(X)$ with $\varphi(P)=\psi(Q)$, then $\Phi_{\langle P,\varphi \rangle}(X) = \Phi_{\langle Q,\psi \rangle}(X)$.
 \item[(b)] If $\varphi,\psi \colon P \to S$ are morphisms in $\PreFix(X)$, then $\Phi_{\langle P,\varphi \rangle}(\op{X}) = \Phi_{\langle P,\psi \rangle}(\op{X})$.
\end{enumerate}
\end{cor}
\begin{proof}
The level-wise closure of $\PreFix(\op{X})$ follows from parts (b)--(d) of Lemma \ref{lem:LevelwiseClosed},
and consequently Lemma \ref{lem:FixOp} implies that $\PreFix(X) = \PreFix(\op{X})$. The
claim in (b) follows from part (a) of Lemma \ref{lem:LevelwiseClosed}, and (a) is just the reformulation 
for $\PreFix(X)$.
\end{proof}

\begin{cor} \label{cor:FrobSatAtP}
Let $S$ be a finite group, and let $X$ be a Frobenius reciprocity element in $\Afree(S,S)\pLoc$ with augmentation not divisible by $p$. Then $\PreFix(X)$ is saturated at $P$ for every $P \leq S$.
\end{cor}
\begin{proof} This follows from Corollary \ref{cor:LWClosedAndStable} and Proposition \ref{prop:SatAtP}.
\end{proof}

\subsection{Closure under restriction}
Now that we have shown that the fixed-point pre-fusion system of a Frobenius reciprocity element $X$ is level-wise closed and saturated at every subgroup, all that remains to prove saturation is to show that $\PreFix(X)$ is closed under restriction. This is easy when $X$ is a biset: If $\varphi \colon P \to S$ is a morphism in $\PreFix(X)$, then $X^{(P,\varphi)}$ is nonempty. If $\psi$ is the restriction to a subgroup $Q \leq P$, then $\Delta(Q,\psi) \leq \Delta(P,\varphi)$ implies $X^{(P,\varphi)} \subseteq X^{(Q,\psi)}$, so $X^{(Q,\psi)}$ is nonempty as well, and $\psi$ is in $\PreFix(X)$. When $X$ is a general element in the double Burnside ring, things are not so simple, as $\Delta(Q,\psi) \leq \Delta(P,\varphi)$ does not imply any relationship between $\Phi_{\langle Q,\psi \rangle}(X)$ and $\Phi_{\langle P,\varphi \rangle}(X)$. Indeed, $\PreFix(X)$ need not be closed under restriction for an arbitrary $X$. The good news is that $\PreFix(X)$ is closed when $X$ satisfies Frobenius reciprocity and has augmentation prime to $p$, as we show below. The first step in this direction is to get some control over the possible ways to extend a given morphism in $\PreFix(X)$.

\begin{lem} \label{lem:CongExt}
Let $S$ be a finite group and let $X$ be an element in $\Afree(S,S)_{(p)}$.
For subgroups $P < Q$ of $S$ such that $P$ has index $p$ in $Q$, and a monomorphism 
$\varphi \colon P \to S$ that can be extended to a homomorphism $\overline{\varphi} \colon Q \to S$, 
we have
\[ \frac{\Phi_{\langle P,\varphi \rangle}(X)}{|C_S(\varphi(P))|} 
   \equiv \sum_{[\psi] \in \extend{P}{\varphi}{Q}} \frac{\Phi_{\langle Q,\psi \rangle}(X)}{|C_S(\psi(Q))|}\quad \modp\,,  \]
where 
\[ \extend{P}{\varphi}{Q} = \left\{ [\psi] \in \Rep(Q,S) \mid [\psi\vert_P] = [\varphi]   \right\}\,. \]
\end{lem}
\begin{proof}
It is enough to prove this when $X$ is an $(S,S)$-biset. For clarity, we make a
notational distinction between the two copies of $S$ by regarding $X$ as an
$(S_1,S_2)$-biset, with the understanding that $S_1 = S_2 = S$.

Now, consider the subset 
\[ Y \defeq S_2 X^{(P,\varphi)} \subseteq X\,,\] 
consisting of elements of the form $ax$ where $a \in S_2$ and $x \in X^{(P,\varphi)}$. Although $Y$ is not necessarily a $(S_1,S_2)$-subset of $X$, we show that $Y$ is closed under the $(Q,S_2)$-action obtained by restriction. To prove this, it is enough to show that if $x \in  X^{(P,\varphi)}$ and $b \in Q$, then $xb \in Y$. As $xb = \overline{\varphi}(b) (\overline{\varphi}(b)^{-1}xb)$, it suffices to show that $\overline{\varphi}(b)^{-1}xb \in X^{(P,\varphi)}$. To do this, we first note that $P$ must be normal in $Q$ because of the index, and for $g \in P$ we have
$\varphi(bgb^{-1}) = \overline{\varphi}(b)\varphi(g)\overline{\varphi}(b)^{-1}.$
Hence, for all $g \in P$ we have
\begin{align*} 
(\overline{\varphi}(b)^{-1}xb)g 
&= \overline{\varphi}(b)^{-1} x (bgb^{-1}) b = \overline{\varphi}(b)^{-1} \varphi(bgb^{-1}) xb \\
&= \overline{\varphi}(b)^{-1} (\overline{\varphi}(b)\varphi(g)\overline{\varphi}(b)^{-1})xb = \varphi(g) (\overline{\varphi}(b)^{-1}xb), 
\end{align*}
so $\overline{\varphi}(b)^{-1}xb \in X^{(P,\varphi)}$.

Next, we consider the induced right $Q$-subset $S_2\backslash Y$. We have a congruence
\[ |S_2 \backslash Y| \equiv |(S_2 \backslash Y)^Q| \quad \modp\,, \]
and the result follows once we show that 
\begin{equation}\label{eq:OrbitsInY}
 |S_2 \backslash Y | = \frac{\Phi_{\langle P,\varphi \rangle}(X)}{|C_S(\varphi(P))|} \, ,
\end{equation}
and
\begin{equation}\label{eq:FixedOrbitsInY}
 |(S_2 \backslash Y)^Q| = \sum_{[\psi] \in \extend{P}{\varphi}{Q}} \frac{\Phi_{\langle Q,\psi \rangle}(X)}{|C_S(\psi(Q))|} \,. 
\end{equation}

Let $x \in X^{(P,\varphi)}$ and $a \in S_2$. For $g \in P$, we have
\[ axg = a\varphi(g)x = \varphi(g)(\varphi(g)^{-1}a\varphi(g))x\,. \]
Since $X$ is left-free, this implies that $ax \in X^{(P,\varphi)}$ if and only if $a \in C_S(\varphi(P))$. Equation \eqref{eq:OrbitsInY} follows.

To prove Equation \eqref{eq:FixedOrbitsInY}, let $Y_0 \subseteq Y$ be the pre-image of 
$(S_2\backslash Y)^Q$ under the projection $Y \to S_2\backslash Y$. Just as in the proof 
of Lemma \ref{lem:Cong}, we obtain a map $\theta \colon Y_0 \to  \Inj(Q,S_2)$ such that
for all $g \in Q$ we have $yg = \theta(y)(g)y$, and, since 
$\theta(ax) = c_a\circ\theta(x)$, an induced map $\widetilde{\theta} \colon S_2\backslash Y_0 \to \widetilde{\Inj}(Q,S_2) = S_2 \backslash \Inj(Q,S_2)$, fitting into a commutative diagram
\[ \xymatrix{
  Y_0 \ar[rr]^{\theta \hphantom{\Inj}} \ar[d]^{q}  && \Inj(Q,S_2) \ar[d]^{q} \\ 
  S_2\backslash Y_0  \ar[rr]^{ \widetilde{\theta} \hphantom{\Inj}} && \widetilde{\Inj}(Q,S_2) 
  }
\]
where the vertical maps are the canonical projection onto $S_2$-orbits. 
Again, as in the proof of Lemma \ref{lem:Cong}, we deduce that 
\[ |(S_2 \backslash Y)^Q| = |S_2 \backslash Y_0| 
        = \sum_{[\psi] \in \widetilde{\Inj}(Q,S_2)} \widetilde{\theta}^{-1}([\psi]) 
        = \sum_{[\psi] \in \widetilde{\Inj}(Q,S_2)} \frac{\Phi_{\langle Q,\psi \rangle}(Y)}{|C_S(\psi(Q))|}\,. \] 

The proof is completed by showing that 
\[ 
\Phi_{\langle Q,\psi \rangle}(Y) = 
\begin{cases} 
\Phi_{\langle Q,\psi \rangle}(X), &\text{if $\psi \in \extend{P}{\varphi}{Q}$};\\
0, &\text{otherwise}\,.
   \end{cases} 
\]
If $y \in Y^{(Q,\psi)}$, then $yb = \psi(b)y$ for all $b \in Q$, in particular $b \in P$. We can also write $y = ax$ with $x \in X^{(P,\varphi)}$ and $a \in S_2$. Therefore, for $b \in P$, we have 
 \[yb = axb = a\varphi(b)x = c_a\circ\varphi(b)ax= c_a\circ\varphi(b)y \, ,\]
so $\psi(b)y = c_a\circ\varphi(b)y$, and by left-freeness $\psi(b) = c_a\circ\varphi(b)$. We deduce that if $Y^{(Q,\psi)}$ is nonempty, then $\psi \in \extend{P}{\varphi}{Q}$, so $\Phi_{\langle Q,\psi \rangle}(Y) = 0$ when $\psi \notin \extend{P}{\varphi}{Q}$.

Next we show that $Y^{(Q,\psi)} = X^{(Q,\psi)}$ when $[\psi] \in \extend{P}{\varphi}{Q}$. We certainly have $Y^{(Q,\psi)} \subseteq X^{(Q,\psi)}$ since $Y \subseteq X$. Now, $[\psi] \in \extend{P}{\varphi}{Q}$ implies that $\psi\vert_P = c_a \circ \varphi$ for some $a \in S_2$. If $x \in  X^{(Q,\psi)}$, then, for all $b \in P$,
\[ a^{-1}xb = a^{-1}\psi(b)x = a^{-1} a \varphi(b) a^{-1} x = \varphi(b) a^{-1} x, \]
so $a^{-1}x \in X^{(P,\varphi)}$, and $x \in Y = SX^{(P,\varphi)}$. Thus $X^{(Q,\psi)} = Y \cap X^{(Q,\psi)} = Y^{(Q,\psi)}$.
\end{proof}

\begin{lem} \label{lem:FrobPreFixClosed}
Let $S$ be a finite group, and let $X$ be a Frobenius reciprocity element in $\Afree(S,S)\pLoc$ with augmentation not divisible by $p$. Then $\PreFix(X)$ is closed.
\end{lem}
\begin{proof} 
Write $\PreFus$ for $\PreFus(X)$ to save notation. By Corollary \ref{cor:LWClosedAndStable}, $\PreFus$ is level-wise closed, and Corollary \ref{cor:FrobSatAtP} shows that $\PreFus$ is saturated at $P$ for every $P \leq S$. 

We prove that $\PreFus$ is closed by showing that  
\[ \Hom_{\PreFus}(Q,S) = \Hom_{\overline{\PreFus}}(Q,S)\,. \] 
for all $Q \leq S$ by downward induction on conjugacy classes of subgroups of $S$.

For the base case, $Q = S$, level-wise closure of $\PreFus$ implies that 
$\Aut_{\PreFus}(S) = \Aut_{\overline{\PreFus}}(S)$.

For the inductive step, let $\calH$ be a family of subgroups of $S$ that is closed under
$\PreFus$-conjugacy and taking supergroups, and assume that for all $Q \in \calH$ we have 
 \[ \Hom_{\PreFus}(Q,S) = \Hom_{\overline{\PreFus}}(Q,S)\,. \] 
Let $P$ be maximal among subgroups of $S$ not in $\calH$, 
and set $\calH' \defeq \calH \cup [P]_\PreFus$. To
show that the induction hypothesis holds for $\calH'$, 
it suffices, since $P$ was chosen arbitrarily from its 
$\PreFus$-conjugacy class, to show that 
$\Hom_{\PreFus}(P,S) = \Hom_{\overline{\PreFus}}(P,S)$. 
Furthermore, as $\PreFus$ is level-wise closed, it suffices 
to show that $\PreFus$ is closed under restricting morphisms to $P$.
That is, we need to show that if $P < Q \leq S$ and 
$\varphi \in \Hom_{\PreFus}(Q,S)$, then the restriction
$\varphi\vert_P$ is in $\PreFus$. By the induction hypothesis, 
$\PreFus$ is closed under restriction of morphisms to groups in $\calH$, 
so it is enough to consider the case where $P \triangleleft Q$ is an 
extension of index $p$.

First consider the case where $\varphi(P)$ is fully centralized in $\PreFus$. 
Using Lemma \ref{lem:CongExt} we have
\[ \frac{\Phi_{\langle P,\varphi\vert_P \rangle}(X)}{|C_S(\varphi(P))|} 
   \equiv \sum_{[\psi] \in \extend{P}{\varphi\vert_P}{Q}} \frac{\Phi_{\langle Q,\psi \rangle}(X)}{|C_S(\psi(Q))|} \quad \modp  \,.\]
Similarly, for the inclusion $i \colon \varphi(P) \hookrightarrow S$, we have
\[ \frac{\Phi_{\langle \varphi(P),i \rangle}(X)}{|C_S(\varphi(P))|} 
   \equiv \sum_{[\rho] \in \extend{\varphi(P)}{i}{\varphi(Q)} } \frac{\Phi_{\langle \varphi(Q),\rho \rangle}(X)}{|C_S(\rho(\varphi(Q)))|}\quad \modp \,. \]
There is a bijection $\extend{P}{\varphi}{Q} \to \extend{\varphi(P)}{i}{\varphi(Q)}$
sending $[\psi]$ to $[\psi \circ \varphi^{-1}]$. Moreover, since $\varphi \in \PreFus$, 
Lemma \ref{lem:FrobFixPts} implies that for $\psi \in \extend{P}{\varphi\vert_P}{Q}$
\[   \Phi_{\langle Q,\psi \rangle}(X) 
   = \Phi_{\langle \varphi(Q),\psi\circ \varphi^{-1} \rangle}(X) \,.
\]
Thus the sums on the right sides of the two congruences above actually agree term by 
term, and we deduce that
\[  \frac{\Phi_{\langle P,\varphi\vert_P \rangle}(X)}{|C_S(\varphi(P))|} 
   \equiv  \frac{\Phi_{\langle \varphi(P),i \rangle}(X)}{|C_S(\varphi(P))|} \quad \modp\,. \]
Now, by Lemma \ref{lem:FixOp}, we have
\[ \Phi_{\langle \varphi(P),i \rangle}(X) = \Phi_{\langle \varphi(P),i \rangle}(\op{X}),\]
and, since $\varphi(P)$ is fully centralized, Lemma \ref{lem:CongSpecial} and Corollary \ref{cor:LWClosedAndStable} give 
\[   \frac{\Phi_{\langle \varphi(P),i \rangle}(\op{X})}{|C_S(\varphi(P))|} \not\equiv 0 \quad \modp \,. \]
We conclude that $\Phi_{\langle P,\varphi\vert_P \rangle}(X) \neq 0$, and hence $\varphi\vert_P \in \Hom_\PreFus(P,S)$.

Now consider the general case, no longer assuming that $\varphi(P)$ is fully
$\PreFus$-centralized. We shall apply an argument analogous to the proof of
Proposition A.2 of \cite{BLO2} to obtain a homomorphism $\alpha \in
 \Hom_{\PreFus}(\varphi(Q),S)$ such that $\alpha(\varphi(P))$ is fully 
$\PreFus$-centralized. The previous argument then implies that the restrictions
$\alpha\vert_{\varphi(P)}$ and $(\alpha\circ\varphi)\vert_P$ are in $\PreFus$, and as
$\PreFus$ is level-wise closed, this implies that $\varphi\vert_P$ is in $\PreFus$.

To obtain $\alpha$, let $\gamma \colon \varphi(P) \xrightarrow{\cong} P'$ be an
isomorphism in $\PreFus$ such that $P'$ is fully $\PreFus$-normalized. As $\PreFus$ 
has Property I${}_P$, this implies that $P'$ is fully 
$\PreFus$-centralized, and that $\Aut_S(P')$ is a Sylow subgroup of
$\Aut_{\PreFus}(P')$. The latter implies that
$\gamma^{-1}\Aut_S(P')\gamma$ is a Sylow subgroup of
$\Aut_\PreFus(\varphi(P))$, and hence there exists $\chi \in \Aut_\PreFus(\varphi(P))$
such that 
\[ \Aut_S(\varphi(P)) \leq \chi^{-1} \circ \gamma^{-1} \Aut_S(P')\gamma \circ \chi \,.\]
This in turn implies that $N_{\gamma \circ \chi} = N_S(\varphi(P))$, and as 
$\gamma\circ \chi(\varphi(P)) = P'$ 
is fully $\PreFus$-centralized, Property II${}_{\varphi(P)}$ implies that 
there exists a homomorphism $\overline{\alpha} \in \Hom_{\overline{\PreFus}}(N_S(P),S)$ such that 
$\overline{\alpha}\vert_{\varphi(P)} = \gamma \circ \chi,$ and in particular 
$\overline{\alpha}(\varphi(P)) = P'$. The desired $\alpha$ is obtained by restricting
$\overline{\alpha}$ to $\varphi(Q)$. (Recall that $P$ is normal in $Q$, so $\varphi(Q) \leq N_S(\varphi(P))$. By the induction hypothesis, $\Hom_{\overline{\PreFus}}(\varphi(Q),S) =  \Hom_{\PreFus}(\varphi(Q),S)$, so $\alpha$ is in $\PreFus$.

This completes the induction, and hence the proof that $\PreFus$ is closed.
\end{proof}

\subsection{Proof of Theorems \ref{mthm:FrobImpliesSat} and \ref{mthm:Bijection}}

Collecting our results from this section, we deduce that Frobenius reciprocity implies saturation. Theorem \ref{mthm:FrobImpliesSat} is a consequence of the following result.

\begin{thm}\label{thm:FrobImpliesSat}
Let $S$ be a finite $p$-group and let $X$ be an element in $\Afree(S,S)\pLoc$. If $\countS(X)$ is not divisible by $p$ and $X$ satisfies Frobenius reciprocity, then 
  \[ \PreFix(X) = \FusFix(X) = \FusOrb(X) = \FusRStab(X)\,,  \]
and $\FusRStab(X)$ is saturated with right characteristic element $X$.
\end{thm}
\begin{proof} 
Corollaries \ref{cor:LWClosedAndStable} and \ref{cor:FrobSatAtP}, and Lemma \ref{lem:FrobPreFixClosed} combine to show that $\F \defeq \PreFix(X)$ is a saturated fusion system. Corollary \ref{cor:LWClosedAndStable} and Lemma \ref{lem:RephraseStable} show that $X$ is right $\F$-stable, so $\F \subseteq \FusRStab(X)$. The rest follows as in Proposition \ref{prop:StabIsSaturated}.
\end{proof}

Theorem \ref{thm:FrobImpliesSat} (combined with Proposition \ref{prop:CharIsFrob}) gives an intrinsic criterion for recognizing characteristic elements without mentioning fusion systems: A bifree element in $A(S,S)\pLoc$ is a characteristic element for a fusion system on $S$ if and only if it has augmentation prime to $p$ and satisfies Frobenius reciprocity. The fusion system, which must be saturated, can be recovered via a stabilizer, fixed-point or orbit-type construction.

Using the correspondence between saturated fusion systems and their characteristic elements, we obtain a surprising characterization of saturated fusion systems, which is Theorem \ref{mthm:Bijection} in the introduction. 

\begin{cor} \label{cor:Bijection}
For a finite $p$-group $S$, there is a bijective correspondence between saturated fusion systems on $S$ and  idempotents in $\Afree(S,S)_{(p)}$ of augmentation $1$ that satisfy Frobenius reciprocity. The bijection sends a saturated fusion system to its characteristic idempotent, and an idempotent to its stabilizer fusion system.
\end{cor}

This bijective correspondence between two a priori unrelated sets gives us a  completely new way to think of saturated fusion systems, and presents many opportunities for further research and better understanding of the role of fusion systems. Some immediate questions are addressed in Section \ref{sec:Reformulate}, where we translate some concepts for fusion systems into the language of characteristic idempotents. The bijection can also be applied to answer important questions in stable homotopy theory, and these are discussed in Sections \ref{sec:Splittings}, \ref{sec:AW} and \ref{sec:RTT}. 

%%%%%%%%%%%%%%%%%%%%%%%%%%%%%%%%%%%%%%%%%%%%%%%%%%%%%%%%%%%%%%
\section{Reformulating fusion theory in terms of idempotents} \label{sec:Reformulate}
%%%%%%%%%%%%%%%%%%%%%%%%%%%%%%%%%%%%%%%%%%%%%%%%%%%%%%%%%%%%%%
Now we know that a saturated fusion system is equivalent to its characteristic idempotent, and the saturation axioms can be formulated in the double Burnside ring. Then it is interesting to see what other notions for fusion systems can be expressed in terms of idempotents. In this section we consider morphisms between fusion systems, fusion subsystems, normal subsystems and quotient fusion systems, and reformulate each concept in terms of characteristic idempotents.

\subsection{Morphisms between fusion systems and fusion subsystems}

Morphisms between fusion systems, and in particular fusion subsystems, are encoded in the double Burnside ring by the following result.

\begin{prop}[\cite{KR:ClSpectra}] \label{prop:DetectSubsystem}
Let $\Ee$ be a fusion system on $P$ and $\Ff$ be fusion system on $S$.
Let $X_\Ff$ be a characteristic element for $\Ff$, let $\omega_\Ee$ be the characteristic 
idempotent for $\Ee$ and let $\gamma \colon P \to S$ be a homomorphism.
The following are equivalent.

 (1)  $g$ is $(\Ee,\Ff)$-fusion-preserving.

 (2)  $X_\Ff \circ [P,g]$ is right $\Ee$-stable.

 (3)  $X_\Ff \circ [P,g] \circ \omega_\Ee = X_\Ff \circ [P,g]$.
 
When $g$ is a monomorphism, the two following additional conditions, 
where $g^{-1} \colon g(P) \to P$ denotes the inverse of the induced isomorphism $g \colon P \to g(P)$, 
are also equivalent to (1), (2) and (3).

 (2') $[g(P),g^{-1}] \circ X_\Ff$ is left $\Ee$-stable.
 
 (3') $\omega_\Ee \circ [g(P),g^{-1}] \circ X_\Ff  = [g(P),g^{-1}] \circ X_\Ff.$
 
\end{prop}
\begin{proof}
The implications (1) $\Rightarrow$ (2) and (2) $\Rightarrow$ (3) are a reformulation of Lemma 7.7 in \cite{KR:ClSpectra}, working in the double Burnside ring instead of stable homotopy, and the implication (3) $\Rightarrow$ (1) is a reformulation of Proposition 7.11 in \cite{KR:ClSpectra}. When $g$ is a monomorphism, the equivalence of (1), (2') and (3') is proved in a similar fashion (see Lemma 8.5 in \cite{KR:ClSpectra}). 
\end{proof}

When $g$ is an inclusion, one can use Proposition \ref{prop:DetectSubsystem} to detect fusion subsystems. 
Moreover the proposition implies  
that, if $g$ is an $(\Ff,\Ff)$-fusion-preserving automorphism of $S$, then 
$[S,g^{-1}]\circ X_\Ff\circ [S,g]$ is a characteristic element for $\Ff$. In particular, if we take $X_\Ff$ to be the characteristic idempotent $\omega_{\Ff}$ of $\Ff$, then $[S,g^{-1}]\circ \omega_{\Ff}\circ [S,g]$ 
is also idempotent and, by uniqueness, $[S,g^{-1}]\circ \omega_{\Ff}\circ [S,g]=\omega_{\Ff}$.

\subsection{Normal fusion subsystems}

The normality condition for the inclusion of a fusion subsystem has a satisfying reformulation in the double Burnside ring, namely as invariance with respect to conjugation.

\begin{thm}
Let $\Ff$ be a saturated fusion system on $S$, and let $\Ee$ a saturated fusion subsystem of $\Ff$ on a strongly closed subgroup $T$ of $S$. Let $\omega_\Ee$ be the characteristic idempotent of $\Ee$. Then the following are equivalent:
\begin{enumerate}
\item $\Ee$ is a normal fusion subsystem of $\Ff$.
\item For every subgroup $Q$ of $T$ and every morphism $\psi\in\Hom_\Ff(Q,S)$, the following 
identity in $A(Q,Q)_{(p)}$ holds: 
$$[\psi(Q),\psi^{-1}]_T^Q \circ \omega_\Ee \circ [Q,\psi]_Q^T =[Q,1]_T^Q\circ \omega_\Ee \circ [Q,\incl]_Q^T\,.$$
\end{enumerate}
\end{thm}

\begin{proof}

Suppose $\Ee$ is a normal fusion subsystem of $\Ff$. Theorem \ref{thm:NormalFusion} says that for $Q\le T$ and $\psi\in\Hom_\Ff(Q,S)$, we have a decomposition $\psi=\chi\circ\phi$ with $\phi\in\Hom_\Ee(Q,T)$ and $\chi\in\Aut_\Ff(T)$. Normalcy of $\Ee$ implies that $\chi$ is $(\Ee,\Ee)$-fusion-preserving. We have $[Q,\psi]_Q^T=[T,\chi]_T^T\circ[Q,\phi]_Q^T$ in $A(Q,T)_{(p)}$, and compute
$$\begin{array}{ll}
[\psi(Q),\psi^{-1}]\circ \omega_\Ee \circ  [Q,\psi] &= [\phi(Q),\phi^{-1}]\circ[T,\chi^{-1}]\circ \omega_\Ee \circ [T,\chi]\circ[Q,\phi]\\
&=[\phi(Q),\phi^{-1}]\circ \omega_\Ee \circ[Q,\phi]\text{\qquad (as $\chi$ is fusion-preserving)}\\
&= [Q,1]\circ \omega_\Ee \circ [Q,\incl] \text{\qquad (by $\Ee$-stability)}
\end{array}$$

Suppose now that (2) is satisfied. Similar arguments to those in part (b) of the proof of Lemma \ref{lem:RephraseStable} give that 
\[ \Phi_{\langle R,\varphi \rangle} (\omega_\Ee) = \Phi_{\langle \psi(R),\psi \circ \varphi \circ \psi^{-1} \rangle} (\omega_\Ee)\,. \] 
for all $\psi\in\Hom_\Ff(Q,T)$, $R,\varphi(R)\le Q$ and $\varphi\in\Hom_\Ff(R,T)$. As $\Ee$ is saturated, $\PreFix(\omega_\Ee)=\Ee$. 
Hence $\varphi\in\Hom_\Ee(R,T)$ if and only if $\psi \circ \varphi \circ \psi^{-1}\in\Hom_\Ee(R,T)$ implying that $\Ee$ is normal in $\Ff$.

\end{proof}

\subsection{Quotient fusion systems}

Let $\F$ be a saturated fusion system on a finite $p$-group $S$, and let $T$ be a strongly $\Ff$-closed subgroup of $S$. Let $\bar{\F}$ denote the quotient fusion system $\F/T$, and denote the quotient group $S/T$ by $\bar{S}$. Let $\pi \colon S \to \bar{S}$ be the projection morphism. By Theorem \ref{thm:QuotientFusion}, this map is $(\Ff,\bar{\Ff})$-fusion-preserving, and induces a morphism of fusion systems $F_\pi \colon \F \to \bar{\F}$. We write $\bar{P}$ for $\pi(P)$ and $\bar{\varphi}$ for $F_\pi(\varphi)$. More generally, for subgroups $P,Q \leq S$, and a group homomorphism $\varphi \colon P \to Q$ such that $\varphi(P \cap T) \leq T$, we write $\bar{\varphi}$ for the (unique) induced homomorphism $\bar{\varphi} \colon \bar{P} \to \bar{Q}$ satisfying $\bar{\varphi} \circ \pi = \pi \circ \varphi.$ 

Our aim is to show that the characteristic idempotent of the quotient system  $\bar{\Ff}$ is the quotient of $\omega_\F$ by $T$, in the sense of the following definition.

\begin{defn}
Let $S$ be a finite $p$-group with normal subgroup $T$. The \emph{bideflation map} $\bidef_{S,T} \colon A(S,S) \to A(\bar{S},\bar{S})$ is the $\Z$-linear map that sends the isomorphism class of an $(S,S)$-biset $X$ to the isomorphism class of the double quotient $T\backslash X/T$ with the induced $(\bar{S},\bar{S})$-action. 
\end{defn}

The name is chosen because bideflation is obtained by deflating on both sides. As there is no danger of confusion, we will simply write $\bidef$ for $\bidef_{S,T}$, and we also use the term bideflation map and write $\bidef$ for the $p$-localized bideflation map. 

The bideflation map is rather complicated to describe on a general basis element $\langle P, \varphi\rangle$. However, when $\varphi$ preserves $T$, which is always the case for morphisms in $\F$, the quotient by $T$ can be controlled, resulting in the following description.

\begin{lem} \label{lem:bidefDesc}
If $(P,\varphi)$ is a $(S,S)$-pair such that $\varphi(P \cap T) \leq T$,
then
 \[ \bidef([P,\varphi]_S^S) = [\bar{P}, \bar{\varphi}]_{\bar{S}}^{\bar{S}}\,. \]
In particular this holds if $\varphi \in \Hom_\F(P,S)$.
\end{lem}
\begin{proof}
Recall that $[P,\varphi]$ is the isomorphism class of the $(S,S)$-biset $S \times_{(P,\varphi)} S$, and the corresponding $(S\times S)$-set is $(S \times S)/ \Delta(P,\varphi)$. The $(\bar{S} \times \bar{S})$-set corresponding to $T \backslash S \times_{P,\varphi} S /T$ is  $(T\times T) \backslash (S\times S) / \Delta_P^\varphi $. Using normalcy of $T$, we make the following 
identifications:
\begin{eqnarray*}
 (T\times T) \backslash (S\times S) / \Delta(P,\varphi) 
 &\cong &(S\times S) / ( (T \times T) \Delta(P,\varphi) )  \\
 &\cong & \frac{S\times S} {T \times T} / \frac{\Delta(P,\varphi)  (T \times T)}{T \times T}\\
 &\cong & (\bar{S} \times \bar{S}) / (\pi \times \pi)(\Delta(P,\varphi)) \\
 &\cong & (\bar{S} \times \bar{S}) / \Delta(\bar{P},\bar{\varphi})
\end{eqnarray*}
As $(\bar{S} \times \bar{S}) / \Delta_{\bar{P}}^{\bar{\varphi}}$ is the $(\bar{S} \times \bar{S})$-set 
corresponding to $\bar{S} \times_{(\bar{P},\bar{\varphi})} \bar{S}$, this proves the first claim. The 
second claim follows since $T$ is assumed to be strongly $\Ff$-closed.
\end{proof}

The bideflation map is a morphism of $\Z$-modules, by definition, but generally not a ring homomorphism. 
Indeed, even when $[P,\varphi]_S^S$ and $[Q,\psi]_S^S$ are basis elements satisfying the condition in 
Lemma \ref{lem:bidefDesc}, one has
\begin{align} 
    \bidef \left([P,\varphi] \circ [Q,\psi] \right) 
 &= \bidef \left( \DCF{P}{\varphi}{Q}{\psi}{S} \right) \notag \\ \label{eq:bidef1}
 &= \sum_{x \in P \backslash S / \psi(Q)} [ \bar{\psi}^{-1}(\bar{\psi}(\bar{Q})\cap \bar{P}^x),\bar{\varphi}\circ \bar{c}_x \circ \bar{\psi}], 
\end{align}
while
\begin{equation} \label{eq:bidef2}
 \bidef \left([P,\varphi]\right) \circ \bidef_{S,T} \left([Q,\psi] \right) 
   = \DCF{\bar{P}}{\bar{\varphi}}{\bar{Q}}{\bar{\psi}}{\bar{S}}.
\end{equation}
These sums contain the same terms, but with different multiplicities, as the indexing 
sets are different in general. In the following lemma we identify an important instance 
when the indexing sets can be identified.

\begin{lem} \label{lem:bidefComp}
Let $(P,\varphi)$ and $(Q,\psi)$ be $(S,S)$-pairs with $\varphi(P\cap T) \leq T$ and $\psi(Q\cap T) \leq T$. If $T \leq P$ or $T \leq \psi(Q)$, then 
\[ \bidef_{S,T} \left([P,\varphi] \circ [Q,\psi] \right) = \bidef_{S,T} \left([P,\varphi]\right) \circ \bidef_{S,T} \left([Q,\psi] \right)\,. \]
\end{lem}
\begin{proof}
In this case it is not hard to see (using the fact that $T$ is normal in $S$) that there is a bijection between $P \backslash S / \psi(Q)$ and $\bar{P} \backslash \bar{S} / \bar{\psi}(Q)$ that allows one to identify the sums in \eqref{eq:bidef1} and \eqref{eq:bidef2}.
\end{proof}

Exercising this control over composition, we can approach the $\F$-stability of quotients.

\begin{lem} \label{lem:QuotientFGenAndFStable}
If $\Omega$ is a right characteristic element for $\F$, then $\bidef(\Omega)$ is $\bar{\F}$-generated and right $\bar{\F}$-stable. The analogous statement holds for left characteristic elements.
\end{lem}
\begin{proof}
Since $\Omega$ is $\F$-generated, Lemma \ref{lem:bidefDesc} implies that $\bidef(\Omega)$ is $\bar{\F}$-generated. To show that $\bidef(\Omega)$ is right $\bar{\F}$-stable, let $R \leq \bar{S}$ and $\rho \in \Hom_{\bar{\F}}(R,\bar{S})$. By definition of $\bar{\F}$ there is a homomorphism $\psi \in \Hom_\F(\pi^{-1}(R),S)$ such that $\bar{\psi} = \rho$. Now, since $T\leq \psi(\pi^{-1}(R))$, $\F$-stability of $\Omega$ and Lemma \ref{lem:bidefComp} imply
\[ \bidef(\Omega) \circ [R,\rho] = \bidef(\Omega \circ [\pi^{-1}(R),\psi]) = \bidef(\Omega \circ [\pi^{-1}(R),\incl]) = \bidef(\Omega) \circ [R,\incl]\,, \]
proving right $\bar{\F}$-stability of $\bidef(\Omega)$. The proof for left characteristic elements is analogous.
\end{proof}

The lemma does not suffice to prove that the bideflation of a characteristic element for $\F$ is a characteristic element of $\bar{\F}$, since bideflation does not preserve augmentation. Instead we have the following lemma.

\begin{lem} \label{lem:AugOfQuotient}
If $(P,\varphi)$ is a $(S,S)$-pair such that $\varphi(P \cap T) \leq T$, then
 \[ \countS(\bidef([P,\varphi])) =  \frac{|S|}{|PT|} = \frac{|T\cap P|}{|T|} \cdot \countS([P,\varphi])\,. \]
In particular, if $T \leq P$, then 
 \[ \countS(\bidef([P,\varphi])) = \countS([P,\varphi])\,.\]
\end{lem}
\begin{proof} We have $\countS([P,\varphi]) = |S|/|P|$, and 
\[ \countS([\bar{P},\bar{\varphi}]) = \frac{|\bar{S}|}{|\bar{P}|} = \frac{|S|/|T|}{|PT|/|T|} = \frac{|S|}{|PT|} = \frac{|S|}{|P|\cdot |T|/|P\cap T|} = \frac{|T\cap P|}{|T|} \cdot \frac{|S|}{|P|}\,. \]
\end{proof}

Applying Lemma \ref{lem:AugOfQuotient} directly to control the augmentation of the bideflation of a characteristic element, we obtain the following result.

\begin{prop} \label{prop:CharOfQuotient}
Let $\F$ be a saturated fusion system on a finite $p$-group $S$, and let $T$ be a strongly $\F$-closed subgroup of $S$. Write $\bar{\F}$ for the quotient system $\F/T$ and $\bar{S}$ for the quotient $S/T$. Suppose $\Omega$ is a characteristic element for $\F$ such that for $P < S$ with $PT = S$ we have 
 \[ m_P(\Omega) \defeq \sum_{[\varphi] \in \Rep_\F(P,S)}     c_{\langle P,\varphi   \rangle}(\Omega)\equiv 0~ \modp \,. \]
Then $\bidef(\Omega)$ is a characteristic element for $\bar{\F}$. 
\end{prop}
\begin{proof}
By Lemma \ref{lem:QuotientFGenAndFStable}, $\bidef{\Omega}$ is $\bar{\F}$-generated and $\bar{\F}$-stable, and it only remains to show that $\countS(\bidef(\Omega)) \not\equiv 0~\modp$. A simple calculation shows 
\begin{align*}
 \countS(\Omega) 
 &= \countS\left( \sum_{[P]_S } \left( \sum_{[\varphi] \in \Rep_\F(P,S)}     c_{\langle P,\varphi   \rangle}(\Omega) \; [P,\varphi] \right) \right) \\
 &= \sum_{[P]_S } m_P(\Omega) \frac{|S|}{|P|} \\
 &\equiv m_S(\Omega) \quad \modp.
\end{align*}
In particular, $m_S$ is not divisible by $p$. Using Lemma \ref{lem:AugOfQuotient}, one similarly obtains
\begin{align*} 
   \countS(\bidef(\Omega)) 
   &= \sum_{[P]_S } m_P(\Omega) \frac{|S|}{|PT|} \\
   &\equiv \sum_{\substack{[P]_S\\ PT = S}} m_P(\Omega) \frac{|S|}{|PT|} \quad \modp.
\end{align*}
By assumption, $m_P(\Omega) \equiv 0 ~\modp$ when $PT = S$ and $P \neq S$, and we are left with
\[  \countS(\bidef(\Omega)) \equiv m_S(\Omega) \not\equiv 0 \quad \modp\,,\]
completing the proof.
\end{proof}

The hypothesis on $m_P(\Omega)$ in Proposition \ref{prop:CharOfQuotient} is not necessarily true for a characteristic element $\Omega$. Nor is it true that $\bidef(\Omega)$ necessarily has augmentation prime to $p$ if no assumptions are made on the numbers $m_P(\Omega)$. Indeed, Example \ref{ex:BidefOmegaNotChar} describes a setting where $\bidef{\Omega}$ is not a characteristic element for $\bar{\F}$.

\begin{example} \label{ex:BidefOmegaNotChar}
Let $S$ be a nontrivial finite $p$-group, and set $\Ff=\Ff_S(S)$ with $T=S$, so $\bar{\Ff} = \Ff/S$ is the trivial fusion system on the trivial group $\bar{S} = 1$. As a characteristic element for $\Ff$ one can take $\Omega = [S,\id] + (p-1)[1,\incl]$. Then $\bidef(\Omega) = p \, [1,\id]$ in $A(1,1)$ has augmentation $p$, and is not a characteristic element for $\bar{\Ff}$.
\end{example}

One instance where we can control the numbers $m_P(\Omega)$ is when $\Omega$ exhibits idempotence, as we saw in Lemma \ref{lem:CoeffsOfIdem}. Combining Proposition \ref{prop:CharOfQuotient} and Lemma \ref{lem:CoeffsOfIdem}, we obtain the following result.

\begin{prop} \label{prop:BidefOfIdempotent}
Let $\F$ be a saturated fusion system on a finite $p$-group $S$, and let $T$ be a strongly $\F$-closed subgroup of $S$. Write $\bar{\F}$ for the quotient system $\F/T$ and $\bar{S}$ for the quotient $S/T$. 
\begin{enumerate}
 \item[(a)] If $\Omega$ is a (right, left or fully) characteristic element for $\F$ that is idempotent $\modp$, then $\bidef(\Omega)$ is a (right, left or fully) characteristic element for $\bar{\F}$ that is idempotent $\modp$.
 \item[(b)] $\bidef(\omega_\F)$ is the characteristic idempotent of $\bar{\F}$.
\end{enumerate}
\end{prop}
\begin{proof} That $\bidef(\Omega)$ is a (right, left or fully) characteristic element for $\bar{\F}$ is an immediate consequence of Proposition \ref{prop:CharOfQuotient} and Lemma \ref{lem:CoeffsOfIdem}. The claim that $\bidef{\Omega}$ is idempotent $\modp$ then follows from the equations
 \[ m_P(\bidef(\Omega)) = \sum_{\substack{[R]_S \leq S \\ \bar{R} = P }} m_R(\Omega)\,, \]
for $P \leq \bar{S}$, and Lemma \ref{lem:CoeffsOfIdem}. The same reasoning applies to part (b).
\end{proof}

The idempotent condition in Proposition \ref{prop:BidefOfIdempotent} is less restrictive than it might appear at first glance. In practice there is usually no reason to prefer one characteristic element over another, unless the former is the characteristic idempotent. Since we know that every saturated fusion system does have a characteristic idempotent, there is at least one characteristic element to which the corollary applies. Even if one is reluctant to work $p$-locally (as one usually must in order to invoke the characteristic idempotent), one can easily obtain a characteristic element ---even a characteristic biset--- that is idempotent $\modp$ for a given saturated fusion system. Indeed, a characteristic biset $\Omega$ always exists by \cite[Proposition 5.5]{BLO2}, and, by \cite[Lemma 4.6]{KR:ClSpectra}, some power of $\Omega$ is idempotent $\modp$.

%%%%%%%%%%%%%%%%%%%%%%%%%%%%%%%%%%%%%%%%%%%%%%%%%%%%%%%%%%%
\section{Relation to stable homotopy of classifying spaces} \label{sec:Stable}
%%%%%%%%%%%%%%%%%%%%%%%%%%%%%%%%%%%%%%%%%%%%%%%%%%%%%%%%%%%
The correspondence between saturated fusion systems and Frobenius idempotents has important consequences in stable homotopy theory. By the Segal conjecture, stable selfmaps of the classifying space of a finite $p$-group $S$ correspond to elements in the $p$-completion of the double Burnside ring, $\pComp{A(S,S)}$. Thus the results in this paper can be used to express saturated fusion systems as stable idempotents of classifying spaces of $p$-groups satisfying Frobenius reciprocity and some technical conditions. In particular, this reinforces the idea that, from a homotopy-theoretic point of view, saturated fusion systems belong to the stable world. 

In Sections \ref{sec:Splittings}, \ref{sec:AW} and \ref{sec:RTT} we apply Theorem \ref{thm:FrobImpliesSat} and Corollary \ref{cor:Bijection} to obtain new results on the stable homotopy of classifying spaces. In this section we provide the tools needed to pass to the stable world. The section is divided into four subsections. In \ref{subsec:Segal} we recall the Segal conjecture. In \ref{subsec:Translate} we discuss the stable homotopy analogues of the Frobenius reciprocity and augmentation conditions of Theorem \ref{thm:FrobImpliesSat}, and also provide motivation for the usage of the term ``Frobenius reciprocity''. The right freeness condition imposed in Theorem \ref{thm:FrobImpliesSat} has no reasonable interpretation in stable homotopy, and in \ref{subsec:Relax} we discuss how to relax this condition by instead  imposing other conditions with more reasonable stable homotopy analogues. Finally, in \ref{subsec:ClSpectra} we recall the theory of classifying spectra of saturated fusion systems from \cite{KR:ClSpectra}, and adapt it to incorporate classifying spectra with added basepoint.

The reader is assumed to have some familiarity with stable homotopy theory and spectra; this is an extensive subject and providing the necessary background can not be reasonably done in this paper. For the results presented here one requires only a category of spectra with minimal structure, such as the homotopy category of spectra developed in \cite{Ad:Stable}. The reader is also referred to \cite{Ad:InfLoopSp} for background on stable transfer maps arising from finite covering maps, such as maps of classifying spaces induced by subgroup inclusions.

\subsection{The Segal conjecture} \label{subsec:Segal}
The stable homotopy of classifying spaces of finite groups is linked to the Burnside category via the Segal conjecture, and we briefly describe this link here. First, recall that for finite groups $G$ and $H$, there is a natural map 
\[ \alpha \colon A(G,H) \longrightarrow \{BG_+,BH_+\}, \quad \quad [K,\varphi] \longmapsto \Stable{B\varphi} \circ tr_K\,, \]
where the subscript $+$ denotes an added disjoint basepoint, $\{BG_+,BH_+\}$ is the group of homotopy classes of stable maps, $tr_K \colon \Stable{BG_+} \to \Stable{BK_+}$ is the stable transfer associated to the subgroup inclusion $K \leq G$, and $\Stable{B\varphi} \colon \Stable{BK_+} \to \Stable{BH_+}$ is the obvious map. The (single) Burnside ring of finite $G$-sets, $A(G)$, acts on $A(G,H)$ by cartesian product, and the Segal conjecture deals with completion at the augmentation ideal $I(G) \subseteq A(G)$.

\begin{thm}[\cite{Car,LMM}]\label{thm:Segal}
For finite groups $G$ and $H$, the natural map $\alpha \colon A(G,H) \longrightarrow \{BG_+,BH_+\}$ is an $I(G)$-adic completion map.
\end{thm}
\begin{proof}
The case where $H = 1$, which is the original Segal conjecture, was proved by Carlsson in \cite{Car}. Lewis--May--McClure used a transfer argument in \cite{LMM} to extend Carlsson's result to finite groups $H$. (The result was further extended to allow $H$ to be a compact Lie group in \cite{MSZ})
\end{proof}
May--McClure showed in \cite{MM} that when $S$ is a $p$-group, the $I$-adic completion of $A(S,H)$ is ``essentially $p$-completion''. This is helpful as the $p$-completion of $A(S,H)$ admits a convenient description: Since $A(G)$ is Noetherian and $A(S,H)$ is finitely generated, the Artin--Rees lemma applies, showing that $\pComp{A(S,H)} \cong \Zp \otimes A(S,H)$. Hence $\pComp{A(S,H)}$ is a free $\Zp$-module on the standard basis of $A(S,H)$ (cf.~Lemma \ref{lem:BurnsideBasis}). We offer the following useful formulation of the May--McClure result.

\begin{prop}[\cite{MM}] \label{prop:IAndpComp}
For a finite $p$-group $P$ and any finite group $G$, the $I(P)$-adic topology on $A(P,G)$ is finer than the $p$-adic topology, and the resulting completion map $\IComp{A(P,G)} \to \pComp{A(P,G)}$ is an injection whose image is the submodule of elements with augmentation in $\Z$.
\end{prop}
\begin{proof}
Write $I = I(P)$. May--McClure showed in \cite{MM} that if $|P| = p^n$, then  $I^{n+1} \subseteq pI$, proving the first claim. Since the $p$-adic completion map $A(P,G) \to \pComp{A(P,G)}$ is injective, it follows that the $p$-adic completion $\IComp{A(P,G)} \to \pComp{A(P,G)}$ is an injection. 

May--McClure also showed that if $K$ is the kernel of the restriction map $A(P,G) \to A(1,G)$, then the $I(P)$-adic topology on $K$ coincides with the $p$-adic topology. Observe also that the $A(P)$-action on $A(1,G) \cong \Z$ is given by $X \cdot n = |X| \cdot n$, so $I$ acts by the zero map, and $\IComp{A(1,G)} \cong A(1,G)$. By the Artin--Rees lemma, $I$-adic and $p$-adic completions are both exact on finitely generated modules, so the short exact sequence $K \to A(P,G) \to A(1,G)$ gives rise to a commutative diagram
\[ 
\xymatrix{ \pComp{K} \ar[r]\ar@{=}[d] & \IComp{A(P,G)} \ar[r] \ar[d]& A(1,G) \ar[d]\\
           \pComp{K} \ar[r] & \pComp{A(P,G)} \ar[r] & \pComp{A(1,G)}
}
\]
with exact rows that arise from $I$-adic and $p$-adic completion, respectively. The vertical maps are the canonical maps from $I$-adic to $p$-adic completion, coming from the fact that $I$-adic topology is finer. Observing that one can identify the restriction $A(P,G) \to A(1,G)$ with the augmentation $\countS \colon A(P,G) \to \Z$ completes the proof.
\end{proof}

Proposition \ref{prop:IAndpComp} says that $I(P)$-adic completion of $A(P,G)$ amounts to $p$-completing everything except one copy of $\Z$, which remains unchanged. This copy of $\Z$ corresponds to stable selfmaps of the sphere spectrum induced by maps $\Stable{BP_+} \to \Stable{BG_+}$ under the natural splitting $\Stable{X_+} \simeq \SphereSpectrum \vee \Stable{X}$, where $\SphereSpectrum$ denotes the sphere spectrum.

In particular, the May--McClure result allows us to regard $\IComp{A(P,G)}$ as a submodule of $\pComp{A(P,G)}$, and so it makes sense to talk about the element in $\pComp{A(P,G)}$ corresponding to a stable map $\Stable{BS_+} \to \Stable{BG_+}$, bypassing the $I$-adic completion. Note that we can also regard $A(S,G)\pLoc$ as a submodule of $\pComp{A(S,G)}$ in the usual way. One can extend the notion of characteristic element to include elements in the $p$-completed or $I$-adically completed double Burnside ring, and all the results obtained thus far for the $p$-localized double Burnside ring carry over to this setting.

\subsection{Translating to stable homotopy} \label{subsec:Translate} 
To translate the augmentation into stable homotopy, we first identify the augmentation $\countS \colon A(G,H) \to \Z$, for finite groups $G$ and $H$, with the map 
 \[ A(G,H) \longrightarrow A(1,1) \cong \Z,\quad X \longmapsto [H,\pi]_H^1 \circ X \circ [1,i]_1^G, \]
where $\pi \colon H \to 1$ and $i \colon 1 \to G$ are projection and inclusion, respectively. Taking the image of this map in the stable homotopy category, using the fact that the map $\alpha$ in the Segal conjecture is natural, one gets a \emph{stable augmentation}
 \[ \{BG_+,BH_+ \} \longrightarrow \{B1_+,B1_+\} \cong \Z, \quad f \longmapsto \Stable B \pi \circ f \circ \Stable Bi. \]
Note that the stable augmentation always takes values in $\Z$, even though, when $G$ is a $p$-group, the only difference between $\{BG_+,BH_+\}$ and the $p$-complete module $\pComp{A(G,H)}$ is one copy of $\Z$ that is not $p$-completed in $\{BG_+,BH_+\}$. This happens because the stable augmentation map factors through this uncompleted summand (cf.~ Proposition \ref{prop:IAndpComp}).  

For a finite group $S$, the Frobenius reciprocity condition 
\begin{equation} \label{eq:FrobeniusAgain}
  (X \times X) \circ [S,\Delta_S] = (X \times 1) \circ [S,\Delta_S] \circ X
\end{equation} 
on an element $X \in A(S,S)$ readily translates to a condition in stable homotopy, that is perhaps more natural-looking than the original, and allows us to explain the relationship to the classical Frobenius reciprocity property in cohomology. Applying $\alpha$ turns \eqref{eq:FrobeniusAgain} into a homotopy  
\begin{equation} \label{eq:FrobeniusStable}
  (\alpha(X) \wedge \alpha(X)) \circ \Stable B\Delta_S \simeq (\alpha(X) \wedge \id_{\Stable{BS_+}}) \circ \Stable B\Delta_S \circ \alpha(X)
\end{equation}
of stable maps from $\Stable{BS_+}$ to $\Stable{BS_+} \wedge \Stable{BS_+}$.

When $G$ is a finite group with Sylow subgroup $S$, let $[G]$ be $G$ regarded as an $(S,S)$-biset. Then $[G]$ is a characteristic biset for $\F_S(G)$, and $\alpha([G])$ factors as 
\[ \alpha([G]) \colon \Stable{BS_+} \xrightarrow{\Stable{B\incl}} \Stable{BG_+} \xrightarrow{tr_S} \Stable{BS_+}\]
where $tr_S$ is the transfer associated to the subgroup inclusion $S \leq G$. A standard result (see \cite{Ad:InfLoopSp}), expressing the naturality of transfers with respect to cartesian products, gives the homotopy 
\begin{equation}  \label{eq:PushPull}
  ( \id_{ \Stable{BG_+}} \wedge tr_S ) \circ \Stable{B\Delta_G}
   \simeq 
   ( \Stable{B\incl}  \wedge \id_{ \Stable{BS_+} }) \circ \Stable{B\Delta_S} \circ tr_S \, .
\end{equation}
of stable maps from $\Stable{BG_+}$ to $\Stable{BG_+} \wedge \Stable{BS_+}$.
Applying the cohomology functor, the diagonal maps $B\Delta_S$ and $B\Delta_G$ induce multiplication maps $\mu_G$ and $\mu_S$, respectively, and we obtain the commutative diagram
\[ 
\xymatrix{ 
   H^*(BG) \otimes H^*(BS) \ar[d]^{\id \otimes Tr} \ar[rr]^{Res \otimes \id} && H^*(BS) \otimes H^*(BS) \ar[r]^{\quad\quad \mu_S} & H^*(BS) \ar[d]^{Tr} \\
   H^*(BG) \otimes H^*(BG) \ar[rrr]^{\mu_G} &&& H^*(BG).
}
\]
This diagram expresses the familiar Frobenius reciprocity relation in cohomology; namely that for all $x\in H^*(BG)$ and $y \in H^*(BS)$, 
 \[ Tr(Res(x) \, y) = x \, Tr(y).   \]

Composing with $( tr_S \wedge \id_{ \Stable{BS_+} } )$ on the left and $ \Stable{B\incl} $ on the right of both sides of \eqref{eq:PushPull}, and using $ \Stable{B\Delta_G} \circ \Stable{B\incl} \simeq ( \Stable{B\incl} \wedge  \Stable{B\incl} ) \circ  \Stable{B\Delta_S} $, yields the homotopy 
 \[ (\alpha([G]) \wedge \alpha([G])) \circ \Stable B\Delta_S \simeq (\alpha([G]) \wedge \id_{\Stable{BS_+}}) \circ \Stable B\Delta_S \circ \alpha([G]), \]
which is equivalent to the Frobenius reciprocity condition for the biset $[G]$.

\subsection{Relaxing the right freeness condition} \label{subsec:Relax}
There is no reasonable interpretation for the right freeness condition in stable homotopy, and in this subsection we describe a way to relax that condition. First we note that the right freeness assumption in Theorem \ref{thm:FrobImpliesSat} cannot be just be removed: If $S$ is a finite group, and $\psi$ is a non-injective, idempotent endomorphism of $S$, then $[S,\psi]$ is an idempotent in $\pComp{A(S,S)}$ that satisfies Frobenius reciprocity and has augmentation $1$, but is certainly not the characteristic idempotent of any fusion system. Instead we show that the right freeness condition is automatically satisfied by an element of $\pComp{A(S,S)}$ that satisfies Frobenius reciprocity if we assume that it is not generated by maps that factor through proper subgroups of $S$. This is a familiar condition in stable homotopy ---first considered by Nishida in \cite{Ni}--- and the resulting statements are also useful in algebraic settings.

We begin by formulating the freeness condition in terms of fixed points. To discuss this, it is convenient to broaden our scope to include bisets with no (left or right) freeness conditions. For finite groups $G$ and $H$, let $B(G,H)$ be the Grothendieck group of isomorphism classes of finite $(G,H)$-bisets, with no freeness requirement. Recall that every $(G,H)$-biset can be regarded as a $(G \times H)$-biset, and this correspondence induces an isomorphism of $\Z$-modules from $B(G,H)$ to the Burnside ring $A(G\times H)$ finite $G \times H$-sets. For a subgroup $K \leq G \times H$ and a $(G,H)$-biset $X$, define the $K$-fixed point set by
 \[ X^K \defeq \{ x \in X \vert \forall (a,b) \in K \colon bx = xa \}, \]
and set
 \[ \Phi_{[K]}(X) = |X^K|, \]
where $[K]$ denotes the conjugacy class of $K$ in $G \times H$. Extending linearly, we get homomorphisms
 \[ \Phi_{[K]} \colon B(G,H) \to \Z, \]
one for every conjugacy class $[K]$, and by \cite{Burnside} an element $X \in B(G,H)$ is determined by the numbers $ \Phi_{[K]}(X)$. This remains true after $p$-localization or $p$-completion.

The Burnside module $A(G,H)$ is a submodule of $B(G,H)$, and for an $(S,S)$-pair $(L,\psi)$, the composite $A(G,H) \hookrightarrow B(G,H) \xrightarrow{\Phi_{[\Delta(L,\psi)]}} \Z$ equals $\Phi_{ \langle L,\psi \rangle }$. Also, the opposite map $\opmap \colon \Afree(G,H) \to \Afree(H,G)$ extends to a map $\opmap \colon B(G,H) \to B(H,G)$, and for $X \in A(G,H)$ we let $\op{X}$ denote the image in $B(H,G)$ under this map. (In Section \ref{sec:Burnside} $\op{X}$ was only defined for $X \in \Afree(G,H)$.)

\begin{lem} \label{lem:FixPtsOfFree}
Let $S$ be a finite $p$-group. 
\begin{enumerate}
 \item[(a)] An element $X \in \pComp{B(S,S)}$ is in $\pComp{A(S,S)}$ if and only if for every subgroup $K \leq S \times S$ that is not of the form $\Delta(P,\psi)$ for some $(S,S)$-pair $(P,\psi)$,
   \[ \Phi_K(X) = 0. \] 
 \item[(b)] An element $X \in \pComp{A(S,S)}$ is in $\pComp{\Afree(S,S)}$ if and only if for every $(S,S)$-pair $(P,\psi)$ where $\psi$ is not injective,
   \[ \Phi_{\langle P, \psi \rangle}(X) = 0. \]
 \item[(c)] If $X \in \pComp{A(S,S)}$ and $(P,\psi)$ is an $(S,S)$-pair where $\psi$ is not injective, then 
   \[ \Phi_{\langle P, \psi \rangle}(\op{X}) = 0.  \]
\end{enumerate}
\end{lem}
\begin{proof}
Parts (a) and (b) are special cases of a general, well-known result on the fixed points of $G$-sets with isotropy groups in a prescribed family of subgroups of $G$ that is closed under subconjugacy. Here one takes $S\times S$ for $G$ and the families are the groups of the form $\Delta(P,\psi)$ for any $\psi$ and for injective $\psi$, respectively. Part (c) is easy to prove for bisets and the general result follows by linearity.
\end{proof}

Thus we need to identify conditions under which an element $X \in \pComp{A(S,S)}$ that satisfies Frobenius reciprocity has $\Phi_{\langle  P,\psi \rangle}(X) = 0$ for every $(S,S)$-pair $(P,\psi)$ where $\psi$ is not injective. To this end we note the following consequences of Frobenius reciprocity.

\begin{lem} \label{lem:FrobFixPtsNotFree} Let $S$ be a finite $p$-group, and assume that $X \in \pComp{A(S,S)}$ satisfies Frobenius reciprocity. If $(P,\psi)$ is an $(S,S)$-pair such that $\Phi_{\langle  P,\psi \rangle}(X) \neq 0$,
then 
\begin{enumerate} 
 \item[(a)] $\Phi_{\langle  \psi(P),\incl \rangle}(X) = \Phi_{\langle  P,\psi \rangle}(X)$, and
 \item[(b)] $\Phi_{\langle  P,\psi \rangle}(\op{X}) = \Phi_{\langle  P,\incl \rangle }(X)$ .
\end{enumerate}
\end{lem}
\begin{proof} Part (a) is proved as in Lemma \ref{lem:FrobFixPts}, using Remark \ref{rmk:FrobFixPtsNotInj}. Part (b) is proved similarly by first showing that Frobenius reciprocity implies
\[  \Phi_{\langle P, \incl \rangle}(X) \cdot  \Phi_{\langle P, \psi \rangle}(X) = \Phi_{\langle  P,\psi \rangle}(\op{X}) \cdot  \Phi_{\langle P, \psi \rangle}(X) \,. \]
The left side of this equation is equal to $\Phi_{\langle P, \incl\negmedspace \times \psi \rangle}(X\times X \circ [S,\Delta])$, so it is enough to prove
\[ \Phi_{\langle P, \incl \negmedspace \times \psi \rangle}(X\times 1 \circ [S,\Delta]  \circ X) = \Phi_{\langle  P,\psi \rangle}(\op{X}) \cdot  \Phi_{\langle P, \psi \rangle}(X) \, . \]
It suffices to consider the case where $X$ is a biset, in which case we can look at actual fixed-point sets. By Lemma \ref{lem:DescSetProducts}, $X\times 1 \circ [S,\Delta]  \circ X$ is isomorphic to $Z \defeq X \times X$, with $(S,S\times S)$ action given by $(b_1,b_2) (x,y) a= (b_1xb_2^{-1},b_2ya)$. Now, the fixed-point set $Z^{(P, \incl \negmedspace \times \psi)}$ consists of pairs $(x,y) \in X \times X$ such that for all $a \in P$ we have
\[ ax\psi(a)^{-1} = x \quad \text{and} \quad ya = \psi(a) y \, . \]
The latter condition is equivalent to $y \in X^{(P,\psi)}$, and rewriting the former condition as $ax = x\psi(a)$, we see that it is equivalent to $x \in (\op{X})^{(P,\psi)}$. Hence $Z^{(P, \incl \negmedspace \times \psi)} = (\op{X})^{(P,\psi)} \times X^{(P,\psi)}, $ and the result follows.
\end{proof}

Lemma \ref{lem:FrobFixPtsNotFree} suggests a plan of attack for showing that $X \in \pComp{\Afree(S,S)}$. If we can show that $\Phi_{\langle  P,\incl \rangle }(X) \neq 0$ for every $P \leq S$, then part (b) of Lemma \ref{lem:FrobFixPtsNotFree}, combined with part (c) of Lemma \ref{lem:FixPtsOfFree}, implies that $ \Phi_{\langle  P,\psi \rangle}(X) \neq 0$ only when $\psi$ is injective. By part (b) of Lemma \ref{lem:FixPtsOfFree} this implies $X \in \pComp{\Afree(S,S)}$. Furthermore, by part (a) of Lemma \ref{lem:FrobFixPtsNotFree}, to show that $\Phi_{\langle  P,\incl \rangle }(X) \neq 0$ it is enough to show that $\Phi_{\langle  Q,\psi \rangle }(X) \neq 0$ for some $(S,S)$-pair $(Q,\psi)$ with $\psi(Q) = P$. We do this by a counting argument, dualizing Lemma \ref{lem:Cong}. For this we need a dual of the augmentation map.

\begin{defn} \label{defn:RightAug}
For a finite group $S$, let $\countS^R \colon \pComp{A(S,S)} \to \pComp{\Z}$ be the $\Zp$-linear map defined
on bisets by $\countS^R(X)= |X/S|$.
\end{defn}

Notice that for a generator $[P,\psi]$ of $A(S,S)$ we have $\countS([P,\psi]) = |S|/|P|$, while $\countS^R([P,\psi]) = |S|/|\psi(P)|$. Hence, for $X \in \pComp{\Afree(S,S)}$ we have $\countS(X) = \countS^R(X)$, but this is not true for general $X \in \pComp{A(S,S)}$.

\begin{lem} \label{lem:CongFromRight} Let $S$ be a finite group and let $X \in \pComp{A(S,S)}$. For $P \leq S$, let $\Sur(P)$ be the set of $(S,S)$-pairs $(Q,\psi)$ such that $\psi(Q) = P$, and let $\SurRep(P)$ be the set of conjugacy classes under the conjugacy relation $(Q,\psi) \sim (Q^x,\psi \circ c_x)$ for $x \in S$. Then 
\[ 
  \sum_{(Q,\psi) \in \SurRep(P)} \frac{\Phi_{ \langle Q,\psi \rangle }}{|C_S(Q)|} \equiv \countS^R(X)~ \modp
\]
where the sum runs over representatives of conjugacy classes in $\SurRep(P)$.
\end{lem}
\begin{proof} This follows from an argument similar to Lemma \ref{lem:FrobFixPtsNotFree}, but with the roles of the actions reversed. One first notes that it is enough to prove this congruence when $X$ is a biset. The difference from Lemma \ref{lem:FrobFixPtsNotFree} is that there we analyzed the $P$-fixed points of the right $S$-set $S\backslash X$, obtaining information about homomorphisms originating in $P$, but now we analyze the $P$-fixed points of the left $S$-set $X/S$, obtaining information about homomorphisms with image $P$. The generalization to non-injective maps comes about because we are not assuming that $X$ is right-free. 
\end{proof}

We now have all the ingredients to prove the first result relaxing the right freeness condition in Theorem \ref{thm:FrobImpliesSat}.

\begin{prop} \label{prop:RightAugRightFree}
Let $S$ be a finite group, and let $X$ be an element in $\pComp{A(S,S)}$. If $X$ satisfies Frobenius reciprocity 
and $\countS^R(X)$ is not divisible by $p$, then $X \in \pComp{\Afree(S,S)}$. In particular $\countS(X) = \countS^R(X)$,
and Theorem \ref{thm:FrobImpliesSat} applies to $X$.
\end{prop}
\begin{proof}
The condition on $\countS^R(X)$ and Lemma \ref{lem:CongFromRight} imply that for every $P \leq S$ there exists an  $(S,S)$-pair $(Q,\psi)$ with $\psi(Q) = P$ such that $\Phi_{\langle  Q,\psi \rangle }(X) \neq 0$. Hence the ``plan of attack'' described after Lemma \ref{lem:FrobFixPtsNotFree}  can be used to prove that $X \in \pComp{\Afree(S,S)}$. The last two claims follow.
\end{proof}

Proposition \ref{prop:RightAugRightFree} can sometimes be useful in certain situations, but stable homotopy is not among them as the right orbit map $\countS^R$ translates no better to stable homotopy than the right freeness condition. However, with further work we can use Proposition \ref{prop:RightAugRightFree} to prove a result that translates better to stable homotopy. For this we need the following definition.

\begin{defn}
Let $S$ be a finite group. The \emph{Nishida ideal} $J(S) \subset A(S,S)$ is the $\Z$-submodule generated by elements $[P,\psi]$, where $\psi(P) < S$. An element $X \in \pComp{A(S,S)}$ is \emph{dominant} if $X \notin \pComp{J(S)}$.
\end{defn} 
The double coset formula readily shows that $J(S)$ is a two-sided ideal of $A(S,S)$.

\begin{lem} \label{lem:DomFrobImpliesFree}
Let $S$ be a finite $p$-group, and let $X$ be an element of $\pComp{A(S,S)}$. If $X$ is dominant and satisfies Frobenius reciprocity, then, for every  non-injective homomorphism $\psi \colon S \to S$, we have $c_{ \langle S,\psi \rangle }(X) = 0 $. In particular $\countS^R(X) \equiv \countS(X)~ \modp$.
\end{lem}
\begin{proof}
First note that for a homomorphism $\psi \colon S \to S$, Lemma \ref{lem:FixedPtsOnBasis} implies
\begin{equation} \label{eq:PhiAndc} 
  \Phi_{ \langle S,\psi \rangle }(X) = |Z(S)| \cdot c_{ \langle S,\psi \rangle }(X). 
\end{equation} 
Hence $c_{ \langle S,\psi \rangle }(X) \neq 0$ if and only if $\Phi_{ \langle S,\psi \rangle }(X) \neq 0$.
 
Now, since $X$ is dominant, there exists an automorphism $\varphi \in \Aut(S)$ such that $c_{ \langle S,\varphi \rangle }(X) \neq 0$, and hence $\Phi_{ \langle S,\varphi \rangle }(X) \neq 0$. Lemma \ref{lem:FrobFixPtsNotFree} then implies that $\Phi_{ \langle S,\incl \rangle }(X) \neq 0$. If $\psi \colon S \to S$ is a group homomorphism such that $\Phi_{ \langle S,\psi \rangle }(X) \neq 0$, then Lemma \ref{lem:FrobFixPtsNotFree} implies that $\Phi_{ \langle S,\psi \rangle }(\op{X}) = \Phi_{ \langle S,\incl \rangle }(X) \neq 0$, and by Lemma \ref{lem:FixPtsOfFree} this means that $\psi$ must be injective. Thus we conclude that for non-injective homomorphisms $\psi \colon S \to S$ we have $c_{ \langle S,\psi \rangle }(X) = 0 $.

For an $(S,S)$-pair $(P,\psi)$ we have $\countS([P,\psi]) = |S|/|P|$ and $\countS^R([P,\psi]) = |S|/|\psi(P)|$. It follows that $\countS(X)$ is congruent $\modp$ to the sum of coefficients $c_{ \langle S,\psi \rangle }(X)$ where $\psi$ runs over conjugacy classes of homomorphisms $S \to S$, while $\countS^R(X)$ is congruent $\modp$ to the sum of coefficients $c_{ \langle S,\varphi \rangle }(X)$ where $\varphi$ runs over conjugacy classes of automorphisms of $S$. Since $c_{ \langle S,\psi \rangle }(X) = 0$ for non-injective $\psi$, these sums are the same and we have $\countS^R(X) \equiv \countS(X)~ \modp$. 
\end{proof}

We can now prove a stable-homotopy-friendly version of Theorem \ref{mthm:FrobImpliesSat}, replacing right freeness by a dominance condition.

\begin{thm}\label{thm:DomFrobImpliesSat}
Let $S$ be a finite $p$-group, and let $X$ be a dominant element of $\pComp{A(S,S)}$ that satisfies Frobenius reciprocity. If $\countS(X)$ is not divisible by $p$, then $\FusRStab(X)$ is a saturated fusion system and $X$ is a right-characteristic element for $\FusRStab(X)$.
\end{thm}
\begin{proof}
By Lemma \ref{lem:DomFrobImpliesFree}, $\countS^R(X)$ is not divisible by $p$. The result now follows from Proposition \ref{prop:RightAugRightFree} and Theorem \ref{thm:FrobImpliesSat}.
\end{proof}

When working with dominant idempotents we can even remove the augmentation condition.

\begin{cor}
Let $S$ be a finite $p$-group, and let $\omega$ be a dominant idempotent in $\pComp{A(S,S)}$. If $\omega$ satisfies Frobenius reciprocity, then $\omega$ is a characteristic idempotent for  $\FusStab(\omega)$, which is saturated. 
\end{cor}
\begin{proof}
The result follows from Theorem \ref{thm:DomFrobImpliesSat} if we can show that $\omega$ has augmentation not divisible by $p$. We have 
 \[ \countS(\omega) \equiv \sum_{\psi \in W} c_{ \langle S,\psi \rangle }(\omega) \quad  \modp, \]
where $W$ is the set of conjugacy classes of homomorphisms $\psi \colon S \to S$ with $c_{ \langle S,\psi \rangle }(\omega) \neq 0 $. By Lemma \ref{lem:DomFrobImpliesFree} we have $W \subseteq \Out(S)$. Using \eqref{eq:PhiAndc}, Lemma \ref{lem:FrobFixPtsNotFree} then implies that for all $\varphi \in W$, we have $c_{ \langle S,\varphi \rangle }(\omega) = c_{ \langle S,\incl \rangle }(\omega) $, and hence
 \[ \countS(\omega) \equiv |W| \cdot c \quad  \modp \]
where $c = c_{ \langle S,\incl \rangle }(\omega)$ is a nonzero constant.

Now consider the projection 
 \[ \pi \colon \pComp{A(S,S)} \twoheadrightarrow \pComp{A(S,S)} / \pComp{J(S)} \cong \Zp \Out(S). \]
This is a homomorphism of $\Zp$-algebras, so $\pi(\omega) = \sum_{\varphi \in W} c \cdot \varphi$ is an idempotent in $\Zp \Out(S)$ with augmentation $|W| \cdot c$. Hence $|W| \cdot c$ equals $0$ or $1$. The former would imply that $|W|$ or $c$ is zero, contradicting dominance, and hence $|W| \cdot c = 1$. Consequently $\countS(\omega)$ is not divisible by $p$, and the result follows from Theorem \ref{thm:DomFrobImpliesSat}. (In fact $\countS(\omega) = 1$ since $\omega$ is idempotent.)
\end{proof}

We now get a stable-homotopy-friendly version of the bijection in Corollary \ref{cor:Bijection}.

\begin{cor} \label{cor:BijectionForStable}
For a finite $p$-group $S$ there is a bijective correspondence between saturated fusion systems on $S$ and dominant idempotents in $\pComp{A(S,S)}$ that satisfy Frobenius reciprocity. The bijection sends a saturated fusion system to its characteristic idempotent, and an idempotent to its stabilizer fusion system.
\end{cor}

\subsection{Pointed classifying spectra of saturated fusion systems} \label{subsec:ClSpectra}
When Linckelmann--Webb defined characteristic elements, part of their motivation was to construct classifying spectra for fusion systems. More precisely, if $\Omega$ is a characteristic element for a saturated fusion system $\F$ on a finite $p$-group $S$, then $\alpha(\Omega)$ is a stable selfmap of $BS_+$ that in $\Fp$-homology is idempotent up to scalar with image the $\F$-stable elements in $\widetilde{H}_*(BS_+,\Fp)$ (see \cite{BLO2}), and hence the mapping telescope of $\alpha(\Omega)$ is a stable summand of $BS_+$ that can be regarded as a classifying spectrum for $\F$. 

In \cite{KR:ClSpectra} more control was exercised over this construction by introducing the characteristic idempotent of a fusion system rather than using an arbitrary characteristic element. This resulted in a functorial assignment of classifying spectra for fusion systems that admits a transfer theory, and is consistent with the (unstable) homotopy theory of fusion system from \cite{BLO2}. However, an unfortunate choice was made in \cite{KR:ClSpectra} by constructing the classifying spectrum as a stable summand of $BS$, rather than $BS_+$, in order to bypass technicalities arising from basepoint issues. In this subsection we recall the construction and some basic properties of classifying spectra for saturated fusion systems, and also use the opportunity to remedy the oversight of \cite{KR:ClSpectra} by introducing ``pointed'' classifying spectra in this account.

We start by observing that if $\F$ is a saturated fusion system on a finite $p$-group $S$, then the characteristic idempotent $\omega_\F$ can, by Proposition \ref{prop:IAndpComp}, be regarded as an element of $\IComp{A(S,S)}$, since it has augmentation 1. Hence there is a corresponding stable map $\tilde{\omega}_\F \defeq \alpha(\omega_\F) \colon \Stable BS_+ \to \Stable BS_+$, which we call the \emph{stable characteristic idempotent.} The \emph{classifying spectrum} of $\F$ is defined as the stable summand carved out of $BS_+$ by $\tilde{\omega}_\F$ via the standard mapping telescope construction
 \[ \ClSpectrum{\F}_+ 
    \defeq \Tel(\tilde{\omega}_\F )
    %\defeq \tilde{\omega}_\F \Stable{BS_+} 
    \defeq \operatorname{HoColim} \left( \Stable BS_+ \xrightarrow{\tilde{\omega}_\F} \Stable BS_+ \xrightarrow{\tilde{\omega}_\F} \cdots \right) \, .\]
We denote the structure map of the homotopy colimit by $\sigma_\F \colon \Stable{BS_+} \to \ClSpectrum{\F}_+ $ and refer to it as the \emph{structure map} of the classifying spectrum. There is a unique (up to homotopy) map $t_\F \colon \ClSpectrum{\F}_+ \to \Stable{BS_+}$ such that $\sigma_\F \circ t_\F \simeq 1_{\ClSpectrum{\F}_+}$ and $t_\F \circ \sigma_\F \simeq \tilde{\omega}_\F$, to which we refer as a \emph{transfer map}. Classifying spectra are functorial with respect to fusion-preserving homomorphisms. A fusion-preserving monomorphism also gives rise to a transfer between classifying spectra, and this transfer construction is functorial.

Restricting the diagonal map $\Delta_S$ of $\Stable{BS_+}$ to $\ClSpectrum{\F}_+$, one obtains a map 
\[ \Delta_\F \defeq (\sigma_\F \wedge \sigma_\F) \circ \Delta_S \circ t_\F \colon \ClSpectrum{\F}_+ \to \ClSpectrum{\F}_+ \wedge \ClSpectrum{\F}_+ \, .\]
The Frobenius reciprocity relation for $\omega_\F$ implies that $\Delta_\F$ is coassociative up to homotopy, so we can think of $\Delta_\F$ as a homotopy diagonal map of $\ClSpectrum{\F}_+$. Frobenius reciprocity for $\omega_\F$ also implies 
 \[ \Delta_\F \circ \sigma_\F \simeq (\sigma_\F \wedge \sigma_\F) \circ \Delta_S, \] \
and the Frobenius reciprocity relation
 \[ (1_{\ClSpectrum{\F}_+}\wedge t_\F) \circ \Delta_{\F} \simeq (\sigma_\F \wedge 1_{\Stable{BS_+}}) \circ \Delta_{S} \circ t_\F. \]

As a consequence of the universal stable element theorem (Theorem \ref{thm:UnivStable}), one can determine the homotopy classes of stable maps to or from classifying spectra of saturated fusion systems.
 
\begin{cor}[\cite{KR:ClSpectra}]  \label{cor:MapsToBF}
Let $\F$ be a saturated fusion system on a finite $p$-group $S$, and let $E$ be any spectrum. The maps
 \[  E^*(\sigma_{\F}) \colon E^*(\PtdClSpectrum{\F}) \longrightarrow  E^*(\Stable BS_+) 
     \quad \text{and} \quad 
     E_*(t_{\F}) \colon E_*(\PtdClSpectrum{\F}) \longrightarrow  E_*(\Stable BS_+)  \]
are split injections with image the $\F$-stable elements in $E^*(\Stable BS_+) $ and $E_*(\Stable BS_+)$, respectively. Furthermore, if $E^*$ is a ring spectrum, then $E^*(\sigma_{\F})$ is a map of algebras, and $E^*(t_{\F})$ is a map of $ E^*(\PtdClSpectrum{\F})$-modules.
\end{cor} 
\begin{proof}
From $\sigma_{\F} \circ t_{\F} \simeq 1_{\PtdClSpectrum{\F}}$ and 
$t_{\F} \circ \sigma_{\F} \simeq \tilde{\omega}_{\F}$ one deduces 
that $E^*(\sigma_{\F})$ is a split injection with image 
$E^*(\tilde{\omega}_{\F})E^*(\Stable BS_+)$, which, by the universal stable 
element theorem, consists of the $\F$-stable elements in 
$E^*(\Stable BS_+)$. The homological statement is proved analogously. 
The last claim follows from Frobenius reciprocity.
\end{proof} 
 
In particular, applying Corollary \ref{cor:MapsToBF} twice, one obtains 
a description of the group of homotopy classes of stable maps between
classifying spectra.
\begin{cor} \label{cor:[BF,BF]}
If $\F_1$ and $\F_2$ are saturated fusion systems on finite $p$-groups $S_1$ and $S_2$, respectively, then the map
 \[ [\PtdClSpectrum{\F_1},\PtdClSpectrum{\F_2}] \longrightarrow [\Stable {BS_1}_+,\Stable {BS_2}_+]\, , 
 \quad \quad f \longmapsto t_{\F_2} \circ f \circ \sigma_{\F_1}, \]
is a split injection with image the $(\F_1,\F_2)$-stable maps in $[\Stable {BS_1}_+,\Stable {BS_1}_+]$.
\end{cor}

Corollary \ref{cor:[BF,BF]} can be regarded as a generalization of the Segal conjecture to fusion systems. We strengthen this analogy by reformulating the result as a completion theorem. This was done independently by Diaz--Libman in \cite{DL} in the case where $S_2 =1$, so 
%$\F_2$ is the trivial fusion system, implying that $\PtdClSpectrum{\F_2} = \PtdStable{B1} = \SphereSpectrum$, and 
$[\PtdClSpectrum{\F_1},\PtdClSpectrum{\F_2}] = [\PtdClSpectrum{\F_1},\SphereSpectrum]$ is the zeroth cohomotopy group of $\PtdClSpectrum{\F_1}$, which matches Segal's original formulation. To this end, we define $A(\F_1,\F_2)$ to be the set of $(\F_1,\F_2)$-stable elements in $A(S_1,S_2)$, meaning elements that are left $\F_2$-stable and right $\F_1$-stable. Observe that $A(\F_1,\F_2)$ is a $\Z$-submodule of $A(S_1,S_2)$, but not an $A(S_1)$-submodule. However, setting $A(\F_1) \subseteq A(S_1)$ to be the subring of $\F_1$-stable elements, the $A(S_1)$-action on $A(S_1,S_2)$ restricts to an $A(\F_1)$-action on $A(\F_1,\F_2)$. The augmentation of $A(S_1)$ restricts to an augmentation of $A(\F_1)$, and we write $I(\F_1)$ for the augmentation ideal.

\begin{cor} \label{cor:FusionSegal}
If $\F_1$ and $\F_2$ are saturated fusion systems on finite $p$-groups $S_1$ and $S_2$, respectively, then the map
\[ A(\F_1,\F_2) \longrightarrow [\PtdClSpectrum{\F_1},\PtdClSpectrum{\F_2}] \, , \quad \quad X \mapsto \sigma_{\F_2} \circ \alpha(X) \circ t_{\F_1}  \]
is an $I(\F_1)$-adic completion map.
\end{cor}
\begin{proof}
By Corollary \ref{cor:[BF,BF]}, $[\PtdClSpectrum{\F_1},\PtdClSpectrum{\F_2}]$ is isomorphic to the submodule of $(\F_1,\F_2)$-stable maps in $[\Stable {BS_1}_+,\Stable {BS_1}_+]$. The claim is proved by showing that we can identify $A(\F_1,\F_2)_{I(\F_1)}^{\wedge}$ with the submodule of $(\F_1,\F_2)$-stable elements in $A(S_1,S_2)_{I(S_1)}^{\wedge}$. By Proposition \ref{prop:IAndpComp}, 
$A(S_1,S_2)_{I(S_1)}^{\wedge}$ can be identified with the submodule of $\pComp{A(S_1,S_2)}$ consisting of elements with integer augmentaion, and a similar argument allows us to identify $A(\F_1,\F_2)_{I(\F_1)}^{\wedge}$ with the submodule of $\pComp{A(\F_1,\F_2)}$ consisting of elements with integer augmentation. Thus it suffices to show that $\pComp{A(\F_1,\F_2)}$ consists of the $(\F_1,\F_2)$-stable elements in $\pComp{A(S_1,S_2)}$.

As $(\F_1,\F_2)$-stability is determined by a system of linear equations, the set of $(\F_1,\F_2)$-stable elements form a closed subspace in $\pComp{A(S_1,S_2)}$. It follows that every element in $\pComp{A(\F_1,\F_2)}$ is $(\F_1,\F_2)$-stable.

Conversely, suppose that $X \in \pComp{A(S_1,S_2)}$ is a $(\F_1,\F_2)$-stable element. Then $X$ is the limit of a $p$-adic Cauchy sequence $(X_n)$ in $A(S_1,S_2)$. Recall from \cite{KR:ClSpectra} that $\omega_{\F_1}$ is the limit of a $p$-adic Cauchy sequence $(\Omega_{1,n})$ in $A(S_1)$, where every $\Omega_{1,n}$ is a characteristic element for $\F_1$, and similarly $\omega_{\F_2}$ is the limit of a Cauchy sequence $(\Omega_{2,n})$ of elements in $A(S_2)$ that are characteristic for $\F_2$. Now consider the sequence $(Y_n) = (\Omega_{2,n} \circ X_n \circ \Omega_{1,n})$. As $ \Omega_{1,n}$ and $ \Omega_{2,n}$ are $\F_1$-stable and $\F_2$-stable, respectively, this is a sequence in $A(\F_1,\F_2)$. As composition is continuous in the $p$-adic topology, we deduce that $(Y_n)$ is a Cauchy sequence with limit $\omega_{\F_1} \circ X \circ \omega_{\F_2} = X$, so $X \in \pComp{A(\F_1,\F_2)}$.
\end{proof}
 
A $\Zp$-basis for $\pComp{A(\F_1,\F_2)}$ is given in \cite[Proposition 5.2]{KR:ClSpectra}, and from this one can deduce the $\Z$-module structure of $\IComp{A(\F_1,\F_2)}$.

The theory of $p$-local finite groups was developed by Broto--Levi--Oliver in \cite{BLO2} as a framework for classifying spaces of saturated fusion systems. The definition of a $p$-local finite group will be recalled in Section \ref{sec:RTT}. For now it suffices to explain that if $\F$ is a saturated fusion system on a finite $p$-group $S$, then a classifying space for $\F$ is modeled by a \emph{centric linking system} $\Link$, a category with objects the $\F$-centric subgroups of $S$ and where the morphism set $\Mor_\Link(P,Q)$ is a free $Z(P)$-module with quotient $\Mor_\Link(P,Q) / Z(P) \cong \Hom_\F(P,Q)$. The \emph{classifying space} is then given by $\ClSp$, and comes equipped with a natural map $\theta \colon BS \to \ClSp$. The triple $(S,\F,\Link)$ is referred to as a \emph{$p$-local finite group}, and can be thought of as a saturated fusion system with a chosen classifying space.

The following proposition describes the compatibility of classifying spectra with classifying spaces. In practice it allows us to identify $\Stable{\theta} \colon \PtdStable{BS} \to \PtdStable{\ClSp}$ with $\sigma_\F \colon \PtdStable{BS} \to \PtdClSpectrum{\F}$. 
\begin{prop}[\cite{KR:ClSpectra}]
For a $p$-local finite group $(S,\F,\Link)$, there is a unique (up to homotopy) map $h \colon \PtdClSpectrum{\F} \to \PtdStable{\ClSp}$ such that $h \circ \sigma_\F \simeq \theta$, and this $h$ is a homotopy equivalence.
\end{prop}

%%%%%%%%%%%%%%%%%%%%%%%%%%%%%%%%%%%%%%%%%%%
\section{Applications to stable splittings} \label{sec:Splittings}
%%%%%%%%%%%%%%%%%%%%%%%%%%%%%%%%%%%%%%%%%%%
The stable splitting of $p$-completed classifying spaces has been studied intensively by many authors, most notably by Martino--Priddy in \cite{MP2} and Benson--Feshbach in \cite{BenFes}. A good overview of the subject was given by Benson in \cite{Be:StablySplittingBG}. Let $G$ be a finite group with Sylow subgroup $S$. By a simple transfer argument, $\pComp{BG_+}$ is a stable summand of $BS_+$. Thus the stable splitting of $\pComp{BG_+}$ can be described by first determining the complete stable splitting of $BS_+$, and then determining how many copies of each stable summand of $BS_+$ can be found in a stable splitting of $\pComp{BG_+}$. This is done in \cite{MP2} and \cite{BenFes}.

Given an idempotent in $e \in \{BS_+,BS_+\}$, the mapping telescope $\Tel(e)$ is a stable summand of $BS_+$. This gives a correspondence between (homotopy types of) stable summands of $BS_+$ and (conjugacy classes of) stable idempotents of $BS_+$, which, by the Segal conjecture, correspond to (conjugacy classes of) idempotents in $\pComp{A(S,S)}$. Under this correspondence, an indecomposable summand corresponds to an irreducible idempotent. Thus a complete stable splitting of $BS$ corresponds to a decomposition of the identity in $\pComp{A(S,S)}$ as an orthogonal sum of irreducible idempotents. In particular, since the double Burnside ring satisfies the Krull--Schmidt theorem, one obtains a uniqueness result for the complete stable splitting of $BS_+$.

Stewart Priddy has asked how one can combine indecomposable stable summands of $BS_+$ together to obtain a stable summand of $BS_+$ that has the stable homotopy type of $\pComp{BG_+}$ for some finite group $G$ with Sylow subgroup $S$. Using the results of this paper, we can give an answer to this question. To do this, we first need to broaden the scope of the question to incorporate fusion systems. Classifying spectra of saturated fusion systems (see \ref{subsec:ClSpectra}) possess all the important stable homotopy-theoretic properties of $p$-completed classifying spaces of finite groups. Thus, instead of asking how we can group indecomposable summands together to obtain the stable homotopy type of a $p$-completed classifying space, we ask how we can obtain the classifying spectrum of a saturated fusion system.

The classifying spectrum of a saturated fusion system $\F$ on $S$ is the stable summand of $BS_+$ corresponding to the characteristic idempotent of $\F$. Thus a reformulation of Priddy's question is to ask when an idempotent in $\pComp{A(S,S)}$ can be the characteristic idempotent of a saturated fusion system. The following result answers that question.
\begin{cor} \label{cor:StableSplittings}
Let $S$ be finite $p$-group. An idempotent in $\pComp{A(S,S)}$ is the characteristic idempotent of a saturated fusion system on $S$ if and only if it is dominant and satisfies Frobenius reciprocity.
\end{cor}
\begin{proof} This is a direct reformulation of Corollary \ref{cor:BijectionForStable}.
\end{proof}

Getting closer to Priddy's original question about groupings of indecomposable summands, we can consider, given a specified irreducible idempotent $e$ in $\pComp{A(S,S)}$, which other irreducible idempotents must occur as summands of a characteristic idempotent $\omega$ in which $e$ appears. Answering this would involve analyzing the idempotents ``detecting'' $(e \times e) \circ [S,\Delta]$ and $(e \times 1) \circ [S,\Delta] \circ e$. There is much interesting work to do here. 

There is a subtle, but significant, difference between the way in which Priddy phrased his original question, and the answer given here, even beyond the generalization to fusion systems. Sticking with the fusion formulation, Priddy's question is when the wedge sum  of a collection of indecomposable stable summands of $BS$ has the \emph{homotopy type} of a $\PtdClSpectrum{\F}$, while our answer is closer to saying when the wedge sum ``is'' $\PtdClSpectrum{\F}$. This really amounts to asking when an idempotent in $\pComp{A(S,S)}$ is conjugate to a characteristic idempotent rather than asking when it \emph{is} a characteristic idempotent. See \cite{KR:NSC} for a discussion on the difference between the information contained in an idempotent in $\pComp{A(S,S)}$ and the information contained in the homotopy class of the corresponding stable summand of $BS_+$.

%%%%%%%%%%%%%%%%%%%%%%%%%%%%%%%%%%%%%%%%
\section{On the Adams--Wilkerson theorem} \label{sec:AW}
%%%%%%%%%%%%%%%%%%%%%%%%%%%%%%%%%%%%%%%%%%%%
In their celebrated paper \cite{AW}, Adams--Wilkerson developed and studied Galois theory for even-graded integral algebras over the $\modp$ Steenrod algebra $\Ap$. Their motivation was to determine which polynomial algebras can appear as the cohomology of a space, and in particular they determined when an even-graded integral ring over the Steenrod algebra can be realized as a ring of invariants in a polynomial ring with generators of degree 2. Following ideas of Lannes, Goerss--Smith--Zarati described in \cite{GSZ} how the $\Fp$-cohomology of the classifying space of an elementary abelian $p$-group is controlled by its evenly graded part, which is a polynomial ring with generators in degree 2. Thus the work of Adams--Wilkerson can be applied to the cohomology of elementary abelian $p$-groups, yielding the following variant of their result.

\begin{thm}[\cite{AW, GSZ, KR:Transfers&plfgs}]\label{thm:AW}
Let $V$ be a finite, elementary abelian $p$-group, put $H^* \defeq H^*(BV;\Fp)$, regarded as an $\Ap$-algebra, and let $f \colon R^* \to H^*$ be the inclusion of a $\Ap$-subalgebra, making $H^*$ a finite $R^*$-algebra. There exists a subgroup $W \leq \Aut(V)$ of order prime to $p$ such that $R^* = (H^*)^W$ if and only if, there exists an $R^*$-linear map of $\Ap$-modules  $t^* \colon H^* \to R^*$ such that $t \circ f = 1_{R^*}$.
\end{thm}
\begin{proof}
The result is ultimately a consequence of \cite{AW}. It is shown in \cite{GSZ} how to apply the work from \cite{AW} to the setting of the cohomology of elementary abelian groups using properties of reduced $\mathcal{U}$-injectives, and the proof of \cite[Proposition 3.11]{KR:Transfers&plfgs} shows how the $R^*$-linearity implies the conditions in \cite[Theorem 1.3]{GSZ}, yielding the desired result.
%This is shown implicitly in \cite{GSZ} and an explicit proof is given in \cite{KR:Transfers&plfgs}.
\end{proof}

\begin{rmk} In the setting of Theorem \ref{thm:AW} one can think of $f$ as the restriction map $H^*(BG;\Fp) \to H^*(BV;\Fp)$, where $G = W \ltimes V$ is the semidirect product, and $t$ as a normalized transfer map. The $R^*$-linearity condition on $t$ is then the usual Frobenius reciprocity relation in cohomology.
\end{rmk}

The results in this paper can be used to prove a more general analogue of Theorem \ref{thm:AW}. ``More general'' means that we allow general finite $p$-groups $S$ instead of just elementary abelian ones, and we recognize a ring of stable elements with respect to a saturated fusion system on $S$ instead of just considering rings of invariants with respect to subgroups of $W \leq \Out(S)$. (The latter is in fact the ring of stable elements with respect to the fusion system associated to the semi-direct product $W \ltimes S$ when $S$ is abelian.) ``Analogue'' means that we lift the statement from cohomology to the $p$-localized double Burnside ring. This should not be surprising as cohomology is not a sufficiently powerful invariant to distinguish between group homomorphisms unless the groups involved are elementary abelian. Some care is needed to reformulate Theorem \ref{thm:AW} in terms of the double Burnside ring, and this is explained in the following paragraphs.

First, the action of the Steenrod algebra plays a crucial part in \cite{AW}, and consequently in the proof of Theorem \ref{thm:AW}. Recall that the cohomology group $H^n(BV;\Fp)$ is naturally isomorphic to the graded group $[\Stable{BV_+},\Sigma^* H\Fp]$, where $H\Fp$ is the Eilenberg--MacLane spectrum with $\Fp$-coefficients, and $[-,-]$ denotes homotopy classes of maps between spectra. The Steenrod algebra can similarly be thought of as the graded group $[H\Fp, \Sigma^* H\Fp]$. Thus asking that $f$ and $t$ preserve the action of the Steenrod algebra corresponds to asking that they be natural with respect to selfmaps of $H\Fp$. (This is not quite precise as a priori we do not know that $R^*$ is realized by a spectrum, but this point of view is helpful to explain the analogy.) Moving to the double Burnside ring we replace the functor $H^*(B-;\Fp)$ by the functor $A(-,S_2)\pLoc$, now applied to a general finite $p$-group $S_1$, with the understanding that $S_1 = S_2 = S$, subscripts being used to notationally distinguish the different roles of $S$. Now we demand naturality with respect to selfmaps of $S_2$, so we replace the Steenrod algebra by the double Burnside ring $A(S_2,S_2)\pLoc$. More precisely, $A(S_1,S_2)\pLoc$ is a left $A(S_2,S_2)\pLoc$-module by composition, and we will consider the inclusion of an $A(S_2,S_2)\pLoc$-submodule $f \colon R \to A(S_1,S_2)\pLoc$ with a morphism of $A(S_2,S_2)\pLoc$-modules $t \colon  A(S_1,S_2)\pLoc \to R$ such that $t \circ f = 1.$

Second, the ring structure of $H^*(BV;\Fp)$ also plays an important part in Theorem \ref{thm:AW} and we must determine the analogous structure for the double Burnside ring. The double Burnside ring does, of course, have a ring structure of its own, obtained by composition. However this plays a very different role from the ring structure in cohomology. The latter is obtained as the composite
 \[ H^*(BV;\Fp) \otimes_{\Fp} H^*(BV;\Fp) \xrightarrow{\cong} H^*(BV\times BV;\Fp) \xrightarrow{\Delta^*} H^*(BV;\Fp), \]
where the first map is the K\"unneth isomorphism, and the second is induced by the diagonal map $BV \to BV \times BV$. This does not carry directly over when lifting to the double Burnside ring, as there is no K\"unneth isomorphism in this case. Instead we look at the composite 
% \[ A(S_1,S_2)\pLoc \times  A(S_1,S_2)\pLoc \xrightarrow{-\times -} A(S_1 \times S_1, S_2 \times S_2)\pLoc \xrightarrow{-\circ [S_1,\Delta]} A(S_1, S_2 \times S_2)\pLoc, \]
\[ \kappa_A \colon A(S_1,S_2) \times  A(S_1,S_2) \xrightarrow{-\times -} A(S_1 \times S_1, S_2 \times S_2) \xrightarrow{-\circ [S_1,\Delta]} A(S_1, S_2 \times S_2), \]
where the first map is induced by Cartesian product, and the second map is induced by the diagonal $\Delta \colon S_1 \to S_1 \times S_1$. We will in fact work with the $p$-localization of $\kappa_A$, which we also denote by $\kappa_A$. To reduce notation, we adopt the following conventions:
\begin{eqnarray} \label{eq:ANotation}
  A_{1,2}    &\defeq &A(S_1,S_2)\pLoc \notag \\
  A_{2,2}  &\defeq &A(S_2,S_2)\pLoc \notag \\
  A_{1,22}   &\defeq &A(S_1,S_2 \times S_2)\pLoc \\
  A_{2,22} &\defeq &A(S_2,S_2 \times S_2)\pLoc \notag \\
  A_{22,22}  &\defeq &A(S_2\times S_2,S_2 \times S_2)\pLoc \notag 
\end{eqnarray}
Of course $A_{1,2} = A_{2,2}$ and $A_{1,22} = A_{2,22}$, but we distinguish them notationally as they appear in different roles, allowing us to think of $A_{1,2}$ and $A_{1,22}$ as modules where elements live, while $A_{2,2}$, $A_{2,22}$ and $A_{22,22}$ induce actions. We regard $A_{1,22}$ as a left   $(A_{2,2} \times A_{2,2})$-module by setting $(a,b)\cdot z \defeq (a \times b) \circ z$ for $a,b \in A_{2,2}$ and $z \in A_{1,22}$. With this convention we have a morphism of $A_{2,2} \times A_{2,2}$-modules
 \[ \kappa_A \colon A_{1,2} \times A_{1,2} \to A_{1,22}.\] 

Having replaced the ring structure in cohomology with $\kappa_A$ for the double Burnside ring, we next replace
the notion of a subring by the notion of a \emph{$\kappa$-preserving submodule}. This consists of an inclusion of left $A_{2,2}$-modules 
 $f \colon R \to A_{2,2} $
and an inclusion of left $A_{22,22}$-modules (in particular an inclusion of $A_{2,2} \times A_{2,2}$-modules)
 $f' \colon R' \to A_{1,22}$
equipped with a morphism of $A_{2,2} \times A_{2,2}$-modules
 \[ \kappa_R \colon R \times R \to R',\]
such that 
\begin{enumerate}
 \item[($\kappa1$)] 
            Every $h \in A_{2,22}$ induces a commutative square 
            \[ \xymatrix{ A_{1,2} \ar[rr]^{h_* = h \circ -}  && A_{1,22} \\ 
                          R \ar[u]^{f} \ar[rr]^{h_*} && R'. \ar[u]^{f'} }
            \] 
            (In other words, the top map restricts to the bottom map.)
 \item[($\kappa2$)] %The pair $(f,f')$ respects $\kappa$, in the sense that 
                  $f' \circ \kappa_R = \kappa_A \circ (f\times f)$.
\end{enumerate}
Condition ($\kappa1$) is a naturality condition. It can be rephrased to say that we have a $p$-local functor $\rho$ defined on the full subcategory of the $p$-localized Burnside category ${\Burnside}\pLoc$ with objects $S_2$ and $S_2 \times S_2$ given by $\rho(S_2) = R$ and $\rho(S_2 \times S_2) = R'$, and that the pair $(f,f')$ constitutes a natural transformation from $\rho$ to ${A(S_1,-)}\pLoc$. Condition ($\kappa2$) replaces the hypothesis in Theorem \ref{thm:AW} that $f$ preserves multiplication.

The appropriate replacement for the $R^*$-linear retraction map $t^*$ in Theorem \ref{thm:AW} is then formulated as follows. Given a $\kappa$-preserving submodule $(R,R')$, a \emph{retractive transfer} for $(R,R')$ consists of a morphism of $A_{2,2}$-modules $t \colon A_{1,2} \to R$ such that $f \circ t = \id_{R}$, and a morphism of left $A_{22,22}$-modules $t' \colon A_{1,22} \to R'$ such that $f' \circ t' = \id_{R'}$, satisfying
\begin{enumerate}
  \item[($\kappa3$)] Every $h \in A_{2,22}$ induces a commutative square 
            \[ \xymatrix{ A_{1,2} \ar[d]^{t} \ar[rr]^{h_* = h \circ -}  && A_{1,22} \ar[d]^{t'} \\ 
                          R  \ar[rr]^{h_*} && R'.  }
            \] 
  \item[($\kappa4$)] The following diagram commutes:
  \[ \xymatrix{ 
     R \times A_{1,2} \ar[d]^{\id \times t} \ar[r]^{f \times \id } & A_{1,2} \times A_{1,2} \ar[r]^{\kappa_A} & A_{1,22} \ar[d]^{t'} \\
     R \times R \ar[rr]^{\kappa_R} && R'. 
     }
  \] 
\end{enumerate}
Condition ($\kappa3$) is a naturality condition similar to ($\kappa1$), and can be rephrased as saying that $(t,t')$ constitutes a natural transformation that is a retract of $(f,f')$. Condition ($\kappa4$) replaces the $R^*$-linearity condition on $t^*$ in Theorem \ref{thm:AW}.

Finally, the finiteness condition on the ring extension $f^* \colon R^* \to H^*$ in Theorem \ref{thm:AW} is replaced by the condition that $R$ should not be contained in the Nishida ideal $J(S)\pLoc$ of $A_{1,2}$. 

\begin{thm}\label{thm:AWAnalogue}
Let $S$ be a finite $p$-group, and let $(R,R')$ be a $\kappa$-preserving submodule of $A(S,S)\pLoc$ such that $R$ is not contained in $J(S)\pLoc$. There exists a saturated fusion system $\F$ on $S$ such that $R$ is the ring of right $\F$-stable elements in $A(S,S)\pLoc$ and $R'$ is the module of right $\F$-stable elements in $A(S,S \times S)$ if and only if $(R,R')$ has a retractive transfer.
\end{thm}
\begin{proof} To reduce confusion, we will write $x \cdot y$ for the composition of elements in Burnside modules, and $g \circ h$ for the composition of maps between Burnside modules, while we write $1$ for the unit element in $A(S,S)\pLoc$ and $\id$ for the identity morphism of $A(S,S)\pLoc$.

If $R$ is the ring of elements in $A(S,S)\pLoc$ that are right stable with respect to a saturated fusion system $\F$, then, by the universal stable element theorem (Theorem \ref{thm:UnivStableII}), we have $R = A(S,S)\pLoc \cdot  \omega_\F$, and similarly $R' = A(S,S\times S)\pLoc \cdot \omega_\F$. The retractive transfer is then obtained by letting $t \colon A(S,S)\pLoc \to A(S,S)\pLoc \cdot \omega_F$ be the map $x \mapsto x \cdot \omega_\F$, and similarly for $t'$. We leave the reader to check that the required naturality conditions are satisfied, and that the Frobenius reciprocity relation for $\omega_{\F}$ implies $(\kappa4)$.

Conversely, if $(R,R')$ has a retractive transfer $(t,t')$, put $\omega = f \circ t(1) \in A(S,S)\pLoc$. Since $f$ and $t$ are both morphisms of left $A(S,S)\pLoc$-modules, we then have
 \[ f \circ t(x) = f \circ t(x \cdot 1) = x \cdot (f \circ t(1)) = x \cdot \omega \]
for all $x \in A(S,S)\pLoc$. In particular,
 \[ \omega \cdot \omega = f\circ t(\omega) = f\circ \underbrace{t \circ f}_{\id} \circ t(1) = f\circ t(1) = \omega, \]
so $\omega$ is idempotent. Similarly, for $h \in A(S,S\times S)$, write $h = h\cdot 1$, where $1$ is still the unit in 
$A(S,S)\pLoc$. Then the naturality conditions ($\kappa1$) and ($\kappa3$) imply that
\begin{equation} \label{eq:f't'} 
  f'(t'(h \cdot 1)) = h \cdot f(t(1)) = h \cdot \omega .
\end{equation}
We deduce that $R = A(S,S)\pLoc \cdot \omega$ and $R' = A(S,S\times S)\pLoc \cdot \omega$. Thus by the universal stable element theorem it suffices to
show that $\omega$ is the characteristic idempotent of a saturated fusion system. The assumption that $R$ is not contained in $J(S)\pLoc$ implies that $\omega$ is dominant, so by Theorem \ref{thm:DomFrobImpliesSat} it is enough to show that $\omega$ satisfies Frobenius reciprocity. 

We deduce Frobenius reciprocity from $(\kappa4)$ as follows. Composing with $f'$, ($\kappa4$) implies
 \[ f' \circ \kappa_R \circ (\id \times t) = f' \circ  t' \circ \kappa_A \circ (f\times \id), \]
as morphisms $R \times A(S,S)\pLoc \to A(S,S\times S)\pLoc$. Applying these morphisms to $(t(1),1)$ yields, 
on the left,
\begin{eqnarray*} 
   f' \circ \kappa_R \circ (\id \times t) (t(1),1) 
   &= &\kappa_A \circ (f \times f) (t(1),t(1)) \text{\qquad (by ($\kappa2$))}\\
   &= &\kappa_A (\omega,\omega) \\
   &= &(\omega \times \omega) \cdot [S,\Delta], \text{\qquad (by definition of $\kappa_A$)}
\end{eqnarray*}
and, on the right,
\begin{eqnarray*}
  f' \circ t' \circ \kappa_A \circ(f \times \id)(t(1),1) 
  &= & (f' \circ t') \big( \kappa_A (\omega,1) \big) \\
  &= & (f' \circ t') \big( (\omega \times 1) \cdot [S,\Delta] \big)  \text{\qquad (by definition of $\kappa_A$)}\\
  &= & (\omega \times 1) \cdot [S,\Delta] \cdot \omega  \text{\qquad (by \eqref{eq:f't'}).}
\end{eqnarray*}
Thus we have the Frobenius reciprocity relation
 \[ (\omega \times \omega) \cdot [S,\Delta] = (\omega \times 1) \cdot [S,\Delta] \cdot \omega, \]
completing the proof.
\end{proof}

\begin{rmk} As we noted earlier, a $\kappa$-preserving submodule amounts to a subfunctor $\rho$ of the functor $A(S,-)\pLoc$ defined on the full subcategory of $\Burnside\pLoc$ with objects $S$ and $S\times S$. If the $\kappa$-preserving submodule has a retractive transfer, then Theorem \ref{thm:AWAnalogue} tells us that $\rho(-) = A(S,-)\pLoc\cdot \omega$, where $\omega$ is the characteristic idempotent of the right stabilizer of $\rho(S)$. Thus $\rho$ can be extended to a $p$-local, $p$-defined Mackey functor by setting $\rho(P) \defeq A(S,P)\pLoc\cdot \omega$ for every finite $p$-group $P$. An alternative formulation of Theorem \ref{thm:AWAnalogue} is therefore to consider an inclusion of $p$-local, $p$-defined Mackey functors $f \colon \rho \hookrightarrow A(S,-)\pLoc$ that preserves the generalized morphism of $A(P,P)\pLoc \times A(Q,Q)\pLoc$-modules 
\[ \kappa_{P,Q} \colon A(S,P)\pLoc \times A(S,Q)\pLoc \xrightarrow{- \times -} A(S\times S, Q \times P) \xrightarrow{-\circ[S,\Delta]} A(S,P\times Q)\pLoc \]
for finite $p$-groups $P$ and $Q$. An argument analogous to the proof of Theorem \ref{thm:AWAnalogue} shows that there exists a saturated fusion system $\F$ on $S$ such that each map $f_P \colon \rho(P) \to A(S,P)\pLoc$ is the inclusion of right $\F$-stable elements, if and only if there exists a natural transformation $t \colon A(S,-)\pLoc \to R$ such that $t \circ f = \id$ and $t$ satisfies a Frobenius reciprocity relation generalizing ($\kappa4$). This formulation is in some ways more satisfying as it involves a more categorical framework encoding the naturality conditions ($\kappa1$) and ($\kappa3$). On the other hand, there is more work involved in producing a functor $\rho$ than just the modules $R = \rho(S)$ and $R' = \rho(S\times S)$, and the functor formulation seems further removed from the statement of Theorem \ref{thm:AW} (although this too can be phrased in terms of functors defined on elementary abelian, finite $p$-groups).
\end{rmk}

%%%%%%%%%%%%%%%%%%%%%%%%%%%%%%%%%%%%%%%%%%%%%%%%%%%%%%%%%%%%%%%%%
\section{Retractive transfer triples and $p$-local finite groups} \label{sec:RTT}
%%%%%%%%%%%%%%%%%%%%%%%%%%%%%%%%%%%%%%%%%%%%%%%%%%%%%%%%%%%%%%%%%

In this section we address a conjecture of Haynes Miller, proposing an alternative model for $p$-local finite groups, which provided a starting point for the work in this paper and motivated the investigation of the Frobenius reciprocity condition. 

\subsection{$p$-local finite groups}
A $p$-local group, as defined by Broto--Levi--Oliver in \cite{BLO2}, can be thought of as a saturated fusion system with a classifying space. Broto--Levi--Oliver explained what a classifying space means in this context, and gave a model for classifying spaces in terms of centric linking systems. We recall their definitions here.

\begin{defn}
Let $\F$ be a fusion system on a finite $p$-group $S$. A \emph{centric
linking system associated to} $\F$ is a category $\Link$ whose
objects are the $\F$-centric subgroups of $S$, together with a
functor
 \[ \pi: \Link \rightarrow \F^{c}, \]
and distinguished monomorphisms \mbox{$P \xrightarrow{\delta_P}
\Aut_{\Link}(P),$} for each $\F$-centric subgroup $P \leq S$, which
satisfy the following conditions.
\begin{enumerate}
  \item[(A)] The functor $\pi$ is the identity on objects and
surjective on morphisms. More precisely, for each pair of objects
$P,Q \in \Link,$ the centre $Z(P)$ acts freely on $\Mor_\Link(P,Q)$ by
composition (upon identifying $Z(P)$ with \mbox{$\delta_P(Z(P))
\leq \Aut_{\Link}(P)$}), and $\pi$ induces a bijection
  \[
\begin{CD}
{\Mor_\Link(P,Q)/Z(P)} @ > {\cong} >> {\Hom_\F(P,Q).} \\
\end{CD}
\]
  \item[(B)] For each $\F$-centric subgroup $P \leq S$ and each $g
\in P$, $\pi$ sends $\delta_{P}(g) \in \Aut_{\Link}(P)$ to $c_{g}
\in \Aut_\F(P)$.
  \item[(C)] For each $f \in \Mor_\Link(P,Q)$ and each $g \in P$, the
following square commutes in $\Link$:
 \[
\begin{CD}
{P} @ > {f} >> {Q} \\ @ VV {\delta_{P}(g)} V @ VV
{\delta_{Q}(\pi(f)(g))} V \\ {P} @> {f} >> {Q.} \\
\end{CD}
\]
\end{enumerate}
\end{defn}

\begin{defn}
A \emph{$p$-local finite group} is a triple $(S,\F,\Link)$ where $S$ is a finite $p$-group, $\F$ is a saturated fusion system on $S$, and $\Link$ is a centric linking system associated to $\F$. The \emph{classifying space} of $(S,\F,\Link)$ is the $p$-completed geometric realization $\ClSp$.
\end{defn}

The distinguished homomorphism $\delta_S$ induces a map $\theta \colon BS \to \ClSp$, which we think of as an inclusion map. 

The driving question in the subject of $p$-local finite groups is on the existence and uniqueness of centric linking systems (and hence classifying spaces) associated to saturated fusion systems. Broto--Levi--Oliver developed an obstruction theory to the existence and uniqueness questions, and using techniques developed by Grodal in \cite{Grodal}, they settled the question for fusion systems on small groups, proving existence when the group has $p$-rank less than $p^3$, and uniqueness when the $p$-rank is less than $p^2$. For a fusion system $\F_S(G)$ arising from a Sylow inclusion $S \leq G$, a centric linking system $\Link_S(G)$ can be constructed from $G$. This linking system satisfies $\pComp{|\Link_S(G)|} \simeq \pComp{BG}$ (\cite{BLO1}), and Oliver, in his proof of the Martino--Priddy conjecture \cite{bob1,bob2}, proved that $\Link_S(G)$ is the only centric linking system associated to $\F_S(G)$. These facts motivate the term \emph{classifying space} for $\ClSp$, with further justification coming from the homotopical properties of $\ClSp$. We recall only one of these properties, on recovering a $p$-local finite group from its classifying space here, and refer the interested reader to \cite{BLO2} for further information. 

Given a finite $p$-group $S$, a space $X$ and a map $f \colon BS \to X$, define a fusion system $\F_{S,f}(X)$ by setting
 \[ \Hom_{\F_{S,f}(X)}(P,Q) \defeq \{ \varphi \in \Inj(P,Q) \mid f\vert_{BP} \simeq f\vert_{BQ} \circ B\varphi \}   \]
for groups $P,Q\leq S$. Here $f\vert_{BP}$ is the composite $BP \xrightarrow{B\incl} BS \xrightarrow{f} X$, and $\simeq$ means non-basepoint-preserving homotopy. In general one should not expect $\F_{S,f}(X)$ to be a saturated fusion system, although this is true when $X$ is the classifying space of a $p$-local finite group (see Theorem \ref{thm:F(X)andL(X)} below). One can also define a category $\Link_{S,f}(X)$ whose objects are the $\F_{S,f}(X)$-centric subgroups of $S$, with morphism sets
 \[ \Mor_{\Link_{S,f}(X)}(P,Q) = \{ (\varphi,[H])  \}, \]
where $[H]$ is a homotopy class of homotopies between $f\vert_{BP}$ and $f\vert_{BQ} \circ B\varphi$. As the following result shows, a $p$-local finite group is determined by its classifying space. We refer the reader to \cite{BLO2} for the precise meaning of the isomorphism of linking systems in the statement.

\begin{thm}[\cite{BLO2}] \label{thm:F(X)andL(X)}
For a $p$-local finite group $(S,\F,\Link)$ we have 
\[ \F_{S,\theta}(\ClSp) = \F \quad \text{ and } \quad \Link_{S,\theta}(\ClSp) \cong \Link . \]
\end{thm}

A stable version of this result was obtained in \cite{KR:ClSpectra}. For a spectrum $E$ and a stable map $f \colon \Stable{BS_+} \to E$, define a fusion system $\F_{S,f}(E)$ by 
 \[ \Hom_{\F_{S,f}(E)}(P,Q) \defeq \{ \varphi \in \Inj(P,Q) \mid f\vert_{\Stable{BP_+}} \simeq f\vert_{\Stable{BQ_+}} \circ \Stable{B\varphi} \} . \]
Again, one should not expect $\F_{S,f}(E)$ to be a saturated fusion system, although this is true when $E$ is the classifying spectrum of a saturated fusion system on $S$ and $f$ is the structure map.

\begin{thm}[\cite{KR:ClSpectra}] \label{thm:F(ClSpec)}
For a saturated fusion system $\F$ on a finite $p$-group $S$ we have
  \[ \F_{S,\sigma_\F}(\ClSpectrum{\F}) = \F.   \]
\end{thm} 
Classifying spectra are compatible with classifying spaces in that $\Stable{\theta} \colon \Stable{BS_+} \to \Stable{{\ClSp}_+}$ is homotopy equivalent to $\sigma_\F \colon \Stable{BS_+} \to \ClSpectrum{\F}$ for a $p$-local finite group $(S,\F,\Link)$ (although the existence of the classifying spectrum is independent of the existence of a centric linking system). Therefore Theorem \ref{thm:F(ClSpec)} says that the equation of fusion systems in Theorem \ref{thm:F(X)andL(X)} holds after suspension. However, there is a subtle difference between the two theorems: Theorem \ref{thm:F(X)andL(X)} can be strengthened to show that $\F$ depends on the homotopy type of $\ClSp$ (and the homotopy type of $\theta$ is determined by the homotopy type of $\ClSp$), while in Theorem \ref{thm:F(ClSpec)} $\F$ really depends on the homotopy type of $\ClSpectrum{\F}$ \emph{and} the structure map $\sigma_\F$. This point is illustrated by \cite[Example 5.2]{MP3}, where Martino--Priddy construct two groups with different fusion systems whose $p$-completed classifying spaces have the same stable homotopy type. Consequently we refer to the pair $(\sigma_\F,\ClSpectrum{\F})$ as the \emph{structured classifying spectrum} of $\F$.

\subsection{Retractive transfer triples}
The definition of $p$-local finite groups includes elements of group theory and category theory, and it would be highly desirable to have a purely homotopy-theoretic model for the $p$-local homotopy theory of classifying spaces of finite groups. Such a model was suggested by Haynes Miller, defined as follows.
\begin{defn}
A \emph{retractive transfer triple} on a finite $p$-group $S$ is a 
triple $(f,t,X)$ where
\begin{itemize}
\item
 $X$ is a connected, $p$-complete, nilpotent space with finite 
fundamental group; 
\item $f \colon BS \to X$ is a homotopy monomorphism; and
\item $t \colon \PtdStable{X} \to \PtdStable{BS}$ is a retractive transfer of $f$.
\end{itemize}
\end{defn}

Here \emph{homotopy monomorphism} means that the map induced by $f$ in $\Fp$-cohomology makes $H^*(BS;\Fp)$ a finitely generated $H^*(X;\Fp)$-algebra. The term \emph{retractive transfer} means that $\PtdStable{f} \circ t \simeq 1_{\PtdStable{X}}$, and we have the Frobenius reciprocity relation
\[ (1_{\PtdStable{X}}\wedge t) \circ \Delta_X \simeq (\PtdStable{f} \wedge 1_{\PtdStable{BS}}) \circ \Delta_{BS} \circ t,\]
where $\Delta_{BS}$ and $\Delta_X$ denote the diagonals of $\PtdStable{BS}$ and $\PtdStable{X}$, respectively. In this case $t \circ \PtdStable{f}$ is idempotent up to homotopy, and hence corresponds to an idempotent $\omega$ in $\pComp{A(S,S)}$, and it is not hard to show that $\omega$ satisfies Frobenius reciprocity (the argument given in \ref{subsec:Translate} can easily be adapted). Note that $X$ is then the stable summand of $BS$ corresponding to the idempotent $\omega$. We say that the retractive transfer triple is \emph{dominant} if $X$ contains a dominant summand of $BS$, which is equivalent to $\omega \notin \pComp{J(S)}.$

\subsection{The connection}
Miller conjectured that (perhaps in the presence of additional structure) retractive transfer triples are equivalent to $p$-local finite groups. A partial confirmation was obtained in the first author's thesis and published in \cite{KR:Transfers&plfgs}. More precisely, it was shown that if $(f,t,X)$ is a retractive transfer triple on an elementary abelian $p$-group $V$, then $(V,\F_{V,f}(X),\Link_{V,f}(X))$ is a $p$-local finite group with classifying space homotopy equivalent to $X$. The converse direction was treated more generally, showing that if $(S,\F,\Link)$ is a $p$-local finite group on \emph{any} finite $p$-group $S$, then there exists a unique (up to homotopy) retractive transfer $t$ for $f$, and $(\theta,t,\ClSp)$ is a retractive transfer triple on $S$. 

The results in the current paper allow us to make further progress toward proving Miller's conjecture. One major obstacle to showing that a retractive transfer triple gives rise to a $p$-local finite group is associating a saturated fusion system to it. This was overcome in \cite{KR:Transfers&plfgs}, when the Sylow $p$-subgroup is an elementary abelian group $V$, by using a variant of the Adams--Wilkerson theorem \cite{AW} (Theorem \ref{thm:AW}) to show that $X$ has the homology type of the classifying space of a semi-direct product $W \ltimes V$, and then using Miller's theorem \cite{Mil} to deduce the necessary homotopical information from the homological information. As discussed in Section \ref{sec:AW}, the Adams--Wilkerson theorem can be replaced by Corollary \ref{thm:DomFrobImpliesSat} to obtain the following result.
 
\begin{prop}\label{prop:F(RTT)}
If $(f,t,X)$ is a dominant retractive transfer triple on a finite $p$-group $S$, then $\F_{S,f}(\Stable{X_+})$ is a saturated fusion system on $S$.
\end{prop}
\begin{proof}
Let $\omega$ be the idempotent in $\pComp{A(S,S)}$ corresponding to the homotopy idempotent $t \circ f$ of $\Stable{BS_+}$. Then $\omega$ is a dominant idempotent that satisfies Frobenius reciprocity, so $\FusRStab(\omega)$ is a saturated fusion system on $S$, and  $\F_{S,f}(\Stable{X_+}) = \FusRStab(\omega)$.
\end{proof}
Proposition \ref{prop:F(RTT)} shows that $X$ has the stable homotopy of the classifying spectrum of a saturated fusion. More precisely, it shows that $(f,\Stable{X_+})$ has the homotopy type of the structured classifying spectrum of a saturated fusion system. To show that $X$ has the homotopy type of the classifying space of a $p$-local finite group, it now remains to show that the necessary unstable homotopy information can be extracted from the stable homotopy information. This matter will be taken up in a separate article, and in the current paper we preview the anticipated results by presenting a statement with rather strong technical conditions that can be proved using existing techniques from \cite{BLO4} and Wojtkowiak obstruction theory \cite{Wo}. 

\begin{thm} \label{thm:RTTgivesPLFG}
Let $(f,t,X)$ be a dominant retractive transfer triple on a finite $p$-group $S$ and assume that
\begin{enumerate}
  \item for every $P \leq S$, the map $\Stable \colon [BP,X] \to \{BP_+,X_+\}$ is injective; and
  \item for every $\F_{S,f}(X)$-centric subgroup $P \leq S$, the induced map of mapping space components,
          \[ \Map(BP,BS)_{B\incl} \xrightarrow{f\circ -} \Map(BP,X)_{f\vert_{BP}}, \]
        is a homotopy equivalence.        
\end{enumerate}
Then $(S,\F_{S,f}(X),\Link_{S,f}(X))$ is a $p$-local finite group with classifying space homotopy equivalent to $X$. 
\end{thm}
\begin{proof}
By Proposition \ref{prop:F(RTT)}, $\F_{S,\Stable{f}}(\Stable{X_+})$ is a saturated fusion system on $S$. In general we obviously have an inclusion $\F_{S,f}(X) \subseteq \F_{S,\Stable{f}}(\Stable{X_+})$, and condition (1) implies that this is an equality in this case. Condition (2) and \cite[Lemma 1.8]{BLO4} now imply that $\Link_{S,f}(X)$ is a centric linking system associated to $\F_{S,f}(X)$. Hence $(S,\F_{S,f}(X),\Link_{S,f}(X))$ is a $p$-local finite group. We will write $\F$ for $\F_{S,f}(X)$ and $\Link$ for $\Link_{S,f}(X)$ for the remainder of the proof. 

It remains to show that $\ClSp \simeq X$, and to do this it is enough to construct a map $h \colon |\Link| \to X$ such that $h \circ \theta \simeq f$. The reason is that in cohomology $f$ and $\theta$ both induce injections with image the $\F$-stable elements in $H^*(BS;\Fp)$, and thus $h$ induces an isomorphism in cohomology with $\Fp$-coefficients, whence the $p$-completion $h \colon \ClSp \to \pComp{X}\simeq X$ is a homotopy equivalence. We outline the construction of $h$ here, leaving a more detailed discussion for a better time.

Recall that $\displaystyle |\Link| \cong \operatorname{HoColim}\limits_\Link (*)$, where $* \colon \Link \to \mathcal{TOP}$ is the functor to topological spaces that sends every object to a point. Let $\Oo^c = \Oo(\F^c)$ be the \emph{centric orbit system} of $\F$; that is, the category with objects the $\F$-centric subgroups of $S$ and morphism sets 
 \[\Mor_{\Oo^c}(P,Q) \defeq Q\backslash\Hom_\F(P,Q) \,. \] 
Let $\tilde{\pi} \colon \Link \twoheadrightarrow \Oo^c$ be the canonical projection functor, and let $\B \colon \Oo^c \to \mathcal{TOP}$ be the left homotopy Kan extension of $*$ along $\tilde{\pi}$. By standard results we have a natural equivalence
 \[ \underset{\Oo^c}{\rm{HoColim}} (\B) \simeq \underset{\Link}{\rm{HoColim}} (*) = |\Link|, \]
so it suffices to construct a map $\tilde{h} \colon \rm{HoColim}_{\Oo^c} (\B) \to X$ compatible with the inclusion maps.

The homotopy colimit $Y \defeq  \rm{HoColim}_{\Oo^c} (\B)$ is constructed  as a disjoint union with one copy of $\B(P_0) \times \Delta^n$ for each sequence of composable morphisms 
 \[ P_0 \xrightarrow{[\varphi_1]} P_1 \xrightarrow{[\varphi_2]} P_2  \xrightarrow{[\varphi_3]} \cdots  \xrightarrow{[\varphi_n]} P_n, \]
subject to the usual identifications. We filter $Y$ by letting $Y_n \subset Y$ be the subspace obtained from sequences of length at most $n$. Following Wojtkowiak \cite{Wo}, we construct $\tilde{h}$ inductively by constructing maps $\tilde{h}_n \colon Y_n \to X$ for each $n \geq 0$, so that for $n \geq 1$, $\tilde{h}_{n+1}$ extends $\tilde{h}_n$ up to homotopy relative to $Y_{n-1}$. The map $\tilde{h}$ is the direct limit of the maps $\tilde{h}_n$.

By \cite[Proposition 2.2]{BLO2}, $\B$ is a homotopy lifting of the homotopy functor $P \mapsto BP$ on $\Oo^c$. This means that for each object $P$ in $\Oo^c$ we have a homotopy equivalence $BP \xrightarrow{u_P} \B(P)$, and for morphisms $P \xrightarrow{[\varphi]} Q$ in $\Oo^c$, $B\varphi \circ u_P$ is homotopic to  $u_Q \circ \B(\varphi)$. Indeed, a fixed homotopy $U_{[\varphi]}$ is given in \cite{BLO2}.
Setting  
 \[ \tilde{h}_P \defeq f\vert_{BP} \circ u_P \colon \B(P) \to X \] 
for every $\F$-centric $P$, the collection $(\tilde{h}_P)_{P \in \Oo^c}$ induces a map $\tilde{h}_0 \colon Y_0 \to X.$ The inclusion $\theta \colon BS \to |\Link|$ factors through the inclusion $\B(S) \hookrightarrow |\Link|$ up to homotopy, so $\tilde{h}_0 \circ \theta \simeq f$ is compatible with inclusions. Extending $\tilde{h}_0$ to $Y_1$ amounts to choosing a homotopy $h_{[\varphi]}$ between $\tilde{h}_Q \circ \B([\varphi])$ and $\tilde{h}_P$ for every $[\varphi] \in \Mor_{\Oo^c}(P,Q)$. This is achieved by choosing a homotopy between $f\vert_{BQ} \circ B\varphi$ and $f\vert_{BP}$, and concatenating this with $f\vert_{BQ} \circ U_{[\varphi]}$, keeping the following diagram in mind:
\[ 
\xymatrix{
   \B(P) \ar[r]^{u_P} \ar[d]_{\B([\varphi])} & BP \ar[r]^{f\vert_{BP}}\ar[d]_{B\varphi} & X \\
   \B(Q) \ar[r]^{u_Q} & BQ \ar[ru]_{f\vert_{BQ}} &
} 
\]

Wojtkowiak showed in \cite{Wo} that for $n \geq 2$, the obstruction to extending $\tilde{h}_{n-1}$ to $Y_{n}$ up to homotopy relative to $Y_{n-2}$ lies in ${\underleftarrow{\textup{lim}}}_{\Oo^c}^n \pi_{n-1}(\Map(\B(-),X)_{h_{(-)}})$. By condition (2) and $\F$-centricity we have 
  \[ \Map(\B(P),X)_{h_{P}} \simeq \Map(BP,X)_{f\vert_{BP}} \simeq \Map(BP,BS)_{B\incl} \simeq BC_S(P) \simeq BZ(P), \] 
so 
 \[ \pi_{n-1}( \Map(\B(P),X)_{h_{P}}) \cong 
    \begin{cases} Z(P) \, , &\text{if $n=2$}; \\
                  0 \, ,    &\text{if $n>2$}.
    \end{cases} 
\]
Thus the only possible obstruction occurs at $n = 2$, and lies in ${\underleftarrow{\textup{lim}}}_{\Oo^c}^2 Z(P)$. Now, by Broto--Levi--Oliver \cite{BLO2}, ${\underleftarrow{\textup{lim}}}_{\Oo^c}^2 Z(P)$ acts freely and transitively on the set of isomorphism classes of centric linking systems for $\F$, and this action is compatible with the Wojtkowiak obstruction theory. That is, if we let $\alpha \in {\underleftarrow{\textup{lim}}}_{\Oo^c}^2 Z(P) $ be the obstruction to extending $\tilde{h}_1 \colon Y_1 \to X$ to $\tilde{h}_2 \colon Y_2 \to X$, then there exists a centric linking system $\Link_\alpha$ for which the corresponding obstruction vanishes, and we can construct a map $h_\alpha \colon |\Link_\alpha| \to X$ that is compatible with the inclusion maps. As before, a cohomological argument shows that $h_\alpha$ induces an equivalence upon $p$-completion. However, this implies that 
 \[ \Link = \Link^c_{S,f}(X) \cong \Link^c_{S,\theta}(\pComp{|\Link_\alpha|}) \cong \Link_\alpha \, , \]  
and we deduce that $\alpha = 0$, so the obstructions to constructing the homotopy equivalence $h$ vanish, proving the claim.
\end{proof}

The conditions in Theorem \ref{thm:RTTgivesPLFG} may appear very stringent, but they are actually always satisfied when $X$ is the classifying space of a $p$-local finite group: condition (1) follows from the computation of $[BP,\ClSp]$ in \cite{BLO2} and the computation of $\{BP_+,{\ClSp}_+\}$ in \cite{KR:ClSpectra}, and condition (2) is proved in \cite{BLO2}. Evidence suggests that these conditions are also automatically satisfied by a retractive transfer triple. Proving this would result in a full, affirmative resolution of Miller's conjecture.

\end{document}